%% file: manuscript.tex
\documentclass{siamart251104}
\usepackage{fullpage}
\usepackage{dkmath}
\usepackage{amsmath,amssymb,amsfonts,commath}
\usepackage{mathtools}
\usepackage{silence}
\usepackage{subcaption}
\usepackage{graphicx}
\usepackage{graphbox}
\usepackage{tikz}
\usetikzlibrary{math}
\usepackage{enumitem}
\input{defs.tex}

\newsiamremark{remark}{Remark}
\newsiamthm{thm}{Theorem}
\newsiamthm{newdef}{Definition}

\crefname{appendix}{Appendix}{Appendices}
\Crefname{appendix}{Appendix}{Appendices}
\crefname{theorem}{Remark}{Remarks}
\Crefname{theorem}{Remark}{Remarks}
\crefname{section}{Section}{Sections}
\Crefname{section}{Section}{Sections}
\crefname{subsection}{Section}{Sections}
\Crefname{subsection}{Section}{Sections}
\crefname{table}{Table}{Tables}
\Crefname{table}{Table}{Tables}

\def\kth{Department of Mathematics, KTH Royal Institute of Technology, Stockholm, Sweden}

\title{Fast summation on rectangular cuboids with arbitrary periodicity in the DMK framework}

\author{
David Krantz%
    \thanks{\kth\,
    ({\tt davkra@kth.se}).}
    \and
Ludvig af Klinteberg%
    \thanks{Department of Business and Mathematics, Mälardalen University (MDU), Västerås, Sweden\,
    ~({\tt ludvig.af.klinteberg@mdu.se}).}
    \and
Anna-Karin Tornberg%
    \thanks{\kth\,
   ({\tt akto@kth.se}).}
}

\date{\today}

\begin{document}

\maketitle

\begin{abstract}\label{s:abstract}
Dual-space multilevel kernel-splitting (DMK) is a fast summation
framework that combines ideas from the fast multipole method, Ewald
summation, and multilevel summation. Originally formulated for
free-space problems, and later extended to fully periodic problems on
a cube, it decomposes the kernel interaction into a smooth global
contribution and a hierarchy of localized interactions evaluated on an
octree.
We extend DMK to problems on rectangular cuboids with periodic
boundary conditions in one, two, or three coordinate directions. The
periodization leverages the fact that interactions on all tree levels
below the root are localized, 
allowing for their evaluation
with minimal modification on a cubical tiling
of the domain. The remaining smooth root-level far-field contribution
is evaluated in Fourier space, with Fourier series in the periodic
directions and Fourier integrals in the free directions. For reduced
periodicity, truncated kernels are used to regularize singular and
near-singular Fourier kernels, yielding rapidly convergent trapezoidal
discretizations and a unified treatment of all periodicities. For
large-aspect-ratio cuboids, the root-level sum can be accelerated
using the fast Fourier transform.
We validate the method for the electrostatic potential and Stokeslet,
stresslet and rotlet potentials, for all periodicities and a wide
range of aspect ratios. Numerical experiments show that the
periodization adds only a small overhead to the original free-space DMK 
algorithm, also for high-aspect-ratio cuboids.
The resulting method provides a framework for applying DMK to problems
with mixed periodicity on rectangular cuboids, and extends naturally
to other non-oscillatory kernels for which a kernel split is
available.
\end{abstract}


\section{Introduction}\label{s:introduction}
Large-scale computations of periodic potentials arise frequently in many
areas of science and engineering, including electrostatic potentials
in molecular dynamics (MD) simulations and discretized layer
potentials in Stokes simulations of microfluidic flow. We here consider such potentials, written
in the form
\begin{equation}
  u^{D\per}(\vx_\tind) = \sum_{\sind=1}^{N_s} \sum_{\v p\in\perindsetD}^* K(\vx_\tind-\vx_\sind + \v p) 
  \rho_\sind, \quad \tind=1,\dots,N_t,
  \label{eq:periodic_sum}
\end{equation}
where the $*$ superscript indicates that the term $\v p=0$ is ignored
if $\vx_\tind=\vx_\sind$. 
The kernel $K$ maps from scalar-, vector- or tensor-valued source strengths $\rho_\sind$ to a scalar- or vector-valued potential $u^{D\per}$, and is typically singular at the origin and decays slowly at infinity. The $N_s$ sources
 are located at points $\vx_\sind \in \reals^3$ inside a rectangular cuboid
$\Omega = [0,L_1) \times [0,L_2) \times [0,L_3)$, called the primary
cell, and the $N_t$ targets are in the same cuboid. The sources
and targets commonly coincide, and we will for simplicity of exposition assume
$N_s=N_t=N$.  
The set $P_{D\per}$ of periodic images depends on the number of periodic directions $D$. Without loss of generality, we assume that the coordinate axes are relabeled such that the first $D$ directions are periodic,
\begin{equation}
  \label{eq:periodic-images}
  \perindsetD = \begin{cases}
    \{ (\bar{p}_1 L_1, \bar{p}_2 L_2, \bar{p}_3 L_3) : \bar{p}_i
    \in \mathbb{Z} \}, & \text{if $D=3$ (triply periodic)}, \\
    \{ (\bar{p}_1 L_1, \bar{p}_2 L_2, 0) : \bar{p}_i
    \in \mathbb{Z} \}, & \text{if $D=2$ (doubly periodic)}, \\
    \{ (\bar{p}_1 L_1, 0, 0) : \bar{p}_i
    \in \mathbb{Z} \}, & \text{if $D=1$ (singly periodic)}.
  \end{cases}
\end{equation}
The free-space case $D=0$ corresponds to $\perindset{0}=\{\v 0\}$.
The computational complexity of evaluating \eqref{eq:periodic_sum} is
clearly quadratic, $O(N^2)$. Moreover, the lattice sum over periodic
images converges slowly, due to the slow decay of $K$, and is in many cases only conditionally convergent \cite{Smith2008,Pozrikidis1996}.
The cases with one or two periodic directions introduce additional
difficulties, since periodic interactions must be combined with
free-space behavior in the non-periodic directions.

This work is concerned with the efficient evaluation of
\eqref{eq:periodic_sum} on rectangular cuboids with one, two, or three periodic directions.  The algorithms presented will make use of a
decomposition, or \emph{kernel-split} of the interaction kernel $K$,
with a splitting parameter $\decpar$,
\begin{equation}
K(\v x)=\Kmoll(\v x,\decpar)+\Kres(\v x,\decpar).
\label{eq:singlelevel_ks}
\end{equation}
The split separates the singular short-range behavior from a smooth long-range part whose Fourier representation decays rapidly enough to be truncated. 
Here, the {\em
  mollified} kernel $\Kmoll(\v x,\decpar)$ is smooth at the origin and
matches the long-range behavior of $K$. Meanwhile, the
\emph{residual} kernel $\Kres(\v x,\decpar)$ matches $K$ at the origin
and then decays rapidly. For the harmonic kernel $H(|\vx|)$ of
periodic electrostatic interactions, the classic kernel-split introduced
by Ewald \cite{Ewald1921} in 1921 is
\begin{align}
    H(r) = \frac{1}{4 \pi r} = \frac{\erf(r/\decpar)}{4 \pi r} + 
\frac{\erfc(r/\decpar)}{4 \pi r}
    := H_M(r,\decpar) + H_R(r,\decpar),
    \label{eq:ewald-split}
\end{align}
with $\erf$ the error function and $\erfc$ its complement,
\begin{equation} \erf(x) = \frac{2}{\sqrt{\pi}} \int_0^x e^{-t^2} \dif t , \qquad 
\erfc(x) = 1 - \erf(x) .
\end{equation}
Due to the smoothness of the mollified kernel $H_M(|\v x|,\decpar)$, its Fourier transform $\widehat{H}_M(|\v k|,\decpar)$ is rapidly decaying,
\begin{equation}
\widehat{H}_M(k,\decpar) = \frac{e^{-k^2\decpar^2/4}}{k^2},
\label{eq:Hhat_k}
\end{equation}
and can be truncated in Fourier space. The residual kernel
$H_R(|\v x|,\decpar)$ acts as a local correction to $H_M$; it matches
$H$ at the origin, but decays rapidly in physical space and can be
truncated beyond a cutoff radius.  Both decay rates are controlled by
the length scale $\decpar$, which for a given error tolerance $\eeps$
is proportional to the cutoff radius.

The split \eqref{eq:singlelevel_ks} is non-unique and can be done in
different ways. Both the Ewald split \eqref{eq:ewald-split} and the
classical Stokeslet split by Hasimoto~\cite{Hasimoto1959} are based on
the Gaussian, but other choices are possible. In~\cite{jiang2025cpam},
a kernel split based on the zeroth-order prolate spheroidal wave
function (PSWF) was derived for the harmonic kernel, and new
PSWF-based splits for the Stokes kernels were later derived in
\cite{StokesDMK2026}. These splits improve efficiency by reducing the Fourier bandwidth needed for a given
real-space support. We use them in the present work, but
the periodization strategy developed below is independent of
this choice.

Reverting to the generic notation, inserting the split 
 \eqref{eq:singlelevel_ks} into the sum \eqref{eq:periodic_sum} for
the triply periodic case ($D=3$), the far-field part can be evaluated in Fourier space using the Poisson summation formula\footnote{The term $\vk=0$ being omitted from the Fourier sum corresponds to electrostatic charge neutrality or Stokes zero-mean flow, and is required for convergence of the periodic sum. In the case of the stresslet, an additional compensation term $\ufar^{3\per,0}(\vx_\tind)$ is required, see e.g.~\cite[Appendix A.1]{bagge_fast_2023} or \cite[Section 3.2]{af_klinteberg_fast_2016}.} 
 \begin{align}
\ufarP{3}(\vx_\tind,\decpar) &= \sum_{\sind=1}^N \sum_{\v p\in\perindset{3}} \Kmoll(\vx_\tind-\vx_\sind + \v p,\decpar) \rho_\sind
= \sum_{\sind=1}^N \frac{1}{\lvert \Omega \rvert} \sum_{\substack{\vk\in\wavenumset^3\\ \vk \ne 0}} \Kmollhat(\vk,\decpar) e^{i\vk\cdot(\vx_\tind-\vx_\sind)}
  \rho_\sind,
\label{eq:ewald_far_3P}
\end{align}
where the domain volume is $\lvert \Omega \rvert = L_1 L_2 L_3$ and
the set of discrete wavenumbers is given by 
\begin{equation}
  \wavenumset^3
  := \left\{ 2\pi
  \left(\frac{\kintx}{L_1}, \frac{\kinty}{L_2}, \frac{\kintz}{L_3}\right)
  : \bar{k}_i \in \mathbb{Z} \right\}
  .
    \label{eq:wavenumbers-3p}
\end{equation}
In the case of partial periodicity, the summation over discrete wave numbers will in the free (non periodic) directions be replaced by integration over continuous wave numbers. As the Fourier kernels are singular at $|\vk|=0$ (see e.g.~\eqref{eq:Hhat_k}), careful numerical treatment is required.  This is true also for the smallest periodic wave numbers, for which integrands can become sharply peaked.

The above process, known as Ewald summation, remedies the slow
convergence of the lattice sum, but is in its basic form still too expensive for large systems. 
It can be accelerated by using particle-mesh Ewald methods to
evaluate the far-field part associated with the mollified kernel
$\Kmoll$ in \eqref{eq:ewald_far_3P}. 
Sources are then spread onto a uniform grid and the computed potential interpolated from the grid to the target points, with intermediate steps accelerated by applying the fast Fourier transform (FFT) on the uniform grid.
With appropriate scaling, this has a computational complexity of $O(N\log N)$, while short-range interactions are evaluated using neighbor lists at $O(N)$ complexity.
However, for
highly clustered configurations
the total cost deteriorates substantially, and may in
the worst case approach $O(N^2)$.

The classical use of Ewald summation has been in electrostatics (i.e.,
the harmonic kernel) with full periodicity \cite{Deserno1998,Essmann1995,Hockney1981,Lindbo2011c}, but the approach has also been generalized to Stokes kernels
and reduced periodicity (including the free space
case)~\cite{AfKlinteberg2016fse,AfKlinteberg2014a,Lindbo2010,Nestler2015,Saintillan2005}.
In settings that are not fully periodic, the use of so-called
truncated (or modified) kernels, based on an idea by \cite{Vico2016},
is critical for removing the singularity of the Fourier kernels.
First introduced for Ewald summation in the free-space
Stokes case in~\cite{AfKlinteberg2016fse}, the approach has later been
extended to periodic boundary conditions in any number (three, two,
one, or none) of the spatial directions, in a unified framework for
electrostatics and the kernels of Stokes flow
\cite{shamshirgar_fast_2021,bagge_fast_2023}.

While Ewald methods are highly efficient, their reliance on the FFT,
which operates on a uniform grid, makes the computational work difficult to adapt to clustered point distributions.  Tree-based methods, by
contrast, are built on adaptive data structures and can
therefore efficiently handle highly nonuniform point distributions. They can typically achieve $O(N)$ complexity, excluding
the tree construction cost, which usually is negligible.  Over the
past few decades, a range of such methods has been developed,
including the fast multipole method
(FMM)~\cite{greengard1988,Greengard1987,Tornberg2008,Ying2004},
$\mathcal H^2$-matrix
methods~\cite{borm2010,giebermann2001comp,hackbusch2002anm}, and
multilevel summation
approaches~\cite{brandt1990jcp,brandt1998,multilevel_summation_2015,tensor_multilevel_ewald}.
The FMM and $\mathcal H^2$-matrix methods reduce computational cost by
approximating far-field interactions via hierarchical low-rank
representations of the original kernel. Multilevel summation takes a
different route. Rather than partitioning interactions into near and
far components, it constructs a telescoping expansion in which the
kernel is replaced by a sequence of smooth, increasingly localized
approximations that are efficiently handled at their respective
scales.

Ewald-based methods are especially efficient for triply periodic problems,
while free directions introduce additional costs. Tree-based methods, on
the other hand, are most naturally formulated for free-space problems.
Periodic boundary conditions can be imposed in FMM-based methods by
explicitly summing over a finite set of neighboring image boxes and
representing the remaining images through a far-field
correction. Classical periodic FMMs achieve this by evaluating lattice
sums of multipole expansions, often using Ewald-type
decompositions~\cite{Amisaki2000,Yoshii2018}. A second class of
methods constructs the far-field correction using a set of proxy
sources outside the primary cell.  In the approach of
\cite{Gumerov2014}, the far field of a periodic point sum is
represented by proxy sources whose coefficients are determined from a
least-squares problem, while \cite{Barnett2018} use proxy sources in
an extended 2D boundary-integral formulation to impose periodicity
constraints. In \cite{Yan2018a}, the equivalent-source machinery of the
kernel-independent FMM (KIFMM) is used to construct a periodic
multipole-to-local operator from a periodic Green's function,
evaluated by Ewald summation. Recently, \cite{Li2026} avoid any
periodization-specific linear solve by introducing a hierarchical
recursion over image boxes using KIFMM equivalent surfaces. In a
distinct approach in 2D, \cite{Pei2023} compute the far-field
correction using plane-wave expansions of the fundamental solutions,
accelerated through the non-uniform fast Fourier transform (NUFFT)
\cite{Barnett2019finufft}.  Many of the FMM periodization methods 
currently available in 3D are most naturally formulated for cubic or nearly
cubic cells, and may lose efficiency in other cases. For high-aspect-ratio rectangular cells, lattice-sum methods typically require many image boxes or Fourier modes, while proxy- or equivalent-source methods require a larger number of auxiliary points, which increases cost and can affect conditioning. 
Thus, while tree methods provide adaptivity, efficient periodization on rectangular cuboids remains nontrivial, especially for large aspect ratios. This motivates adaptive periodic summation methods that can handle such geometries without sacrificing efficiency.

The recently developed dual-space multilevel kernel-splitting (DMK)
framework~\cite{jiang2025cpam,StokesDMK2026} brings together elements
of the methods discussed. It is based on a multilevel
decomposition into regularized kernels, leverages Fourier-space
diagonalization of the interactions at each level, and transfers local expansions through the tree hierarchy. 
In reported benchmarks, DMK has been shown to be as fast or faster than state-of-the-art FMM implementations~\cite{jiang2025cpam,StokesDMK2026}.
Like the FMM, DMK is adaptive, but it is more
naturally amenable to periodization. For cubic domains, an extension
of the free-space formulation to full periodicity is straightforward
and computationally efficient, and has already been demonstrated
\cite{StokesDMK2026}.

\subsection{Contributions and outline}\label{ss:contributions_outline}
In this paper, we extend the periodic DMK framework from cubic, fully periodic domains to rectangular cuboids with periodicity in one, two, or three directions. 
The main observation underlying the extension is that in DMK, kernel interactions are localized at all levels of the tree below the root. Periodicity therefore affects these levels only through periodic wrapping of near-neighbor interactions, while the periodic contribution from the lattice of image cells is isolated in the smooth root-level far-field.

The main components of our contribution are as follows.
\begin{itemize}
\item A cubical tiling of the primary cell into cubes whose side length is the shortest periodic length $L_\text{min}$, with one octree associated with each cube. When all periodic side lengths are integer multiples of this length, the tiling is exact. Otherwise, the primary cell is embedded in a bounded extension of width $O(L_\text{min})$. The localized part of DMK is evaluated on this tiled domain.
\item A Fourier-space treatment of the smooth root-level far-field contribution for arbitrary periodicity. For reduced periodicity, truncated kernels are used to regularize the singular and near-singular Fourier kernels that arise in the non-periodic directions.
\item An FFT-accelerated evaluation of the root-level far-field contribution for large-aspect-ratio cuboids, reducing the geometry-dependent cost of the far-field contribution.
\end{itemize}

The resulting method preserves the adaptive structure of DMK, and 
 the per-source cost stays nearly independent of aspect ratio in experiments at constant point density. 
We demonstrate the periodization method for electrostatic potentials and for the Stokeslet, stresslet, and rotlet potentials of Stokes flow, using PSWF-based splits of the harmonic and biharmonic kernels. The periodization strategy itself is not tied to this particular choice of split, and should extend directly to other non-oscillatory kernels for which a suitable kernel split is available.

The remainder of this paper is organized as follows.
\Cref{sec:multilevel} introduces the multilevel kernel-splitting for arbitrary periodicity.
\Cref{sec:DMKalg} reviews the original DMK method in free space and on a periodic cube.
\Cref{sec:3P} extends the method to arbitrary cuboids in the triply periodic case and introduces FFT-based acceleration for large aspect ratios. 
\Cref{sec:reduced_per} treats singly and doubly periodic cuboids, using truncated kernels for the Fourier integrals in the free directions. 
\Cref{s:numerical_results} presents numerical experiments on 
accuracy, parameter selection, periodization overhead, and aspect-ratio dependence.
\Cref{s:conclusions} concludes the paper.
Technical details on Fourier conventions, Stokes kernels, truncated kernels, quadrature estimates, and additional Stokes-kernel experiments are collected in \cref{sec:fourier-defs,sec:stokes-split,sec:trunc_biharmonic,sec:trunc_Stokes,sec:error_estimates,sec:stokes_results}.

\section{Multilevel splitting and summation in arbitrary periodicity}\label{sec:multilevel}
This section introduces the multilevel kernel splitting used in DMK and formulates the resulting decomposition for singly, doubly, and triply periodic problems. We first describe the telescoping split, then review the PSWF-based kernels used in the numerical examples,
and finally state the compact and far-field contributions that are evaluated by the periodized DMK algorithm.

\subsection{Multilevel kernel-splitting}

The single-level split introduced in \eqref{eq:singlelevel_ks}
is extended in DMK to the multilevel telescoping expansion
\begin{equation}
 K(\vx)= \Kmoll(\vx,\decpar_0)
 +\sum_{\ell=0}^{\noLevels-1} \kd_\ell(\vx) +\Kres(\vx,\decpar_0/2^\noLevels)  \qquad (\noLevels \geq 1)
\label{eq:multilevel_ks}
\end{equation}
where $\noLevels$ is the number of levels and
\begin{equation}
\kd_\ell(\v x) =\Kmoll(\v x,\decpar_0/2^{\ell+1}) -\Kmoll(\v x,\decpar_0/2^{\ell}).
\label{eq:diffkernel}
\end{equation}
The terms of this expansion are associated with the refinement levels of an octree. If $r_{\ell}=r_0/2^\ell$ 
denotes the box side length at level $\ell$, 
then both $r_\ell$ and the splitting parameter $\decpar_\ell=\decpar_0/2^\ell$ decrease by a factor two from one level to the next. Thus the terms of the kernel expansion have the same relative scale on every level, see illustration in \cref{fig:Diffkernel_decay}.

The difference kernel $\kd_\ell$ will be both rapidly decaying and smooth, and $\decpar_0$ will be chosen depending on an error tolerance such that the contribution from both the difference kernel and the residual kernel can be localized to nearest neighbor boxes. 
Both the difference kernel $\kd_\ell$ and the residual kernel $\Kres$ are rapidly decaying and either compactly supported (more on this in \cref{sec:window_functions}) or can be truncated beyond a cutoff radius. We will, with a slight abuse of terminology, refer to the kernels as being compactly supported in both cases.

With the split in \eqref{eq:multilevel_ks} inserted in \eqref{eq:periodic_sum}, we get 
\begin{align}
  \begin{split}
    \uP{D}(\vx_{\beta}) &= \ufarP{D}(\vx_{\beta},\decpar_0)+ \sum_{\ell=0}^{\noLevels-1} \udifflP{\ell}{D}(\vx_{\beta}) +\ulocalP{D}(\vx_{\beta},\decpar_0/2^\noLevels) + \uselfP{D}(\vx_\beta,\decpar_0) \\
    &= \ufarP{D}(\vx_{\beta},\decpar_0) + \ucompactP{D}(\vx_{\beta},\decpar_0, \noLevels) ,
  \end{split}
  \label{eq:u_four_partsDP}
\end{align}
Here $\ufarP{D}$ is the smooth, root-level contribution which depends on the imposed periodicity, while $\uselfP{D}$ is a pointwise local correction. The compact part
$\ucompactP{D}$ consists of the localized terms in the telescoping
split. 
For periodization, the key observation is that all far-field
contributions due to periodicity are contained in $\ufarP{D}$. The compact part requires only nearby periodic images, since all its interactions are restricted to local neighborhoods.
The individual terms in \eqref{eq:u_four_partsDP} are described in \cref{ss:eval_residual_difference,ss:eval_far}.

\subsection{Window functions and PSWF-based splits}
\label{sec:window_functions}

We now briefly review kernel-splitting and the use of prolate spheroidal wave functions (PSWFs).
Let $\wfunc:\R\to\R$ be an even \emph{window function}, chosen to be rapidly
decaying or compactly supported in real space, with Fourier transform
(as defined in \cref{sec:fourier-defs})
satisfying
\begin{align}
  \widehat\wfunc(0) = 1,
  \quad \text{and} \quad
  \widehat\wfunc^{(n)}(0)=0\quad \text{for $n$ odd}.
  \label{eq:wfunchat_prop}
\end{align}
Furthermore, let $\Phi(r)$ be defined by the formula
\begin{align}
    \Phi(r) = 2\int_r^\infty \wfunc(t) \dif t= 1-2\int_0^r \wfunc(t) \dif t.
    \label{eq:def_erfc_like_fcn}
\end{align}
The harmonic kernel can then be split as
\begin{equation}
  \widehat H_M(k,\decpar) = \frac{1}{k^2} \widehat\wfunc(k\decpar),
  \qquad 
H_R(r,\decpar)=\frac{\Phi(r/\decpar)}{4\pi r}.
  \label{eq:Hhatsplit_wfunc}
\end{equation}
This is equivalent to the Ewald split in \eqref{eq:ewald-split} for the Gaussian $\widehat\wfunc(k\decpar)=e^{-k^2\decpar^2/4}$. A window function as defined in \eqref{eq:wfunchat_prop} can also be used to define splitting functions for the Stokeslet and stresslet, via the biharmonic kernel. See formulas in \cref{sec:stokes-split}, and \cite{StokesDMK2026} for details. 

In \cite{jiang2025cpam}, it was shown that a more efficient
kernel-split can be obtained by replacing the Gaussian with $\prol$,
the zeroth-order PSWF, defined as the leading eigenfunction to the
truncated Fourier transform
\begin{align}
    \lambda_n^c \psi_n^c(x) = \int_{-1}^1 \psi_n^c(t) e^{ictx} \dif t,
    \quad n \ge 0, \: c > 0, \:  x\in[-1,1].
    \label{eq:PSWFeig}
\end{align}
For a comprehensive review of PSWFs, see
\cite{osipov_prolate_2013}. Equation \eqref{eq:PSWFeig} implies that
if $\prol$ is extended by zero outside $[-1,1]$, it is self-similar under the Fourier
transform inside the band $[-c,c]$,
\begin{align}
  \hatprol(k) = \lambda_0^c \prol(k/c), \quad |k| \le c .
         \label{eq:PSWFhat}
\end{align}
It is well-known that a function cannot be compactly supported in both
real and Fourier space. The key feature of $\prol$ is that it
maximizes spectral concentration, meaning that among functions
supported on $[-1,1]$, $\prol$ minimizes the fraction of $L^2$ energy
outside the band $[-c,c]$. In this sense we get the optimal compactly
supported window function through the definition
\begin{align}
   \wfunc(x)  = \frac{\prol(x)}{\lambda_0^c\prol(0)} = \frac{\prol(x)}{\int_{-1}^1 \prol(t)\dif t},
  \quad |x| \le 1.
  \label{eq:wfunc_pswf}
\end{align}
We numerically truncate its Fourier transform to zero outside the band
$[-c,c]$ and, using \eqref{eq:PSWFhat}, set
\begin{align}
  \widehat\wfunc(k) = \frac{\hatprol(k)}{\hatprol(0)} = \frac{\prol(k/c)}{\prol(0)},
  \quad |k| \le c.
  \label{eq:wfunchat}
\end{align}
The error due to this truncation decreases exponentially with the
bandlimit $c$, and in parameter selection for tolerance $\eeps$ it
holds that
\begin{align}
  c = O(\log(1/\eeps)).
  \label{eq:pswf_c}
\end{align}
In the numerical experiments, $c$ is selected by the procedure described in \cref{ss:error_control_param_selection}.
For a detailed analysis of the use of
PSWFs in kernel-splitting, in the context of Ewald summation, see
\cite{EwaldPSWF2026arxiv}.

\begin{remark}
\label{rm:comp_supp}
With $\wfunc(x)$ compactly supported in $|x|\le 1$, it holds that
$\Kres(\v x,\decpar)$ is compactly supported on
$|\v x|\in[0, \decpar]$, and equivalently that
$\Kmoll(\v x,\decpar)\equiv K(\v x)$ for $|\v x| > \decpar$. 
For the harmonic kernel split, this property is immediate from the formulas above. For the biharmonic kernel, it was proved in \cite[Theorem 1]{StokesDMK2026}, and it carries over to the Stokeslet and stresslet.
With a rapidly decaying window
function, such as the Gaussian, the residual kernel will not be
compactly supported, but decay rapidly, and allow for truncation. 
\end{remark}

\begin{remark}\label{rem:PSWF_eval}
To evaluate $\prol$, its Legendre expansion is used. Finding and
evaluating this expansion can be done efficiently, and we use the
implementation for this found in Chebfun \cite{chebfun}.
Derivatives of $\prol$ can be accurately evaluated through the
differentiation and subsequent Chebfun expansion of
\eqref{eq:PSWFeig}.
\end{remark}

\subsection{Evaluation of contributions from the residual and difference kernels}\label{ss:eval_residual_difference}
Consider the four parts in the sum \eqref{eq:u_four_partsDP}.
As was stated in \cref{rm:comp_supp}, it holds that the scaled residual kernel $\Kres(\v x,\decpar)=0$, for $|\v x|>\decpar$. The residual kernel contains the original singularity at the origin, and with an appropriate scaling, the evaluation of the local contribution 
 \begin{align}
\ulocalP{D}(\vx_\tind,\decpar) &= \sum_{\sind=1}^N \sum_{\v p\in\perindset{D}}^* \Kres(\vx_\tind-\vx_\sind + \v p,\decpar) \rho_\sind.
\label{eq:u_local}
\end{align}
is done with a direct summation in a local neighborhood of $\vx_\tind$.

Now, consider the contribution from the difference kernel at level $\ell$,
\begin{align}
\udifflP{\ell}{D}(\vx_{\beta}) &= \sum_{\sind=1}^N \sum_{\v p\in\perindset{D}}^* \kd_\ell(\vx_\tind-\vx_\sind + \v p) \rho_\sind.
\label{eq:u_diff_l}
\end{align}
The difference kernel $\kd_\ell(\v x)=\Kmoll(\v x,\decpar_0/2^{\ell+1}) -\Kmoll(\v x,\decpar_0/2^{\ell})$ is defined in \eqref{eq:diffkernel}. Since $\kd_\ell$ is smooth and rapidly decaying in Fourier space, we evaluate the sum using its Fourier transform. We therefore write 
\begin{align}
\kd_\ell(\vx)=\frac{1}{(2\pi)^3} \int_{\R^3} \kdhat_\ell(\vk)
e^{i \vk \cdot \vx} \dif\vk.
\label{eq:Dl_int}
\end{align}
This representation is useful because it expresses translations in terms of plane waves, which are used to form the expansion in the DMK algorithm. 

Let $\Kmax$ denote the effective Fourier cutoff of the selected kernel split to the prescribed tolerance $\eeps$. For a PSWF-based split, this cutoff is the PSWF bandlimit, so that $\Kmax=c$, with $c$ chosen e.g.~according to \eqref{eq:pswf_c}.
Since $\Kdiff_\ell$ contains the finer mollified scale $\decpar_{\ell+1}$, the level-dependent cutoff is $K_\ell=\Kmax/\decpar_{\ell+1}$. Thus the Fourier integral \eqref{eq:Dl_int} is truncated to $[-K_\ell,K_\ell]^3$. 
We approximate this truncated integral 
by the trapezoidal rule. Let $\setm{n_f}=\{-n_f,\ldots,n_f\}^3$, and $\NF=2n_f+1$ denote the total number of Fourier grid points per coordinate direction. With $\v k=h_\ell\v m$, we use
\begin{align}
\kd_\ell(\vx) \approx \frac{h_{\ell}^3}{(2\pi)^3}
\sum_{\vm \in \setm{n_f}} \kdhat_\ell(h_{\ell}\vm)e^{ih_{\ell}\vm \cdot \vx}.
\label{eq:Dl_int_discr}
\end{align}
The relevant real-space range is determined by the support of $\Kdiff_\ell$. At level $\ell$, the kernel vanishes outside $|\v x|>\decpar_{\ell}$. Hence the range of $\x$ is halved when $\ell$ is increased by one. At the same time, $K_\ell=\Kmax/\decpar_{\ell+1}=2(\Kmax/\decpar_\ell)$, so the Fourier interval doubles from one level to the next. 
The oscillations in the truncated integral are due to the phase factor $e^{i\v k \cdot\mathbf{x}}$, which retains the same range independent of level $\ell$. 
Following \cite[Section 3.2]{jiang2025cpam}, we require $2\pi/h_\ell\geq 3\decpar_\ell$ to avoid aliasing. We use the limiting value $h_\ell=(2\pi)/(3\decpar_\ell)$, and hence $n_f =\lceil K_\ell/h_\ell \rceil=\lceil 3\Kmax/\pi \rceil$.
With this, $n_f$, and therefore $\NF=2n_f+1$, are independent of $\ell$.

\subsection{The self-interaction correction}
\label{sec:self-inter-corr}

In \eqref{eq:u_four_partsDP}, a so called self interaction term $\uselfP{D}(\vx_\tind)$ is introduced. 
  As stated in \eqref{eq:periodic_sum}, the term when $\v p=0$ 
and $\vx_\tind=\vx_\sind$, the so called self-interaction, should be excluded from the sum. In the evaluation of $\ufarP{3}$ in \eqref{eq:ewald_far_3P}, part of it is included, and it will be so independent of the periodicity, as we evaluate in Fourier space. Hence, we can correct this by adding the term
\begin{equation}
\uselfP{D}(\vx_\tind,\decpar)= -
\left(\lim_{r\rightarrow 0} K_M(r,\decpar)\right)\rho_{\tind}.
\end{equation}
For the electrostatic case, with the mollified kernel $H_M(r,\decpar)=(1-\Phi(r/\decpar))/(4\pi r)$, we get 
$\uselfP{D}(\vx_\tind,\decpar)=-\wfunc(0)\rho_{\tind}/(2\pi \decpar)$.
The self-interaction terms for the Stokes kernels can be found in \cref{sec:stokes-split}. 

\subsection{Far-field evaluation for different periodicities}\label{ss:eval_far}

The-far field term in the triply periodic case 
$\ufarP{3}(\vx_\tind,\decpar)$, was stated in \eqref{eq:ewald_far_3P} with the generic notation for the mollified kernel. 
Specifically for electrostatics (the harmonic kernel), using the definition in \eqref{eq:Hhatsplit_wfunc}, this reads  
\begin{align}
\ufarP{3}(\vx_\tind,\decpar) 
= \frac{1}{L_1 L_2 L_3} 
\sum_{\sind=1}^N \sum_{\substack{\vk\in\wavenumset^3\\ \vk \ne 0}} 
\frac{\wh\wfunc(|\vk|\decpar)}{|\vk|^2} 
e^{i\vk\cdot(\vx_\tind-\vx_\sind)}
  \rho_\sind.
\label{eq:ewald_far_3P_es}
\end{align}
For the doubly ($D=2$) and singly periodic ($D=1$) case (as defined in \eqref{eq:periodic-images}), we get
\begin{align}
\ufarP{2}(\vx_\tind,\decpar) 
& = \frac{1}{L_1 L_2 \cdot 2\pi} 
\sum_{\sind=1}^N \sum_{\substack{(k_1,k_2) \in\wavenumset^2}} 
 \int_{\mathbb{R}} \frac{\widehat\wfunc(|(k_1,k_2,\kappa_3)|\decpar)}
 {k_1^2+k_2^2+\kappa_3^2} 
e^{i (k_1,k_2,\kappa_3)\cdot(\vx_\tind-\vx_\sind)}
  \rho_\sind  \dif\kappa_3
\label{eq:ewald_far_2P_es} \\
\ufarP{1}(\vx_\tind,\decpar) 
& = \frac{1}{L_1 (2\pi)^2} \sum_{\sind=1}^N \sum_{\substack{k_1 \in\wavenumset^1}} \int_{\mathbb{R}} \int_{\mathbb{R}}
\frac{\widehat\wfunc(|(k_1,\kappa_2,\kappa_3)|\decpar)}
{k_1^2+\kappa_2^2+\kappa_3^2} 
e^{i (k_1,\kappa_2,\kappa_3)\cdot(\vx_\tind-\vx_\sind)}  \rho_\sind  \dif\kappa_2 \dif\kappa_3,
\label{eq:ewald_far_1P_es} 
\end{align}
where the sets of discrete wavenumbers are given by, cf.~\eqref{eq:wavenumbers-3p},
\begin{equation}
  \label{eq:wavenumbers-Dp}
  \wavenumset^D
  := \left\{ 2\pi
  \left(\frac{\kintx}{L_1}, \cdots, \frac{\kint_D}{L_D}\right)
  : \kint_i \in \mathbb{Z} \right\}
  \subset \R^D, \quad D=1,\,2,\,3.
\end{equation}
Hence, for each periodic direction that is replaced by a free direction, the summation over discrete wavenumbers is replaced by integration over continuous wavenumbers, as the Fourier transform must be used instead of a discrete Fourier series. 

We could similarly generalize the expression $\ufarP{3}(\vx_\tind,\decpar)$ in \eqref{eq:ewald_far_3P} to $D=1,2$. 
The explicit formulas in \eqref{eq:ewald_far_2P_es} and \eqref{eq:ewald_far_1P_es} reveal several features that are important for the discretization for all kernels.
The Fourier window function $\widehat\wfunc$ is compactly supported, which means that the integral can be approximated with the trapezoidal rule to spectral accuracy as long as the integrand is smooth \cite{Trefethen2014}. For the term when $k_1=k_2=0$ in \eqref{eq:ewald_far_2P_es} and similarly for $k_1=0$ in \eqref{eq:ewald_far_1P_es}, the integrand will however have a singularity. 
In addition, if e.g.~$k_1=2\pi/L_1$ with $L_1$ large, then the integrand in \eqref{eq:ewald_far_2P_es} with $k_2=0$, and the corresponding integrand in \eqref{eq:ewald_far_1P_es}, will be sharply peaked.
While formally smooth, it will require a fine uniform grid to resolve. 
We will in \cref{sec:reduced_per} discuss how to utilize truncated kernels \cite{Vico2016} to still allow for the use of the trapezoidal rule in both these cases with a moderate number of discretization points.

\begin{remark} With the Gaussian window function, $\widehat\wfunc(k\decpar)=e^{-k^2\decpar^2/4}$, analytical formulas for the integrals above can be found in \cite[Equations (13)--(16)]{shamshirgar_fast_2021}. Similarly, for the Stokes kernels, corresponding integrals are given in \cite[Section 2.3]{bagge_fast_2023}. These formulas contain special functions and are not used in the fast methods in these Ewald summation papers. 
\end{remark}

\section{The DMK framework}
\label{sec:DMKalg}

We here briefly describe the original DMK algorithm for free-space \cite{jiang2025cpam} and for a triply periodic cube \cite{StokesDMK2026}. This will be necessary before we can derive its extension to rectangular cuboids and arbitrary periodicities in \cref{sec:3P,sec:reduced_per}, respectively. 

\subsection{Steps of the algorithm} We summarize the algorithm in five steps. 
\subsubsection{Octree construction} The sources and targets 
are first embedded in a box of equal sides. This is then
subdivided into an octree, where each parent box is recursively
subdivided into 8 children, until each leaf node contains at most $n_s$
sources. It is then ensured that the tree is \emph{level-restricted}\footnote{Also called \emph{2:1 balanced}.}; that
is, any two leaf nodes that share a boundary point differ by at most
one refinement level. Each box has a list of neighboring boxes, and on
the triply periodic cube these lists are extended by periodic
wrapping, after which level-restriction is enforced again. Of special
importance is the concept of a \emph{colleague} of a box $\unitbox$,
which is any box at the same refinement level as $\unitbox$ that
shares a boundary point with $\unitbox$ (this includes $\unitbox$ itself).

\subsubsection{Upward pass: construction of proxy charges} 
\label{sec:dmk_upward}
Every box
on every level is covered by a $p \times p \times p$ grid of
first-kind Chebyshev points. Starting at the leaf nodes and going up
the tree, \emph{proxy charges} are formed at the grid nodes in each box through so-called anterpolation, see \cite[Section 2.2]{jiang2025cpam}. On level $\ell$ of the tree, these proxy charges will by construction reproduce the field of the original sources in the box
to a prescribed precision $\eeps$, for the (locally smooth)
difference kernel $\kd_\ell(\v x)$.

\subsubsection{Top level: far-field interactions} 
At the root level, $\ell=0$, we need to compute $\ufarP{D}$ in
\eqref{eq:u_four_partsDP}, based on the mollified kernel $\Kmoll$. In free-space, the mollified kernel is truncated to the domain size and evaluated as a Fourier integral \cite{Vico2016}. In periodicity, the lattice sum over all periodic boxes is evaluated as a Fourier series as in \eqref{eq:ewald_far_3P_es}. The sources are the root-level proxy points, and the potential is evaluated on the Chebyshev grid.

\subsubsection{Downward pass: evaluation of difference kernels} 
\label{sec:dmk_downward}
Starting at $\ell=0$ and going down in the tree, the $p^3$ proxy points in
each box are used to compute local, $\NF^3=(2n_f+1)^3$-term Fourier expansions of the
potentials based on the difference kernel $\kd_\ell$.  For each box
$\unitbox$, the contributions to $\udifflP{\ell}{D}$
are limited to the colleagues of $\unitbox$, see \cref{fig:Diffkernel_decay}. The potential is evaluated on the Chebyshev grid
by translating and accumulating the local Fourier expansions. In a
coarse-to-fine sweep, the accumulated potentials are represented as
Chebyshev expansions, which are shifted and rescaled into the
children at level $\ell+1$.

\subsubsection{Leaf level: residual interactions} 
At each leaf node, the accumulated Chebyshev expansion of the smooth
potential is evaluated at each target location.  The remaining kernel
interactions governed by $\Kres(\vx,\decpar/2^{\noLevels})$ are localized
to nearest neighbor boxes and also evaluated directly at each target.

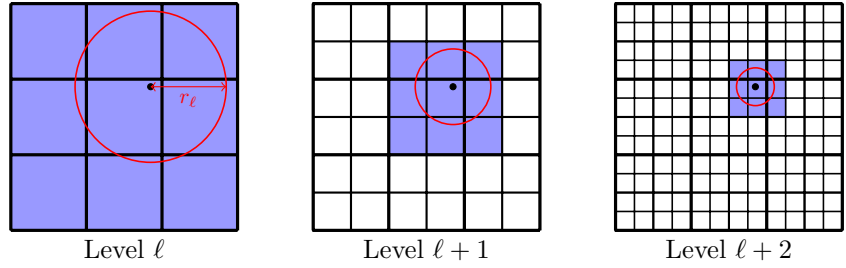
\begin{figure}[!t]
\centering
\input{interaction_levels.tikz}
\caption{A 2D sketch of how the truncation radius $r_\ell$ for
$\kd_\ell(\vx- \vx^*)$ 
decreases at each level to match the box side length, s.t.~a source
point in a box $\unitbox$ contributes only to targets in the {\em colleagues} of $\unitbox$, 
i.e.~neighboring boxes and the box itself at each level.
}
\label{fig:Diffkernel_decay}
\end{figure}

\subsection{Computational complexity}
\label{sec:comp-compl}

In the algorithm, interaction with source and target points only
happens on the leaf level, which is necessary for achieving linear
scaling. Note that $\NF$ and $p$ are constant on each level, and hence
the work required per box to evaluate the difference kernel
contribution is independent of the level.
A detailed analysis of the computational complexity of the algorithm can be found in \cite[Section 3.5.1]{jiang2025cpam}. In simplified form, it can be written as
\begin{align}
  C(N,n_s,p,n_f) = O \left( N n_s + p^3 N + p^{4}  + p \NF^3\right).  
\end{align}
An important feature contributing to the speed of the algorithm is that both the Fourier grid and the Chebyshev grid have tensor-product structure. This structure can be exploited in all operators acting on and between such grids through \emph{sum factorization}. For example, applying the Kronecker-product operator 
$A \otimes A \otimes A$, with $A \in \mathbb R^{m \times n}$, to a tensor $X \in \mathbb R^{n \times n \times n}$ has a naive cost of $O(m^3 n^3)$, but sum factorization reduces it to
\begin{align}
    O(m n^3 + m^2 n^2 + m^3 n) .
\end{align}

\section{DMK for cuboids in triple periodicity}
\label{sec:3P}

In this section, we introduce a DMK algorithm for triply periodic
boundary conditions on a rectangular cuboid. The modifications needed for a problem with reduced
periodicity will be discussed in \cref{sec:reduced_per}.
Our periodization of DMK relies on the following two key features of
the original algorithm, as formulated in \cref{sec:DMKalg}:
\begin{itemize}
\item The hierarchical subdivision of the domain into a tree, with proxy points representing the sources at each level.

  \item The coupling between the multilevel summation summation and the tree,
  such that the difference kernels and the residual kernel are localized
  to the nearest neighbor boxes on their levels,
  $\supp(\kd_\ell)=r_\ell$ and $\supp(\Kres(\vx,\decpar_0/2^\noLevels))=r_\noLevels$.
\end{itemize}
Due to the above features, it is only the mollified kernel $\Kmoll$
that mediates the contribution from distant particles (including
infinity) in a periodic setting.
At all other levels, the only effect
of periodicity is that interactions using difference kernels and the residual kernel should include periodic images of the sources inside the primary box, up to radius $r_\ell$.
Our periodization is therefore divided into two distinct parts, one for
compactly supported contribution, and one for the periodic far-field
contribution.

Consider a general rectangular domain (cuboid) with triply periodic
boundary conditions, i.e. let
$\Omega=[0,L_1)\times[0,L_2)\times[0,L_3)$ and assume (without loss of
generality) that $L_1$ is the \emph{shortest} periodic length.  We
will first formulate the strategy for the creation of a tree structure
that will allow the evaluation of the contribution from the difference
kernel \eqref{eq:u_diff_l} and the residual (local) kernel
\eqref{eq:u_local}. The evaluation of the far field part
$\ufarP{D}(\vx_{\beta},\decpar_0)$ in \eqref{eq:u_four_partsDP}, based
on the mollified kernel (\eqref{eq:ewald_far_3P},
\eqref{eq:ewald_far_3P_es}), is discussed in 
\cref{ss:eval_far_3p}.

\subsection{Domain setup and tree construction}
\label{sec:dom-set-up}
\subsubsection{Cube tiling}

We tile space using cubes of side length $L_1$, and consider the
minimal union of such cubes that contains the original periodic domain
$\Omega$. This union defines a domain
\begin{equation}
\boxset = \bigcup_{m=1}^{\notopboxes} Q_m,
\end{equation}
where each $Q_m$ is a cube with side length $L_1$ such that $\Omega
\subseteq \boxset$,
and $\notopboxes$ is the total number of such cubes.

If both $L_2$ and $L_3$ are integer multiples of $L_1$, then the
tiling is exact and $\boxset=\Omega$, as in \cref{fig:cube_tiling_exact}. Otherwise, the tiling
necessarily produces \emph{cut cubes}, defined as cubes that are only partially inside $\Omega$, and $\boxset$ is a strict
extension of $\Omega$. In this case, we further pad the domain
by adding one more layer of cubes on each side in the periodic
direction(s) that cannot be tiled exactly, yielding a total number of
$\notopboxesext$ boxes. Denote this padding region by $\Omegapad$,
such that $\Omegaext=\boxset \cup \Omegapad$.
See \cref{fig:cube_tiling_cut} for an illustration, where the two center cubes constitute $\boxset$ and the full domain (four cubes) is $\Omega_{\mathrm{ext}}$.  
In the general case, where neither $L_2$ nor $L_3$ is an integer multiple of $L_1$, we have $\Omegaext=[0, \, L_1] \times
[-L_1, \, L_1(\lceil L_2/L_1 \rceil + 1)] \times [-L_1, \, L_1(\lceil L_3/L_1 \rceil +1)]$.
For an exact tiling, $\Omegapad=\emptyset$, $\Omegaext=\boxset=\Omega$, and $\notopboxesext=\notopboxes$.

\begin{figure}[!t]
\centering
\begin{subfigure}[t]{0.49\textwidth}
\includegraphics[width=0.8\linewidth]{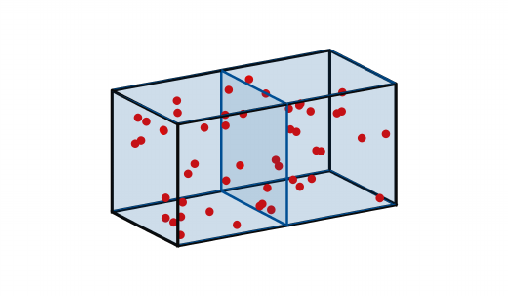}
\caption{Exact tiling: $\Omegaext=\boxset$ and $\Omegapad=\emptyset$.}
\label{fig:cube_tiling_exact}
\end{subfigure}
\begin{subfigure}[t]{0.49\textwidth}
\includegraphics[width=\linewidth]{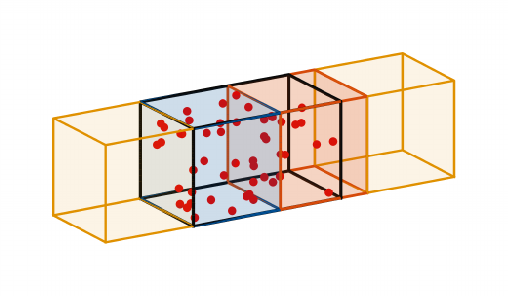}
\caption{Non-exact tiling: $\Omegaext=\boxset\cup\Omegapad$.}
\label{fig:cube_tiling_cut}
\end{subfigure}
\caption{Cube tiling for a triply periodic cuboid. Panel (a) shows and exact tiling, while panel (b) shows a non-exact tiling requiring padding. The blue and red cubes denote the uncut and cut cubes, respectively, and together define $\boxset$. The yellow cubes denote the padded domain $\Omegapad$.}
\label{fig:cube_tiling}
\end{figure}

\subsubsection{Periodicity and source replication}
If $\Omegaext \ne \Omega$, the extension 
$\Omegaext \setminus \Omega$ is populated with periodic images of all sources in $\Omega$ that lie inside $\Omegaext$. See \cref{fig:DMK_extdom} for an illustration. 
In all coordinate directions $i$ for which the tiling is exact (this is always true for $i=1$), $\Omegaext$ is set to be periodic. 
In the remaining coordinate direction(s) it is set to be non-periodic.

\begin{figure}[!t]
\centering
\input{periodic_L1.tikz}
\caption{A sketch of the domain $\Omegaext$ (outer red line) in the $1 \! - \! 2$ plane for evaluation of $\ucompactP{D}$, as defined in \eqref{eq:u_four_partsDP} with $D=2$ or $3$ and $L_2$ not an integer multiple of $L_1$. Sources inside $\Omega$ in black, periodic copies of these sources in blue. Each of the red $\notopboxesext$ cubes will be adaptively refined in an octree structure.}
\label{fig:DMK_extdom}
\end{figure}
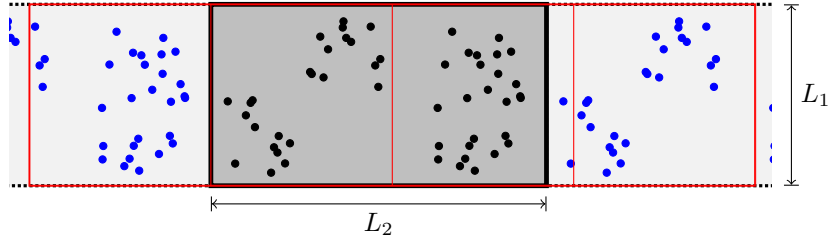

\subsubsection{Tree construction on tiled domain}

After tiling and padding, the cuboid $\Omegaext$ is a \emph{polycube} domain, built by cubes of size $L_1$. This is hierarchically subdivided into a \emph{forest of octrees}, where each cube contains an octree, and the neighbor lists on each level include boxes from octrees in neighboring cubes. Periodicity is added to the neighbor lists in the exactly tiled coordinate directions, and level-restriction is enforced on the entire forest.

\subsection{Evaluation of compact contribution}
With the forest of octrees built on $\Omegaext$, where $\Omegaext$ has periodic images of sources inside $\Omega$ in extended directions and periodicity set in the exactly tiled directions, 
 we are ready to compute the compact contribution $\ucompactP{3}$ in \eqref{eq:u_four_partsDP} with small modifications to the standard DMK algorithm. We defer discussion of computation of the top level potential $\ufarP{3}$ to \cref{ss:eval_far_3p}, and here describe the upward and downward passes. The upwards pass (\cref{sec:dmk_upward}) is largely the same, the only change being that proxy points are created in each tree, with the process stopping at the root of each tree.
The downward pass as described in \cref{sec:dmk_downward} is then run in each octree, starting at the tree root with $r_0=L_1$ at $\ell=0$. On each level, the potential is only accumulated in boxes that contain target points, meaning that all boxes in $\Omegapad$ and uncut boxes of $\boxset \setminus \Omega$ at any level $\ell$ are excluded.

\subsection{Evaluation of far-field contribution}\label{ss:eval_far_3p}

Before the compact contribution is evaluated in the downward pass, using the tiling and possible
extension described above, the root-level, 
far-field contribution $\ufarP{3}$ in \eqref{eq:ewald_far_3P} must be computed on the
Chebyshev grids of the $\notopboxes$ top-level boxes in $\boxset$.
These grids are tensor products of the $p$ Chebyshev points $\eta_j$,
$j=1,\ldots,p$, scaled to the interval $L_1[-1/2,1/2]$ and shifted to the box center.
Our goal is thus to evaluate
\begin{align}
  \ufar^\boxind(j_1,j_2,j_3) = \ufarP{3}(\boxcenter+(\eta_{j_1},\eta_{j_2},\eta_{j_3})),
  \label{eq:ufar_chebgrid}
\end{align}
for all $j_i=1,\ldots,p$, $\boxind=1,\ldots,\notopboxes$, where
$\boxcenter$ is the center coordinate for box $\boxind$.

Introduce the structure factor
\begin{align} \chi(\vk)&= 
\frac{1}{|\Omega|} \sum_{\sind=1}^N
e^{-i\vk\cdot \vx_\sind}
  \rho_\sind, \qquad \vk \in \wavenumsettrunc^{3\per}.
\label{eq:structurefacgen}
\end{align}
The Fourier modes are given by the set
\begin{align}
  \wavenumsettrunc^{3\per}
  &:= \left\{ 2\pi
  \left(\frac{\kintx}{L_1}, \frac{\kinty}{L_2}, \frac{\kintz}{L_3}\right)
  : \kint_i \in \mathbb{Z},  \ 
  -\kmaxi \le \kint_i \le \kmaxi, |\kintx|+|\kinty|+|\kintz|\ne 0
  \right\},
  \label{eq:wavenumbers-3p-trunc}
  \\
  \kmaxx&=\lceil L_1 \Kmax/(2\pi) \rceil,
  \quad \kmaxy=\lceil L_2 \Kmax/(2\pi) \rceil,
  \quad \kmaxz=\lceil L_3 \Kmax/(2\pi) \rceil,
\label{eq:kintmax}
\end{align}
where $\Kmax$ is the Fourier bandwidth required to satisfy error
tolerance $\eeps$, given the kernel-split parameter $\decpar_0$.

Furthermore, let
\begin{align}
  S(j,\kbar,L)=e^{i 2\pi \kbar \eta_j/L},
\end{align}
so that \eqref{eq:ufar_chebgrid} can be written as
\begin{equation}
\ufar^\boxind(j_1,j_2,j_3)=\sum_{\kintx=-\kmaxx}^{\kmaxx,*} S(j_1,\kintx,L_1) \sum_{\kinty=-\kmaxy}^{\kmaxy} S(j_2,\kinty,L_2) \sum_{\kintz=-\kmaxz} ^{\kmaxz} S(j_3,\kintz,L_3) 
e^{i \vk \cdot \boxcenter} \Kmollhat(\vk,\decpar) \chi(\vk),
\label{eq:eval_tensor}
\end{equation}
where $*$ denotes that the term $\kintx=\kinty=\kintz=0$ is excluded, in accordance with the definition in \eqref{eq:wavenumbers-3p-trunc}.

Now, in the case that the tiling is exact, i.e.~$\boxset=\Omega$, the set of $\notopboxes$ cubes matches the periodic domain, and in each of these top cubes, we can use the equivalent proxy sources to evaluate the structure factor $\chi(\vk)$. Remember that the location for the proxy sources are the $p^3$ Chebyshev points, and denote the proxy strengths by $\tilde{\rho}^\boxind(j_1,j_2,j_3)$, where $\boxind$ is the cube index.
The structure factor can then be evaluated as
\begin{equation}
\chi(\vk)=\frac{1}{|\Omega|}
\sum_{\boxind=1}^{\notopboxes} e^{-i \vk \cdot \boxcenter}
\sum_{j_3=1}^p \overline{S(j_3,\kint_3,L_3)}
\sum_{j_2=1}^p \overline{S(j_2,\kint_2,L_2)}
\sum_{j_1=1}^p \overline{S(j_1,\kint_1,L_1)}
\, \tilde{\rho}^\boxind(j_1,j_2,j_3),
\label{eq:structurefac}
\end{equation}
Computationally, it is beneficial to use this tensor-product structure of the identical set of (proxy) sources \eqref{eq:structurefac} and targets \eqref{eq:eval_tensor}, as mentioned in \cref{sec:comp-compl}. In the case that the tiling is not exact, proxy sources can however not be used in boxes that are only partially in $\Omega$.
For such boxes we instead use the original, nonuniform source points, but only those that lie inside $\Omega$. The outer sum for the structure factor in \eqref{eq:structurefac} is then over the uncut boxes, with an additional contribution from the cut boxes. To reduce the number of original source points that needs to be summed, we can for each cut box use proxy sources of its children if any of them are uncut, continuing down the tree to fill $\Omega$ with boxes to the extent possible. We then use the original sources in the part of $\Omega$ not covered by the union of these boxes, easily identified as the source points lying in cut leaf boxes. See \cref{fig:top_level_proxy} for an illustration.

\begin{figure}[!t]
\centering
\begin{subfigure}[t]{0.49\textwidth}
\includegraphics[trim={0.45cm 0.0cm 0.75cm 0.0cm},clip,width=\linewidth]{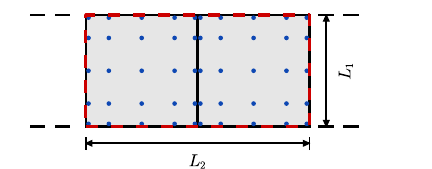}
\caption{$\Omega=\boxset$ (exact tiling)}
\label{fig:top_level_proxy_uncut}
\end{subfigure}
\begin{subfigure}[t]{0.49\textwidth}
\includegraphics[trim={0.0cm 0.0cm 1.2cm 0.0cm},clip,width=\linewidth]{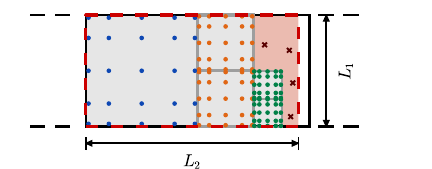}
\caption{$\Omega \subset \boxset$}
\label{fig:top_level_proxy_cut}
\end{subfigure}
\caption{Depending on the domain, the structure factor $\chi(\vk)$ can be evaluated (a) using only proxy sources on the root level, (b) using proxy sources from boxes inside $\Omega$ on different levels in the tree, augmented by original sources in leaf boxes that are only partially inside $\Omega$. 
}
\label{fig:top_level_proxy}
\end{figure}

\begin{remark}
Note that the number of periodic Fourier modes for the far field evaluation typically is  small if aspect ratios of the periodic cuboid $\Omega$ are moderate. 
In the case of a cubic domain, the results in 
\cref{tab:param-table-laplace,tab:param-table-stokes} show that the total number of Fourier modes in each direction ($\Nper_i=2\kmaxgen+1$) is $3$, $5$, $9$ and $11$, for error tolerances of $10^{-3}$, $10^{-6}$, $10^{-9}$ and $10^{-12}$, respectively, for the parameter choices made. This is true for the harmonic kernel (electrostatics) (\cref{tab:param-table-laplace}) and the Stokes kernels (Stokeslet, stresslet, rotlet) (\cref{tab:param-table-stokes}), apart from a lower number of modes for the most strict tolerance for the rotlet. 
For cuboids, the number of modes needed in the two longer directions will be approximately multiplied with $L_2/L_1$ and $L_3/L_1$, see \eqref{eq:kintmax}. As long as these ratios are moderate, so will the number of Fourier modes be. 
\end{remark}

\subsection{Large aspect ratio boxes}\label{ss:large_aspect_ratio}
\label{sec:3P-large-ar}

With $\notopboxes$ top-level boxes, direct evaluation of $\ufarP{3}$ in \eqref{eq:ufar_chebgrid}, using \eqref{eq:eval_tensor}, has cost $O(\notopboxes^2)$. This can become significant for large-aspect-ratio cuboids (remember, for a cubic domain we have $M=1$).
In the case where the ratios $L_2/L_1$ and $L_3/L_1$ are integers, so that the cubical tiling of $\Omega$ is exact, 
we can use \eqref{eq:structurefac} to first evaluate the structure factor, followed by \eqref{eq:eval_tensor} to evaluate the far-field contribution at all proxy points in the top-level cubes. 
Both these computations can be accelerated using FFTs.
We will first consider the case where $L_3/L_1$ is large, while $L_2/L_1$ is moderate, and then move onto the case where both ratios are large.

\subsubsection*{Case 1: elongated cuboid}

Assume for simplicity of notation that $L_2=L_1$ and $L_3=\notopboxes L_1$.  Thus $\Omega$ is tiled by $\notopboxes$ cubes of side length $L_1$, stacked in the $x_3$-direction. The center of cube $m$ is
\begin{equation} \cb_m =
L_1\left(\frac12,\frac12,m-\frac12\right),
\qquad
m=1,\ldots,\notopboxes.
\end{equation}
For fixed $(\kint_1,\kint_2)$, define
\begin{equation}
\beta_m(j_3;\kint_1,\kint_2)
:=
\sum_{j_2=1}^p \overline{S(j_2,\kint_2,L_1)}
\sum_{j_1=1}^p \overline{S(j_1,\kint_1,L_1)}
\,\tilde{\rho}^{m}(j_1,j_2,j_3).
\label{eq:beta_def}
\end{equation}
We have that
\begin{equation}
e^{-i\vk\cdot \cb_m}
=
e^{-i\pi\left(\kint_1+\kint_2+\kint_3/\notopboxes\right)}
e^{-i2\pi \kint_3(m-1)/\notopboxes}.
\label{eq:exp_center_coord}
\end{equation}
Inserting this explicit expression and $\beta_m$ as defined in 
\eqref{eq:beta_def} into \eqref{eq:structurefac}, and swapping the
order of the two outer sums, we get
\begin{equation}
\chi(\vk)
=
\frac{1}{L_1L_2 L_3}
e^{
-i\pi\left(
\kint_1+\kint_2
+\kint_3/\notopboxes
\right)}
\sum_{j_3=1}^p \overline{S(j_3,\kint_3,L_3)}
\sum_{m=1}^{\notopboxes}
e^{-i2\pi \kint_3 (m-1)/\notopboxes}
\, \beta_m(j_3;\kint_1,\kint_2).
\end{equation}
Let $s = \kint_3 \bmod \notopboxes$, and write
$\kint_3=j \notopboxes+s$, $j,s\in \mathbb{Z}$ and $0\le s <
\notopboxes$. Then the exponential in the sum above becomes $e^{-i2\pi s(m-1)/\notopboxes}$.
We now define the one-dimensional (1D) discrete Fourier transform
\begin{equation}
\widehat{\beta}_s(j_3;\kint_1,\kint_2)
:=
\sum_{m=1}^{\notopboxes}
e^{-i2\pi s(m-1)/\notopboxes}
\,\beta_m(j_3;\kint_1,\kint_2),
\qquad
s=0,\ldots,\notopboxes-1.
\label{eq:1d-DFT-beta}
\end{equation}
This is exactly the quantity computed by a length-$\notopboxes$ FFT
over the cube index $m$.
With this, the structure factor can be written as
\begin{equation}
\chi(\vk)
=
\frac{1}{L_1L_2 L_3}
e^{-i\pi\left(\kint_1+\kint_2+\kint_3/\notopboxes\right)}
\sum_{j_3=1}^p
\overline{S(j_3,\kint_3,L_3)}
\,\widehat{\beta}_{(\kint_3 \bmod \notopboxes)}(j_3;\kint_1,\kint_2).
\label{eq:structfac-case1-wDFTCoeffs}
\end{equation}

From \eqref{eq:kintmax}, we have
$
\kmaxx = \kmaxy = O(\Kmax L_1)$
and 
$\kmaxz = O(\Kmax L_3) = O(\notopboxes\,\Kmax L_1)$.
Hence, the number of wave vectors $\vk$ in $\wavenumsettrunc^{3\per}$ scales as
\begin{equation}
O\big(\kmaxx \kmaxy \kmaxz\big) = O(\notopboxes\, \Kmax^3).
\end{equation}
Direct evaluation of $\chi(\vk)$ using \eqref{eq:structurefac} costs
$O(\notopboxes\, p^3)$ for each vector  $\vk\in \wavenumsettrunc^{3\per}$,
and therefore, the total cost is $O(\notopboxes^2\, \Kmax^3\, p^3)$. Exploiting the tensor-product structure, sum factorization reduces this to
\begin{align}
    O\left(\notopboxes\, \Kmax\, p^3 + \notopboxes\, \Kmax^2\, p^2 + \notopboxes^2\, \Kmax^3\, p\right).
    \label{eq:compcompl-sumfact-1long}
\end{align}
With FFT acceleration, we reduce the final term of the above expression. Step-by-step, the procedure is:
\begin{enumerate}[label=(\roman*)]
  \item Form all $\beta_{\boxind}(j_3;\kint_1,\kint_2)$ using sum factorization, at cost
$O(\notopboxes\, \Kmax\, p^3 + \notopboxes\, \Kmax^2\, p^2)$.
  \item Compute the 1D FFT \eqref{eq:1d-DFT-beta} over $m$ for each $(j_3, \kint_1,\kint_2)$, at total cost
$O(\notopboxes\, \log \notopboxes \, \Kmax^2\, p)$.
  \item Compute the final sum  \eqref{eq:structfac-case1-wDFTCoeffs} over $j_3$, at cost   $O(\notopboxes\, \Kmax^3\, p)$.
\end{enumerate}
Thus, the total complexity in forming $\chi(\vk)$ becomes
\begin{equation}
O\left(
\notopboxes
\left( \Kmax\, p^3
+ \Kmax^2\, p^2
+ \Kmax^3\, p \right)
+ \notopboxes\, \log \notopboxes \, \Kmax^2\, p
\right).
\label{eq:compcompl-fft-acc-1long}
\end{equation}
As we assume $\notopboxes$  to be large, the main improvement is that
the cost of summation over boxes in the long direction is reduced from $O(\notopboxes^2)$
to $O(\notopboxes\, \log \notopboxes)$.

Now, in the next step, when we want to evaluate $\ufar^m(j_1,j_2,j_3)$
as given in \eqref{eq:eval_tensor}, we have similarly to above, 
\begin{equation}
e^{i\vk\cdot \cb_m}
=
e^{i\pi\left(\kint_1+\kint_2+\kint_3/\notopboxes\right)}
e^{i2\pi \kint_3(m-1)/\notopboxes}.
\end{equation}
With this, the innermost sum over $\kintz$
can be written as
\begin{equation}
\sum_{s=0}^{\notopboxes-1}
e^{i2\pi s(m-1)/\notopboxes}
\sum_{\substack{\kintz=-\kmaxz\\ \kintz \bmod \notopboxes=s}}^{\kmaxz}
e^{i\pi\kintz/\notopboxes}
S(j_3,\kintz,L_3)
\Kmollhat(\vk,\decpar)\chi(\vk). 
\label{eq:k3_sum_mod}
\end{equation}
Note, that for each $s$, we are summing over a subset of $\kint_3$. In total, we
are still summing over all  $\kint_3$, using all information in $\Kmollhat(\vk,\decpar)\chi(\vk)$.
Considering the full sum, move out the sum over $s$, to obtain
\begin{equation}
\ufar^m(j_1,j_2,j_3)
=
\sum_{s=0}^{\notopboxes-1}
e^{i2\pi s(m-1)/\notopboxes}
G_s(j_1,j_2,j_3),
\label{eq:ufar_inverse_fft}
\end{equation}
where
\begin{equation}
\begin{aligned}
G_s(j_1,j_2,j_3)
&:=
\sum_{\kintx=-\kmaxx}^{\kmaxx,*}
e^{i\pi\kintx} S(j_1,\kintx,L_1)
\sum_{\kinty=-\kmaxy}^{\kmaxy}
e^{i\pi\kinty} S(j_2,\kinty,L_2) \\
&\qquad \times
\sum_{\substack{\kintz=-\kmaxz\\ \kintz \bmod \notopboxes=s}}^{\kmaxz}
e^{i\pi\kintz/\notopboxes}
S(j_3,\kintz,L_3)
\Kmollhat(\vk,\decpar)\chi(\vk),
\qquad s=0,\ldots,\notopboxes-1.
\end{aligned}
\label{eq:G_def_inverse_fft}
\end{equation}
Thus, for each fixed $(j_1,j_2,j_3)$, the values
\begin{equation}
\ufar^m(j_1,j_2,j_3),
\qquad m=1,\ldots,\notopboxes,
\end{equation}
are obtained from the sequence $G_s(j_1,j_2,j_3)$ by a length-$\notopboxes$ inverse FFT over $s$.
Again, the cost of evaluation is reduced from $O(\notopboxes^2)$ to $O(\notopboxes \,\log \notopboxes)$.

Here, we assumed $L_2=L_1$. If $L_2/L_1>1$, only the indexing of the boxes would change, the acceleration in the long direction would remain the same. 

\subsubsection*{Case 2: two long sides}
Assume that
$L_2 = M_2 L_1$, $L_3 = M_3 L_1$, $\notopboxes = M_2 M_3$, 
so that the domain $\Omega$ is tiled by $M_2 \times M_3$ cubes of side length $L_1$. Index the cubes by $(m_2,m_3)$, with $m_2=1,\ldots,M_2$, $m_3=1,\ldots,M_3$, and define their centers by
\begin{equation}
\cb_{m_2,m_3} = L_1\left(\tfrac12,\, m_2-\tfrac12,\, m_3-\tfrac12\right).
\end{equation}
Then the phase factor separates as
\begin{equation}
e^{-i\vk\cdot \cb_{m_2,m_3}}
=
C(\kint_1,\kint_2,\kint_3)
e^{-i2\pi \kint_2 (m_2-1)/M_2}
e^{-i2\pi \kint_3 (m_3-1)/M_3},
\end{equation}
where $C(\kint_1,\kint_2,\kint_3)$ is independent of $(m_2,m_3)$,
and can be explicitly written out, similar to 
\eqref{eq:exp_center_coord}. For fixed $\kint_1$, define
\begin{equation}
\beta_{m_2,m_3}(j_2,j_3;\kint_1)
=
\sum_{j_1=1}^p
\overline{S(j_1,\kint_1,L_1)}
\,\tilde{\rho}^{m_2,m_3}(j_1,j_2,j_3),
\end{equation}
and its 2D discrete Fourier transform
\begin{equation}
\widehat{\beta}_{s_2,s_3}(j_2,j_3;\kint_1)
=
\sum_{m_2=1}^{M_2}\sum_{m_3=1}^{M_3}
e^{-i2\pi s_2 (m_2-1)/M_2}
e^{-i2\pi s_3 (m_3-1)/M_3}
\,\beta_{m_2,m_3}(j_2,j_3;\kint_1),
\end{equation}
with $s_2=\kint_2 \bmod M_2$, $s_3=\kint_3 \bmod M_3$. The structure factor becomes
\begin{equation}
\chi(\vk)
=
\frac{C(\kint_1,\kint_2,\kint_3)}{L_1L_2L_3}
\sum_{j_3=1}^p \overline{S(j_3,\kint_3,L_3)}
\sum_{j_2=1}^p \overline{S(j_2,\kint_2,L_2)}
\,\widehat{\beta}_{\kint_2 \bmod M_2,\ \kint_3 \bmod M_3}(j_2,j_3;\kint_1).
\end{equation}
Direct evaluation for all $\vk \in \wavenumsettrunc^{3\per}$ costs
$O(\notopboxes^2 \Kmax^3 p^3)$. With the 2D FFT, forming all
$\beta$ costs $O(\notopboxes\Kmax p^3)$, computing all FFTs costs
$O(\notopboxes \log \notopboxes \Kmax p^2)$, and assembling all $\chi(\vk)$ costs $O(\notopboxes \Kmax^3 p^2)$, yielding the total complexity
\begin{equation}
O(\notopboxes \Kmax p^3)
+
O(\notopboxes \log \notopboxes \, \Kmax p^2)
+
O(\notopboxes \Kmax^3 p^2).
\label{eq:complexity-2dfft}
\end{equation}

Similarly, we can modify what was done in the previous case to evaluate 
$\ufar^m(j_1,j_2,j_3)$, $m=1,\ldots,\notopboxes$, using 2D IFFTs. 

\section{Extension to reduced periodicity}
\label{sec:reduced_per}
When evaluating $\uP{D}(\x)$ for $D=1,2$, the modifications to the evaluation of $\ucompactP{D}$ are minor compared to the triply periodic case. To see this, consider the definition of $\Omegaext$ in \cref{sec:dom-set-up}. For $D=1$, $L_1$ is the only periodic length, while $L_2$ and $L_3$ are chosen as the smallest integer multiples of $L_1$ such that the domain contain all sources. For $D=2$, $L_3$ is again chosen as the smallest possible integer multiple of $L_1$. If the periodic length $L_2$ is an integer multiple of $L_1$, the tiling is exact and this direction is treated as periodic. Otherwise, the required domain extension and source replication, are performed in this direction, which is then treated as non-periodic. With this construction, the forest of octrees can be built, after which the same evaluation procedure for $\ucompactP{D}$ applies for all $D$, provided that the periodicity is correctly reflected in the nearest-neighbor lists.

Major modifications are, however, needed in the evaluation of $\ufarP{D}$. Unlike the $3\per$ case, $\ufarP{D}$ is not simply given by a discrete Fourier sum, but instead involves Fourier integrals in the free directions, as in \eqref{eq:ewald_far_2P_es}--\eqref{eq:ewald_far_1P_es}.
To enable discretization of these integrals by the trapezoidal rule, thereby retaining the same discrete-sum structure as in the triply periodic case, we employ so-called truncated kernels \cite{Vico2016} and build on the approaches developed in \cite{shamshirgar_fast_2021,bagge_fast_2023}.
To describe this approach, we first adopt a partial differential equation (PDE) perspective to expose the corresponding Green's functions, before introducing the truncated-kernel idea and the resulting discretization. As in the triply periodic case, FFT acceleration can be used when needed for large aspect ratios. This is discussed in the last subsection.

\subsection{The PDE perspective}\label{ss:pde_perspective}
\label{sec:screening}
The field $\uP{D}(\x)$ as defined in \eqref{eq:periodic_sum}, is for the electrostatic case, i.e.~the harmonic kernel, the solution to 
\begin{align}
-\Delta \uP{D}(\x) =  f^{D\per}(\x), \quad f^{D\per}(\x)=\sum_{\v p \in \perindsetD} \sum_{\sind=1}^N \rho_{\sind} \delta (\x-\x_{\sind} + \v p),  \quad \x \in \R^3.
\label{eqn:laplace}
\end{align}
We have introduced the decomposition of the solution as a split of the
kernel. Equivalently, we can view it as a
result of a split of the right hand side.
For electrostatics, we introduce the so-called screening function
$\screen^H(r/\decpar)=\wfunc(r/\decpar)$, a radial function with Fourier transform
$\screenhat^H(k\decpar)=\widehat\wfunc(k\decpar)$, where $\wfunc$ is the window function introduced previously. The screening function, and hence this equality, is specific to the electrostatic problem. This accounts for the seemingly double notation; see \cref{rem:biharm_screening}.
With this, we decompose $f^{D\per}$ into two parts:
\begin{align*}
  f^{D\per}(\x) = \underbrace{f^{D\per}(\x) - (f^{D\per}* \screen^H)(\x)}_{=:
    f^{D\per,R}(\x,\decpar)} + \underbrace{(f^{D\per}* \screen^H)(\x)}_{=:
    f^{D\per,M}(\x,\decpar)}.
\end{align*}
The Poisson equation can be solved for each of the two parts of the 
right-hand side to find $\ulocalP{D}(\x,\decpar)$ and $\ufarP{D}(\x,\decpar)$, which can then be added\footnote{This is the single level split. A multilevel split parallel to \eqref{eq:multilevel_ks} can also be introduced.}. Introduce 
\begin{equation}
f^{D\per,M}(\x,\decpar)=\sum_{\v p \in P_D} f(\vx+\v p,\decpar), \quad f(\vx,\decpar)=\sum_{\sind=1}^N \rho_{\sind} \screen^H(\x-\x_{\sind},\decpar).
\label{eq:sum_of_screenH}
\end{equation} 
For the triply periodic case, we can expand $\ufarP{3}(\x,\decpar)$ in a Fourier series. Using the Poisson summation formula \eqref{eqn:poisson_summation} for
$D=3$ on $f^{3P,M}(\x)$, and discrete orthogonality after inserting into \eqref{eqn:laplace}, we obtain the formula  \eqref{eq:ewald_far_3P_es} as $\screenhat^H(|\k|\decpar)=\widehat\wfunc(|\k|\decpar)$.
The factor $\wh\wfunc(|\vk|\decpar)/|\vk|^2$ is the specific case of $\Kmollhat(\vk,\decpar)=\Khat(\vk) \screenhat(|\vk|\decpar)$.

In the singly periodic case, $D=1$, we expand in a Fourier series in the periodic direction. We introduce the notation $\x=(x,\vr)=(x,y,z)$, $\k=(k_1,\vkappa)=(k_1,\kappa_2,\kappa_3)$, and $\kappa^2=|\vkappa|^2=\kappa_2^2+\kappa_3^2$. We write  
\begin{equation}
\ufarP{1}(\x,\decpar)=\sum_{k_1} \uFourierP_{k_1}(\vr) e^{ik_1 x}.
\label{eq:1Psum_coeffs_of_r}
\end{equation}
Using the Poisson summation formula \eqref{eqn:poisson_summation} for
$f^{1\per,M}(\x,\decpar)$, we get a similar expansion. Inserting into \eqref{eqn:laplace} and using (discrete) orthogonality, we get
\begin{equation}
\left(-\Delta_{2D} +k_1^2\right) \uFourierP_{k_1}(\vr) = \frac{1}{L_1} F_{k_1}(\v r), \quad 
k_1=2\pi \kintx/L_1, \quad \kintx \in \mathbb{Z}, 
\label{eq:2DmodH_k}
\end{equation}
where $\Delta_{2D}$ denotes the two-dimensional Laplacian and 
\begin{equation}
F_{k_1}(\v r)=\int_{\reals} f(\x,\decpar) \, e^{-i k_1 x} \dif x,
\notag
\end{equation}
where $f$ was introduced in \eqref{eq:sum_of_screenH}.
The discrete Fourier modes $\uFourierP_{k_1}(\vr)$ can be represented in terms of a continuous Fourier transform in the non-periodic directions $y$ and $z$, 
\begin{equation}
\uFourierP_{k_1}(\vr) =\frac{1}{(2\pi)^2} \int_{\reals^2} \uHatFourierP_{k_1}(\vkappa)e^{i \vkappa \cdot \vr} \dif\vkappa,
\label{eq:Uk1r_from_Uk1kappa}
\end{equation}
and similarly for $F_{k_1}(\v r)$.
Inserting into \eqref{eq:2DmodH_k}, and using orthogonality, we get 
\begin{equation}
\uHatFourierP_{k_1}(\vkappa)= \frac{1}{k_1^2+\kappa^2} \wh{F}_{k_1}(\vkappa),
\notag
\end{equation}
where $\wh{F}_{k_1}(\vkappa)=\wh{f}((k_1,\vkappa),\decpar)=\wh{f}(\k,\decpar)$, i.e.~the three-dimensional Fourier transform of $f(\x,\decpar)$, hence
\begin{equation}
\uHatFourierP_{k_1}(\vkappa)=\frac{1}{L_1} \frac{1}{k_1^2+\kappa^2} 
\sum_{\sind} \rho_{\sind} \screenhat^H(|\k|\decpar) e^{-i \k \cdot \x_{\sind}}.
\end{equation}
Inserting this into \eqref{eq:Uk1r_from_Uk1kappa}, and further into the sum \eqref{eq:1Psum_coeffs_of_r}, we obtain \eqref{eq:ewald_far_1P_es}. 
We can follow a similar procedure for the doubly periodic case, and arrive at \eqref{eq:ewald_far_2P_es}. In that case, we expand over two periodic directions, and get coefficients $\uFourierP_{k_1,k_2}(z)$ that obey
\begin{equation}
\left(-\frac{d^2}{dz^2} +(k_1^2+k_2^2)\right) \uFourierP_{k_1,k_2}(z) = \frac{1}{L_1 L_2} F_{k_1,k_2}(z), \quad 
k_j=2\pi \kint_j/L_j, \quad \kint_j \in \mathbb{Z},  \quad j=1,2.
\label{eq:1DmodH_k}
\end{equation}

As noted before, the integrands in \eqref{eq:ewald_far_2P_es}--\eqref{eq:ewald_far_1P_es} are singular/sharply peaked for zero/small magnitude discrete wave number(s), and it is not suitable to discretize the integrals with the trapezoidal rule. 
The fact that our coefficients $\uFourierP_{k_1}(\vr)$ and $\uFourierP_{k_1,k_2}(z)$ are solutions to differential equations however allows us to use the technique of truncated Green's functions, to be discussed next. 

\begin{remark}
The screening function $\screen^H$ is specific for the PDE considered. If we instead of \eqref{eqn:laplace} consider the biharmonic equation, $\Delta^2 \uP{D}(\x) =  f^{DP}(\x)$, the harmonic screening function $\screen^H$ cannot be used, as it will not lead to a rapidly decaying local part. In this case we will use the biharmonic screening function $\screen^B$ introduced in \eqref{eq:gammaB_in_phi}. This will be important in our derivations, as while the rotlet can be found by differentiation of the harmonic Green's function, the other two Stokes kernels considered here (the Stokeslet and stresslet) are based on the biharmonic Green's function. 
\label{rem:biharm_screening}
\end{remark}

\subsection{Truncated kernel modification} \label{ss:truncated_kernel_modification}
The differential equations in \eqref{eq:2DmodH_k} and \eqref{eq:1DmodH_k} are free space problems that we write in the form
\begin{align}
\left( -\Delta_{2} +\alpha^2 \right) U(\vr) & = g(\v r), \quad \vr \in \reals^2,
\label{eq:2DmodH}
\\
\left(-
\Delta_{1}+ \alpha^2 \right) U(z) & = g(z), \quad z \in \reals,
\label{eq:1DmodH}
\end{align}
where $\Delta_{2}=\partial_y^2+\partial_z^2$, $\Delta_{1}=d_z^2$, and $\alpha \ge 0$.
The Green's functions $G^{\alpha}_{d}$ for these equations are equal to\footnote{Note, that $d$ here denotes the dimension of the problem, and in our context we will have $d=3-D$, where $D$ is the number of periodic directions.} 
\begin{equation}
G^{\alpha}_{2}(r) = \left\{ 
\begin{array}{cl}
 K_0(\alpha r)/(2\pi),& \alpha \ne 0,\\
-\log(r)/(2\pi), & \alpha = 0,
\end{array}
\right. \quad r=|\vr|,
\label{eq:Green2D}
\end{equation}
where $K_0$ is the incomplete modified Bessel function of the second kind, and
\begin{equation}
G^{\alpha}_{1}(z) = \left\{ 
\begin{array}{cl}
e^{-\alpha |z|}/(2\alpha), & \alpha \ne 0,\\
- |z|/2, & \alpha = 0.
\end{array}
\right. 
\label{eq:Green1D}
\end{equation}
For $\alpha \ne 0$, the free-space Green’s functions are exponentially decaying, and hence solutions satisfy decay conditions at infinity in both 1D and 2D. For $\alpha=0$, there would generally be growth at infinity. In our case, the right hand sides are partial Fourier transforms of the right hand side defined in \eqref{eq:sum_of_screenH}. Under the charge neutrality condition $\sum_{\sind=1}^{N}\rho_{\sind}=0$, these growth terms vanish. The two-dimensional solution (the singly periodic case), again vanishes at infinity, while for the one dimensional solution (doubly periodic case), it will approach a signed constant, determined by the dipole moment in the free direction. See \cite{tornberg2016ewald} for more details. 

\begin{remark}
The limit $\alpha \to 0$ cannot be taken directly in the Green’s functions for $\alpha \ne 0$ in \eqref{eq:Green2D}–\eqref{eq:Green1D} to recover the case $\alpha = 0$. We present these expressions together for notational convenience, but in what follows we treat $\alpha$ as either identically zero or strictly positive, bounded below by the smallest nonzero discrete wave number, as introduced in the previous section. 
For completeness, we note that as $\alpha \to 0$, the modified Helmholtz Green’s function ($\alpha \ne 0$) converges to the harmonic Green’s function ($\alpha = 0$) only up to an additive constant (diverging logarithmically in 2D and linearly in 1D), reflecting the non-uniqueness of the harmonic Green’s function.
\label{rem:alpha_limit}
 \end{remark}

The Green's functions have the Fourier symbol $\widehat{G}^{\alpha}_{d}(\kappa)=(\kappa^2+\alpha^2)^{-1}$, where $\kappa=|\vkappa|$, $\vkappa \in \reals^d$, and will have a singularity at $\kappa=0$ when $\alpha=0$.
The solution to \eqref{eq:2DmodH} can be written as
\begin{equation}
    U(\vr)=\int_{\reals^2} G^{\alpha}_{2}(|\vr-\v y|) g(\v y) \dif\v y = 
    \frac{1}{(2\pi)^2} \int_{\reals^2}\widehat{G}^{\alpha}_{2}(|\vkappa|) \, \widehat{g}(\vkappa) e^{i \vkappa \cdot \vr} \dif\vkappa.
\label{eq:conv_integral}
\end{equation}
Now, assume that the right hand side $g(\v r)$ in
\eqref{eq:2DmodH} is compactly supported in $[0,\tilde{L}_1] \times [0,\tilde{L}_2]$, and that we want to evaluate the solution in that same domain. That makes $\Rtrunc=\sqrt{\tilde{L}_1^2+\tilde{L}_2^2}$ the maximum value of $|\v r -\v y|$ in \eqref{eq:conv_integral}. 
Hence, we would change nothing in the solution by replacing $G^{\alpha}_{2}(\vr)$ with 
\begin{equation}
   G^{\alpha,\Rtrunc}_{2}(r)=G^{\alpha}_{2}(r) \rect\left(\frac{r}{2\Rtrunc}\right), \quad 
   \rect(x)= \left\{ \begin{array}{cl}
   1, & |x| \le 1/2, \\
   0, & |x| > 1/2.
   \end{array} \right.
   \label{eq:Green2DRect}
\end{equation}
To find the Fourier transform, we use \eqref{eq:radial_ft_2D} as the kernel is radial, and get ($\kappa=(\kappa_2^2+\kappa_3^2)^{1/2}$)
\begin{equation}
\widehat{G}^{\alpha,\Rtrunc}_{2}(\kappa) = \left\{ 
\begin{array}{cl}
    \left(
      1 + \kappa \Rtrunc\, J_1(\kappa \Rtrunc)\, K_0(\alpha \Rtrunc)
        - \alpha \Rtrunc\, J_0(\kappa \Rtrunc)\, K_1(\alpha \Rtrunc)
    \right)/(\alpha^2 + \kappa^2),
 & \alpha \ne 0,\\
\left(1 - J_0(\Rtrunc \kappa)\right)/\kappa^2,
    & \alpha = 0, \kappa \ne 0, \\
 \Rtrunc^2/4, & \alpha = 0, \kappa = 0, 
\end{array}
\right. 
\label{eq:GreenR2D_Fourier}
\end{equation}
where the third case is obtained from the second, by taking the limit as $\kappa \rightarrow 0$. The formulas for $\alpha =0$ and $\alpha \ne 0$ can be found in
\cite{bagge_fast_2023} and \cite{palsson_spectrally_2019}\footnote{Note that there is a typo in \cite{palsson_spectrally_2019} in the end formula (32).}, respectively. Here $J_m$ denotes the Bessel function of the first kind and $K_m$ the modified Bessel function of the second kind, both of order $m$.

In the $1D$-case, we simply have $\Rtrunc=\tilde{L}_1$, and we find ($\kappa=|\kappa_3|$)
\begin{equation}
\widehat{G}^{\alpha,\Rtrunc}_{1}(\kappa) = \left\{ 
\begin{array}{cl}
    \left(
      1 - e^{-\alpha \Rtrunc}\cos(\kappa \Rtrunc)
      + \kappa\, e^{-\alpha \Rtrunc}\sin(\kappa \Rtrunc)/\alpha
    \right)/(\alpha^2 + \kappa^2),  & \alpha \ne 0,\\
\left(1-\Rtrunc\kappa\sin(\kappa\Rtrunc) - \cos(\kappa \Rtrunc)\right)/\kappa^2, & \alpha = 0, \kappa \ne 0, \\
-\Rtrunc^2/2, & \alpha = 0, \kappa = 0, 
\end{array}
\right. 
\label{eq:GreenR1D_Fourier}
\end{equation}

The Stokeslet and stresslet kernels of Stokes flow can be obtained by differentiation of the biharmonic Green's function. Truncated versions of these kernels with a Fourier transform bounded at the origin can be found by applying the same differentiation to the truncated biharmonic kernel, as was first done in \cite{AfKlinteberg2016fse}, in the 3D case. 
The truncated biharmonic kernels derived in \cite{Vico2016} for 2D/3D and used in \cite{AfKlinteberg2016fse} for 3D however decay slower in $\kappa$ than the original biharmonic kernels, and we will use the proposed alternate truncated biharmonic kernels derived in \cite[Appendix D]{bagge_fast_2023}. Here, the decay of the original kernels is recovered by exploiting the non-uniqueness of the biharmonic Green's function. The same idea was used for the harmonic kernel, i.e.~$\alpha=0$ in \eqref{eq:GreenR2D_Fourier}, where the stated formula is from \cite{bagge_fast_2023}, removing the $1/\kappa$ term involving $J_1$ that is present in the corresponding formula in \cite{Vico2016}.

In \cite{bagge_fast_2023}, truncated kernels are used in the Spectral Ewald method only for $\alpha=0$, i.e.~in cases where the discrete wave numbers in the periodic direction(s) are zero. As will be discussed in the next section, it is however advantageous to use truncated kernels also for small $\alpha \ne 0$. 
For the biharmonic equation, 
the equations corresponding to \eqref{eq:2DmodH}--\eqref{eq:1DmodH} become 
\begin{align}
\left( -\Delta_{2} +\alpha^2 \right)^2 U(\vr) & = g(\v r), \quad \vr \in \reals^2,
\label{eq:2DmodB}
\\
\left(
-\Delta_{1} 
+ \alpha^2 \right)^2 U(z) & = g(z), \quad z \in \reals, 
\label{eq:1DmodB}
\end{align}
where the right hand side in the equations that we will consider for the far field component (c.f.~\eqref{eq:2DmodH_k} and \eqref{eq:1DmodH_k}) will contain the biharmonic screening function instead of the harmonic screening function, see \cref{rem:biharm_screening}. The Green's functions have the Fourier symbols, 
$\widehat{G}^{\alpha}_{d}(\kappa)=(\kappa^2+\alpha^2)^{-2}$
, where $\kappa=|\vkappa|$, $\vkappa \in \reals^d$, and with an even stronger singularity at $\kappa=0$ when $\alpha=0$. 

In \cref{sec:trunc_biharmonic}, we derive Fourier transforms for the corresponding truncated kernels for $\alpha \ne 0$. The corresponding truncated Stokes kernels are derived in \cref{sec:trunc_Stokes}, where we state also the truncated Stokes kernels for $\alpha=0$ from \cite{bagge_fast_2023}, as well as the final integrals to be discretized.

\subsection{Discretization of Fourier integrals}\label{ss:discretization_fourier_integrals}
The integrals in Fourier space in the non-periodic directions that contain the smooth Fourier transforms of the truncated kernels can be discretized with the trapezoidal rule. In doing so, we obtain the same kind of discrete sums as in the periodic directions. The step size however needs to be set such that the approximation is accurate enough. For the case with $\alpha\ne 0$ (i.e.~$k_1\ne 0$ and $k_1^2+k_2^2\ne 0$ in the singly/doubly periodic cases), we also have the choice to discretize the original Fourier kernel, which can be more efficient for large enough discrete wave number(s). The choice can be made based on resolution requirements, as will be discussed at the end of this subsection. 

To approximate the sum and integrals in \eqref{eq:ewald_far_3P_es}--\eqref{eq:ewald_far_1P_es}, we need to truncate the infinite sum/domain, and will use 
$\kmaxx$, $\kmaxy$, and $\kmaxz$
as defined in \eqref{eq:kintmax}, where we however do not need to round up to the closest integer in the non-periodic directions. 

Given step sizes $\hk_j$, $j=2,3$,
define the set of discretization points
\begin{equation}
\begin{aligned}
   \freesetkP{1}(k_1,\hk_2,\hk_3)=
   \Bigl\{ 
  \Bigl(k_1, & n_2 \hk_2, n_3 \hk_3 \Bigr) : \\ 
  &   n_j=-N_j,\ldots,N_j \ \text{s.t.}
  \ (N_j-1)\hk_j \le \frac{2\pi}{L_j}\kmaxj \le N_j\hk_j, \ j=2,3,
\Bigr\}  
\end{aligned}
\label{eq:k1set-1p}
\end{equation}
and from this, 
\begin{align}
     \wavenumsettrunc^{1\per,0} & :=\freesetkP{1}(0,\hk_2,\hk_3),
     \label{eq:wavenumbers-1p-trunc-0}\\
  \wavenumsettrunc^{1\per,\ne0} & :=\bigcup_{\substack{\kintx=-\kmaxx \\ \kintx \ne 0}}^{\kmaxx} 
  \freesetkP{1}(2\pi \kint_1/L_1,\hk_2,\hk_3),
   \label{eq:wavenumbers-1p-trunc-ne0}
\\
  \wavenumsettrunc^{1\per} & := \wavenumsettrunc^{1\per,0} \, \bigcup \, \wavenumsettrunc^{1\per,\ne0},
  \label{eq:wavenumbers-1p-trunc}
\end{align}
and with the step size $\hk_3$, similarly
\begin{equation}
\begin{aligned}
   \freesetkP{2}(k_1,k_2,\hk_3)=
   \Bigl\{ 
  \Bigl(k_1, & k_2, n_3 \hk_3 \Bigr) : \\ 
  &   n_3=-N_3,\ldots,N_3 \ \text{s.t.}
  \ (N_3-1)\hk_3 \le \frac{2\pi}{L_3}\kmaxz \le N_3\hk_3
\Bigr\},
\end{aligned}
\label{eq:k1k2set-2p}
\end{equation}
and
\begin{align}
     \wavenumsettrunc^{2\per,0} & :=\freesetkP{2}(0,0,\hk_3),
     \label{eq:wavenumbers-2p-trunc-0}\\
\wavenumsettrunc^{2\per,\ne0} & := \bigcup_{\kint_1=-\kmaxx}^{\kmaxx}
\ \bigcup_{\substack{
\kint_2=-\kmaxy\\
(\kint_1,\kint_2)\ne(0,0)
}}^{\kmaxy}
\freesetkP{2}\!\left(
2\pi \kint_1/L_1,2\pi \kint_2/L_2,\hk_3
\right),
 \label{eq:wavenumbers-2p-trunc-ne0}
\\
  \wavenumsettrunc^{2\per} & := \wavenumsettrunc^{2\per,0} \, \bigcup \, \wavenumsettrunc^{2\per,\ne0}.
  \label{eq:wavenumbers-2p-trunc}
\end{align}

Using what was discussed in \cref{sec:screening}, below \eqref{eq:sum_of_screenH},
our general original mollified kernel (without truncation) can be written $\Kmollhat(\vk,\decpar)=\Khat(\vk) \screenhat(|\vk|\decpar)$, where $\screenhat=\screenhat^H$ or $\screenhat=\screenhat^B$. With truncation, the  $\Khat(\vk)$ factor will be modified, and we will introduce the generic notation $\genericGhat(\vk)$. 
Hence, the discretized versions of 
\eqref{eq:ewald_far_3P_es}--\eqref{eq:ewald_far_1P_es}, can for a general (possibly truncated) kernel $\genericGhat(\vk)$ and screening function $\screenhat$ be written 
 \begin{equation}
\ufarP{D}(\vx_\tind,\decpar)
= C^{D\per} \sum_{\sind=1}^N  \sum_{\vk\in\wavenumsettrunc^{D\per}} 
\genericGhat(\vk) \screenhat(|\vk|\decpar) 
e^{i\vk\cdot(\vx_\tind-\vx_\sind)}
  \rho_\sind,
\label{eq:ewald_far_allP_discr}
\end{equation}
where $\wavenumsettrunc^{3\per}$ is defined in \eqref{eq:wavenumbers-3p-trunc} in addition to $\wavenumsettrunc^{1\per}$ and $\wavenumsettrunc^{2\per}$ above, and 

 \begin{equation}
 \begin{aligned}
 C^{3\per}& = 1/(L_1 L_2 L_3), \\
C^{2\per}(\hk_3)&= \hk_3 /(L_1 L_2 \cdot 2 \pi),  \\
C^{1\per}(\hk_2,\hk_3)& = \hk_2 \hk_3/(L_1 \cdot (2 \pi)^2),
 \end{aligned}
 \end{equation}
where we have suppressed the argument of $C^{D\per}$ in \eqref{eq:ewald_far_allP_discr}. 

For electrostatics, in the doubly periodic case ($D=2)$, when using the truncated kernel also for $\sqrt{k_1^2+k_2^2}>0$, the specific formula for the discrete approximation of \eqref{eq:ewald_far_2P_es} is 
\begin{equation}
\ufarP{2}(\vx_\tind,\decpar) 
= C^{2\per}(\hk_3) \sum_{\sind=1}^N  \sum_{\substack{
\vk\in\wavenumsettrunc^{2\per}\\
\vk=(k_1,k_2,k_3)}}
\widehat{G}^{\sqrt{k_1^2+k_2^2},\Rtrunc}_{1}(|k_3|) \ \screenhat^H(|\vk| \decpar) \,  e^{i\vk\cdot(\vx_\tind-\vx_\sind)}
  \rho_\sind,
\label{eq:ewald_far_es_2P_discr_a}
\end{equation}
where $\widehat{G}^{\alpha,\Rtrunc}_{1}(\kappa)$ is defined in
\eqref{eq:GreenR1D_Fourier}\footnote{Note, $\screenhat^H=\widehat\wfunc$. The Fourier symbols for the Stokeslet and the stresslet will instead involve the biharmonic screening function 
$\screenhat^B$ in \eqref{eq:gammaB_in_phi}, $\screenhat^B\ne \widehat\wfunc$, see also 
\cref{rem:biharm_screening}.}.
In the singly periodic case, if only using the truncated kernel for $k_1=2\pi \kintx/L_1=0$, and the original harmonic kernel for the non-zero $k_1$ modes, we get 
\begin{equation}
\begin{aligned}
\ufarP{1}(\vx_\tind,\decpar) 
= \sum_{\sind=1}^N  
\ \Bigl( \, 
C^{1\per}(\hk_2^0,\hk_3^0) \sum_{\substack{
\vk\in\wavenumsettrunc^{1\per,0}\\
\vk=(k_1,k_2,k_3)}} &
\
\widehat{G}^{0,\Rtrunc}_{2}\left(\sqrt{k_1^2+k_2^2}\, \right) \ \screenhat^H(|\vk| \decpar) \,  e^{i\vk\cdot(\vx_\tind-\vx_\sind)} \rho_\sind 
\\ 
 &+C^{1\per}(\hk_2,\hk_3) \sum_{\substack{
\vk\in\wavenumsettrunc^{1\per,\ne 0} \\ \vk=(k_1,k_2,k_3)}}
\frac{\screenhat^H(|\vk| \decpar)}{|\vk|^2} \,  e^{i\vk\cdot(\vx_\tind-\vx_\sind)} \rho_\sind
 \Bigr),
\end{aligned}
\label{eq:ewald_far_es_1P_discr_b}
\end{equation}
where $\widehat{G}^{\alpha,\Rtrunc}_{2}(\kappa)$ is defined in
\eqref{eq:GreenR2D_Fourier}.
The $\hk_2^0,\hk_3^0$ as argument to $C^{1\per}$ in the first part is to indicate that different step sizes are typically needed for the two parts. 
The expressions for using the truncated kernel only for $k_1=k_2=0$ in the doubly periodic case, and for using the truncated kernel for all $k_1$ in the singly periodic case follows naturally. 

The truncation radius $\Rtrunc$ in the truncated Fourier symbols \eqref{eq:GreenR2D_Fourier}--\eqref{eq:GreenR1D_Fourier} as used in \eqref{eq:ewald_far_es_2P_discr_a}--\eqref{eq:ewald_far_es_1P_discr_b} should be chosen as small as possible, as frequency of oscillations increase with $\Rtrunc$. 
Hence, we set $\Rtrunc$ equal to the maximum value of $|\v r -\v y|$ in \eqref{eq:conv_integral}. 
The right hand side in \eqref{eq:2DmodH_k} has support in a box of size 
$[-\decpar,L_2+\decpar]\times [-\decpar,L_3+\decpar]$ given by the compact support and locations of the superimposed $\screen^H(|\x|/\decpar)$
in \eqref{eq:sum_of_screenH} in the free directions.
The target points in this plane are inside $[0,L_2]\times [0,L_3]$, and hence we pick $\Rtrunc=\sqrt{(L_2+\decpar)^2+(L_3+\decpar)^2}$. In the doubly periodic case, we have a one-dimensional problem and $\Rtrunc=L_3+\decpar$.
Compared to the original integrals in \eqref{eq:ewald_far_2P_es}--\eqref{eq:ewald_far_1P_es}, the ratio of the 
highest frequency to resolve for the truncated kernel relative to the original kernel is 
$\Rtrunc/L_j$ in the non-periodic direction $j$. 
A minimal resolution requirement will be determined by the Nyquist sampling criteria, and we will need 
$\hk_j<\pi/L_j$ and $\hk_j<\pi/\Rtrunc$, for the original and truncated kernels respectively.\footnote{ See \cite[\href{https://dlmf.nist.gov/10.17.E3}{Equation 10.17.3}]{NIST:DLMF}.}
However, a much smaller step size can be needed to resolve the factor $(\alpha^2+\kappa^2)^{-1}$ in the original mollified kernel. Here
\begin{equation}
\alpha = 
\begin{cases}
\sqrt{k_1^2+k_2^2}, & D=2,\\
|k_1|,& D=1,
\end{cases}
\label{eq:alpha}
\end{equation}
and the most restrictive nonzero mode is
\begin{equation}
\alphamin = 
\begin{cases}
2\pi/\max\{L_1,L_2\}, & D=2,\\
2\pi/L_1,& D=1.
\end{cases}
\label{eq:alphamin}
\end{equation}
\Cref{sec:error_estimates} gives trapezoidal-rule error estimates for these integrals. We evaluate these estimates at $\alpha=\alphamin$ with the step size set by the Nyquist criterion for the original kernel. If the estimated error exceeds the requested tolerance, then the truncated kernel will be used for all $\alpha\neq 0$.

\begin{remark}
In our current implementation, we use either the truncated or the original kernel for all $\alpha \ne 0$. 
We have not considered a switch from truncated to original once the discrete wave number is large enough.
The effect is expected to be quite small, and the computational cost of the far field $\ufarP{D}$ is in turn subdominant to the evaluation of $\ucompactP{D}$, as will be shown in the numerical results section.
\label{rem:alpha_switch}
\end{remark}

\subsection{Large aspect ratio boxes}
\label{sec:lessP-large-ar}

In \cref{sec:3P-large-ar}, we discussed the possible acceleration of the far-field contribution in the triply periodic case by the use of FFTs. In the doubly periodic case, with one periodic length much longer than the other, we can accelerate the evaluation in the long direction as discussed in what was called {\em Case 1}. There is however also a possible case with short periodic direction(s) and one or two long non-periodic (free) directions. In the free directions, there is no exact length that must be respected, the domain simply has to be large enough to enclose all source and target points, and there is no loss of generality to assume that the length is an integer multiple of $L_1$. 

Now, assume that $L_3/L_1$ is large, with $L_2/L_1$ small or moderate. Assuming $L_2=L_1$ for simplicity, this is like {\em Case 1} in \cref{sec:3P-large-ar}. The difference is the definition of the wave number sets $\wavenumsettrunc^{1\per}$  \eqref{eq:wavenumbers-1p-trunc}, and $\wavenumsettrunc^{2\per}$ \eqref{eq:wavenumbers-2p-trunc} where the free directions are sampled differently. 
To define the 2D discrete Fourier transform in \eqref{eq:1d-DFT-beta}, we used the fact that $k_3=2\pi\kint_3/L_3$ where $\kint_3$
is an integer and amenable for the rewrite 
$\kint_3=j \notopboxes+s$, $j,s\in \mathbb{Z}$ and $0\le s < \notopboxes$. Here, we instead need to set $\hk_3$ in order to resolve the integral. 
However, if we define\footnote{Remember, the FFT acceleration will only be done if $\notopboxes$ is large.}
\begin{equation}
\hk_3=2\pi /(L_3 q)=2\pi/(qML_1),
\label{eq:kappaqdef}
\end{equation}
with $q$ such that $\hk_3$ is small enough, and $q \notopboxes \in \mathbb{N}$, we can still use the FFT, albeit with some upsampling. 

We need to evaluate $\chi(\vk)$ for all $\k\in 
\wavenumsettrunc^{2\per}$, that we write on the form 
$\k=(2\pi \kint_1/L_1,2\pi \kint_2/L_2,n_3 \hk_3)$, where $\kint_1,\kint_2,n_3\in \mathbb{Z}$. 
Let $s = n_3 \bmod q\notopboxes$. Here is where we use $q\notopboxes\in\mathbb{N}$, so that this modulo operation is well-defined. Consider the $2\per$ case and define
\begin{equation}
\widehat{\beta}^{(q)}_s(j_3;\kint_1,\kint_2)
:=
\sum_{m=1}^{\notopboxes}
e^{-i2\pi s(m-1)/(q\notopboxes)}
\,\beta_m(j_3;\kint_1,\kint_2),
\qquad
s=0,\ldots,q\notopboxes-1.
\label{eq:beta_hat_oversampled}
\end{equation}
This is obtained by zero-padding
\(\{\beta_m(j_3;\kint_1,\kint_2)\}_{m=1}^{\notopboxes}\)
(as defined in \eqref{eq:beta_def}) to length \(q\notopboxes\) and applying a length-\(q\notopboxes\) FFT.
With this notation, \eqref{eq:structfac-case1-wDFTCoeffs} becomes
\begin{equation}
\chi(\vk)
=
C^{D\per}
e^{-i\pi\left(\kint_1+\kint_2+n_3/(q\notopboxes)\right)}
\sum_{j_3=1}^p
e^{-in_3 \hk_3 \eta_{j_3}}
\,\widehat{\beta}^{(q)}_{n_3 \bmod q\notopboxes}
(j_3;\kint_1,\kint_2).
\label{eq:structurefac_oversampled_fft}
\end{equation}
The similarity with the non-oversampled case is that the FFT is still taken only
over the cube index $m$. The difference is that the FFT length is enlarged
from $\notopboxes$ to $q\notopboxes$, and the Fourier index is now
$(n_3 \bmod q\notopboxes)$ instead of $(\kint_3 \bmod \notopboxes)$. A similar oversampling is used when using the IFFT as corresponding to \eqref{eq:ufar_inverse_fft}.
In the $1\per$ case, the second direction is also free, and we can no longer use $\kint_2$ in the expressions above. This does however not affect the FFT treatment in the long direction, and we do not explicitly write these formulas. 

If we have one short periodic direction and two long non-periodic ones, we will similarly modify {\em Case 2}, and use $2D$ FFT/IFFTs with upsampling. 

\section{Numerical results}\label{s:numerical_results}

We have implemented the DMK algorithm with singly, doubly and triply periodic boundary conditions for the Laplace and Stokes kernels in MATLAB. The code \cite{DMK_matlab} is intended for experimentation and algorithmic development, and has not been optimized for runtime. All experiments were run on an AMD EPYC 9655P 96-core CPU. Reported timings exclude construction of the adaptive octree(s) and, unless otherwise stated, are measured as the median over 10 repetitions. In our implementation, the subdivision parameter $n_s$, defined as the maximum number of points in a leaf box, is set to 250 in the reported runs. 
This parameter affects efficiency, but not accuracy.

The numerical experiments are organized as follows. 
We first describe the parameter-selection procedure and verify that it delivers the requested accuracy. We then assess the cost of periodization in two steps: first on the unit cube, where periodic DMK is compared with the corresponding free-space DMK evaluation, and then on rectangular cuboids, where we examine how the cost of the periodic part depends on the geometry of the primary cell. Together, these experiments quantify both the overhead introduced by the periodization scheme and its dependence on the cell geometry.
All experiments in this section use the Laplace kernel, except the final one, where we demonstrate the method for a singly periodic Stokeslet problem. Additional results for the Stokes kernels are provided in \cref{sec:stokes_results}.

In the experiments below the primary cell has side lengths $L=[L_1,L_2,L_3]$. We consider singly periodic problems, with periodicity in the first coordinate direction; doubly periodic problems, with periodicity in the first two coordinate directions; and triply periodic problems, with periodicity in all three coordinate directions. These cases are denoted by $D=1$, $D=2$, and $D=3$, respectively.

\subsection{Error control and parameter selection}\label{ss:error_control_param_selection}
The parameter choices in DMK are driven by the prescribed tolerance $\eeps$. We consider four such parameters: the Chebyshev order $p$, the Fourier cutoff $\Kmax$ associated with the kernel split, the difference-kernel Fourier grid size $\NF$, and the Fourier discretization sizes used in the far-field evaluation. The order $p$ controls the accuracy of the tensor-product Chebyshev proxy representation in the upward and downward passes, while $\Kmax$ determines the Fourier support needed to resolve the mollified and difference kernels. For PSWF-based kernel splits, $\Kmax=c$, where $c$ is the PSWF bandwidth introduced in \cref{sec:window_functions}. The quantity $\NF$ denotes the number of Fourier grid points per coordinate direction used for the difference-kernel translations, so that the tensor-product grid has $\NF^3$ points. The far-field Fourier discretization depends on the periodicity and on the geometry of the primary cell. Periodic directions lead to Fourier series, whereas free directions require quadrature of Fourier integrals; see \eqref{eq:ewald_far_3P_es}--\eqref{eq:ewald_far_1P_es}. We denote the corresponding mode and quadrature counts by $\Nper_i$ and $\Nfree_i$, respectively, and write $\NFP{D}$ for the total number of Fourier discretization points in the $D$-periodic far-field evaluation. For example, in the doubly periodic case with the first two directions periodic, $\NFP{2}=\Nper_1\Nper_2\Nfree_3$. This quantity should be distinguished from $\NF$, which denotes the fixed Fourier grid size used for the difference-kernel translations.

The parameters are selected as follows. 
First, $\Kmax$ and $p$ is chosen based on $\eeps$. 
For the PSWF split, this means choosing the bandlimit $c=\Kmax$ large enough that the truncation error of the screening function is below $\eeps$.
Given this kernel split, the order $p$ is then chosen so that the level-zero difference kernel $\kd_0$ is interpolated to accuracy $\eeps$.
For the Stokes kernels, we use the automatic parameter-selection procedure of \cite[Section 7]{StokesDMK2026}, which chooses $c$ and $p$ so that the relative $\ell_2$ error of DMK is below $\eeps$.
Once $\Kmax$ has been fixed, the difference-kernel Fourier grid is set by $\NF=2\lceil3\Kmax/\pi\rceil-1$, or equivalently $n_f=\lceil3\Kmax/\pi\rceil-1$. This differs from the rule in \cref{ss:eval_residual_difference} only by the omission of the outermost endpoint modes, which can be safely discarded for the PSWF-based split.
Finally, the far-field discretization is selected. In the triply periodic case, the retained modes are those in $\wavenumsettrunc^{3\per}$, defined in \eqref{eq:wavenumbers-3p-trunc}, with cutoffs chosen as in \eqref{eq:kintmax}, so that $\Nper_i=2\kint_i^{\mathrm{max}}+1$. For reduced periodicity, the periodic directions are treated in the same way, while the free directions are discretized by quadrature; see \cref{ss:discretization_fourier_integrals}. 
The zero mode in the periodic variables is always evaluated using the truncated Green's function. For nonzero periodic modes, we choose between the original and truncated Fourier symbols using the Nyquist criteria and the error estimates in \cref{sec:error_estimates}, as described at the end of \cref{ss:truncated_kernel_modification}. 
When a truncated symbol is used, the Fourier-grid spacing is set by the resolution requirement associated with the truncation radius $\Rtrunc$; see  \cref{ss:discretization_fourier_integrals}.

\cref{tab:param-table-laplace} lists the selected parameters for the Laplace kernel on the unit cube, and \cref{tab:param-table-stokes} shows the corresponding parameters used for the Stokes kernels. For fixed $\eeps$, the order $p$, bandwidth $c$, and difference-kernel grid size $\NF$ are independent of periodicity, while the far-field discretization changes with the number of free directions. In particular, each free direction replaces a periodic mode count $\Nper_i$ by a quadrature count $\Nfree_i$, which is larger because it comes from discretizing a Fourier integral. To illustrate the corresponding geometry dependence, \cref{fig:heatmap_fourier_discretization} shows how the total number of Fourier discretization points $\NFP{D}$ changes for rectangular cells $[1,L_2,L_3]$. As expected, $\NFP{D}$ generally increases with the side lengths and with the number of free directions. The visible jumps in \cref{fig:heatmap_fourier_discretization_1P} occur where the parameter selection switches from the original Fourier symbol to the truncated symbol for nonzero modes in the periodic variable.

We now verify the accuracy obtained with these parameter choices for a number of different test cases. \cref{tab:laplace_error} reports the results for the Laplace kernel, and \cref{tab:stokes-tolerance} gives the corresponding results for the Stokes kernels. Across the tested kernels, periodicities, and domains, the measured relative $\ell_2$ errors remain below the requested tolerance, typically by only a modest margin.

\begin{figure}[!t]
\centering
\begin{subfigure}[t]{0.33\textwidth}
\includegraphics[width=\linewidth]{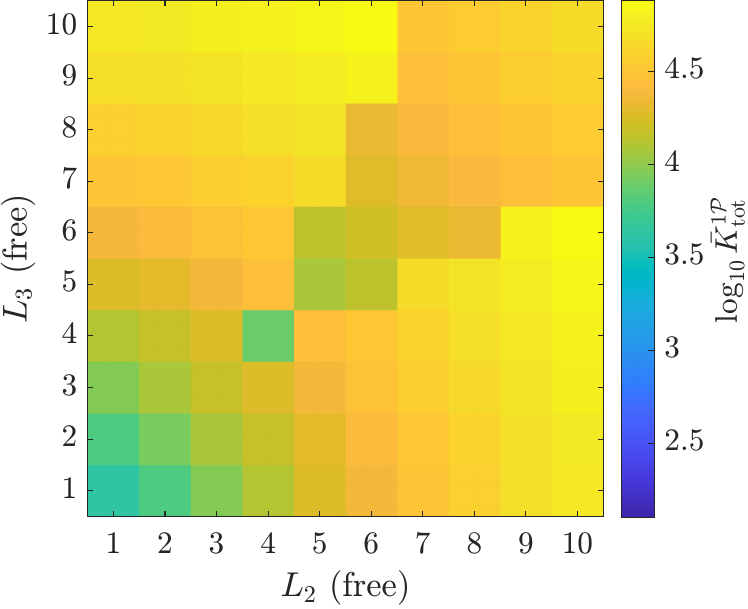}
\caption{$\NFP{1}=\Nper_1\Nfree_2\Nfree_3$}
\label{fig:heatmap_fourier_discretization_1P}
\end{subfigure}
\begin{subfigure}[t]{0.33\textwidth}
\includegraphics[width=\linewidth]{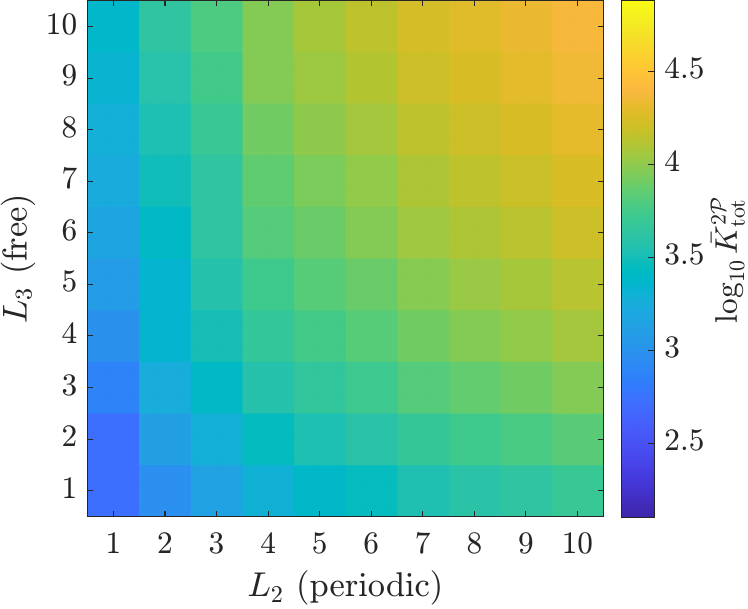}
\caption{$\NFP{2}=\Nper_1\Nper_2\Nfree_3$}
\label{fig:heatmap_fourier_discretization_2P}
\end{subfigure}%
\begin{subfigure}[t]{0.33\textwidth}
\includegraphics[width=\linewidth]{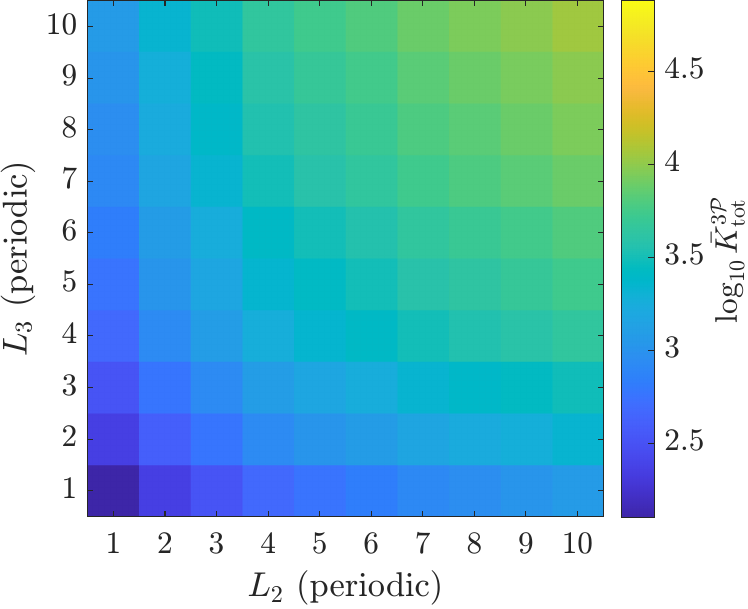}
\caption{$\NFP{3}=\Nper_1\Nper_2\Nper_3$}
\label{fig:heatmap_fourier_discretization_3P}
\end{subfigure}
\caption{Panels (a), (b), and (c) show the total number of Fourier discretization points $\NFP{D}$, plotted on a $\log_{10}$ scale, for the singly, doubly, and triply periodic Laplace kernels, respectively. The domain is $L=[1,L_2,L_3]$, with integer $L_2$ and $L_3$, and $\eeps=10^{-6}$.}
\label{fig:heatmap_fourier_discretization}
\end{figure}

\begin{table}[tb]
\centering
\caption{Parameters selected for the Laplace kernel on the unit cube for requested tolerance $\eeps$ and periodicity $D\per$.}
\label{tab:param-table-laplace}
\begin{tabular}{ll|llllllllll}
\hline
$D\per$ & $\varepsilon$ & $c$ & $p$ & $\NF$ & $\Nper_1$ & $\Nper_2$ & $\Nper_3$ & $\Nfree_1$ & $\Nfree_2$ & $\Nfree_3$ & $\Rtrunc$ \\
\hline
$3\per$ & $10^{-3}$ & 7.00 & 12 & 13 & 3 & 3 & 3 & -- & -- & -- & -- \\
$3\per$ & $10^{-6}$ & 14.50 & 20 & 27 & 5 & 5 & 5 & -- & -- & -- & -- \\
$3\per$ & $10^{-9}$ & 21.50 & 30 & 41 & 7 & 7 & 7 & -- & -- & -- & -- \\
$3\per$ & $10^{-12}$ & 29.00 & 38 & 55 & 9 & 9 & 9 & -- & -- & -- & -- \\
\hline
$2\per$ & $10^{-3}$ & 7.00 & 12 & 13 & 3 & 3 & -- & -- & -- & 7 & 2.00 \\
$2\per$ & $10^{-6}$ & 14.50 & 20 & 27 & 5 & 5 & -- & -- & -- & 21 & 2.00 \\
$2\per$ & $10^{-9}$ & 21.50 & 30 & 41 & 7 & 7 & -- & -- & -- & 29 & 2.00 \\
$2\per$ & $10^{-12}$ & 29.00 & 38 & 55 & 9 & 9 & -- & -- & -- & 39 & 2.00 \\
\hline
$1\per$ & $10^{-3}$ & 7.00 & 12 & 13 & 3 & -- & -- & -- & 7 & 7 & 2.83 \\
$1\per$ & $10^{-6}$ & 14.50 & 20 & 27 & 5 & -- & -- & -- & 29 & 29 & 2.83 \\
$1\per$ & $10^{-9}$ & 21.50 & 30 & 41 & 7 & -- & -- & -- & 41 & 41 & 2.83 \\
$1\per$ & $10^{-12}$ & 29.00 & 38 & 55 & 9 & -- & -- & -- & 55 & 55 & 2.83 \\
\hline
\end{tabular}
\end{table}

\begin{table}[tb]
\centering
\caption{Achieved relative $\ell_2$ errors for the Laplace kernel for different periodicities, domains, and requested tolerances $\eeps$. The total number of sources and targets is $N=1000$, with half of the points uniformly distributed in the domain, and the other half clustered near the center of one top-level cube.}
\label{tab:laplace_error}
\begin{tabular}{ll|llll}
\hline
$D\per$ & $[L_1,L_2,L_3]$ & $\varepsilon=10^{-3}$ & $\varepsilon=10^{-6}$ & $\varepsilon=10^{-9}$ & $\varepsilon=10^{-12}$ \\
\hline
3$\per$ & $[2.0,1.0,1.0]$ & $2.3{\times}10^{-4}$ & $1.0{\times}10^{-7}$ & $8.6{\times}10^{-11}$ & $8.8{\times}10^{-13}$ \\
3$\per$ & $[1.0,1.3,1.7]$ & $4.1{\times}10^{-4}$ & $2.0{\times}10^{-7}$ & $8.5{\times}10^{-11}$ & $8.9{\times}10^{-13}$ \\
\hline
2$\per$ & $[3.0,1.0,2.0]$ & $2.0{\times}10^{-3}$ & $1.6{\times}10^{-7}$ & $3.8{\times}10^{-10}$ & $1.7{\times}10^{-12}$ \\
2$\per$ & $[1.0,5.0,2.0]$ & $3.7{\times}10^{-4}$ & $6.8{\times}10^{-8}$ & $8.7{\times}10^{-11}$ & $4.9{\times}10^{-13}$ \\
\hline
1$\per$ & $[3.0,1.0,2.0]$ & $8.2{\times}10^{-5}$ & $5.9{\times}10^{-8}$ & $7.1{\times}10^{-11}$ & $1.1{\times}10^{-12}$ \\
1$\per$ & $[1.0,1.0,2.0]$ & $1.7{\times}10^{-4}$ & $1.4{\times}10^{-7}$ & $6.6{\times}10^{-11}$ & $8.8{\times}10^{-13}$ \\
\hline
\end{tabular}
\end{table}

\begin{figure}[!t]
\centering
\includegraphics[width=0.49\linewidth]{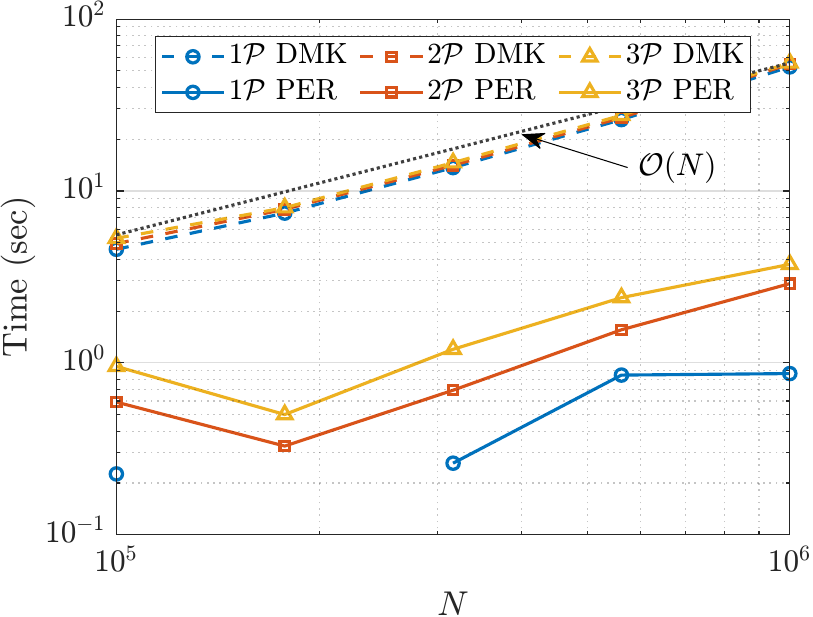}
\caption{Linear scaling of periodic DMK for the Laplace kernel on the unit cube. The figure shows the runtime of the full $D\per$ DMK evaluation and the difference between the full $D\per$ runtime and the corresponding free-space runtime, denoted  $D\per$ PER. The sources and targets are uniformly distributed, and $\eeps=10^{-6}$. The data point for $1\per$ PER and $N=177\,828$ is omitted because this case was faster than the corresponding free-space problem.}
\label{fig:cube_periodizing_scaling}
\end{figure}

\subsection{Cost of periodization and aspect ratio dependence}\label{ss:cost_of_periodization}
We first demonstrate the linear scaling of our new periodic DMK scheme on the unit cube. We consider $N$ uniformly randomly distributed sources (targets coincide with sources) and set $\eeps=10^{-6}$. \cref{fig:cube_periodizing_scaling} shows timing results for the Laplace kernel. Also plotted is the periodization overhead $D\per$ PER, which is the difference between the runtimes of the full periodic DMK evaluation, and the corresponding free-space DMK evaluation.
 We observe linear scaling in $N$ for both the full periodic method and $D\per$ PER, and note that the periodization overhead remains a small fraction of the full runtime.

We next examine the effect of rectangular geometry. For all examples in this subsection, the sources are uniformly randomly distributed.
We set $L=[1,L_2,L_3]$, with $L_2,\,L_3\in\{1,2,3,4,5\}$, and take $N=L_2L_3\Ncube$, where $\Ncube$ is the number of sources in the unit cube, i.e.~the point density is kept constant. We define the \emph{normalized runtime} by $(t(L)/N)/(t([1,1,1]/\Ncube)$, where $t(L)$ is the runtime on the cell $L$. Thus values close to one indicate a linear scaling with $N$, with little dependence on the rectangular geometry. \cref{fig:vary_L2_L3_scatter} shows the normalized runtime results for $\Ncube=10\,000$ and $\eeps=10^{-6}$. We observe that all tested configurations lie within $[0.9,1.1]$.

\begin{figure}[!t]
\centering
\includegraphics[width=0.49\linewidth]{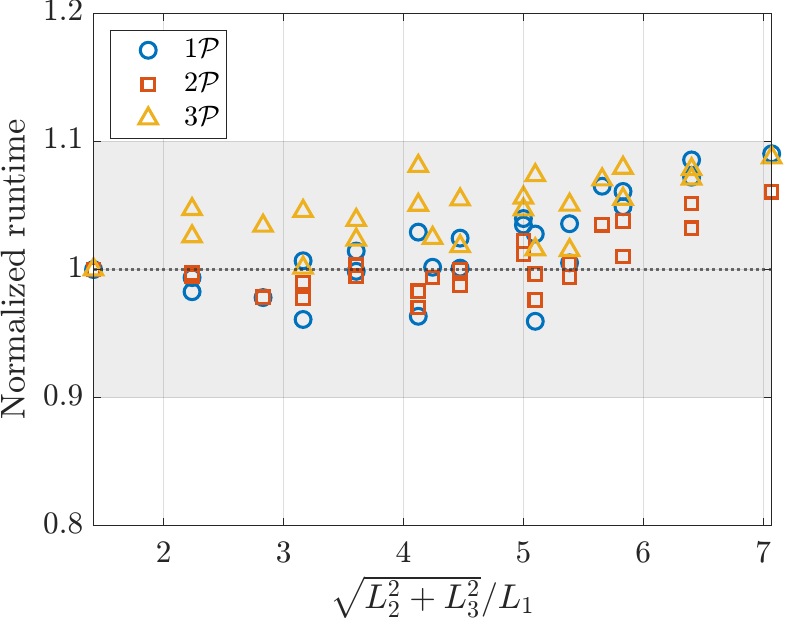}
\caption{Normalized runtime for rectangular cells $L=[1,L_2,L_3]$. 
The point density is fixed, with $\Ncube=10\,000$ point in the unit cube, and $\eeps=10^{-6}$.}
\label{fig:vary_L2_L3_scatter}
\end{figure}

We next examine triply periodic rectangular cells in more detail by separating the compact and far-field contributions, $\ucompactP{3}$ and $\ufarP{3}$, respectively, defined in \eqref{eq:u_four_partsDP}. We consider $L=[1,1,L_3]$, allow $L_3$ to be non-integer, and keep the point density fixed, so that $N=L_3\Ncube$. The non-integer case allows us to examine the domain extension and the reuse of proxy representations at different levels of the octree hierarchy, as described in  \cref{sec:3P}. \cref{fig:lap_3p_moderate_aspect_ratio} shows runtimes for $\Ncube=50\,000$ and $\eeps=10^{-6}$. From \cref{fig:lap_3p_moderate_aspect_ratio_timing} we can note that the computational cost for the compact part grows linearly with $L_3$, and hence with $N$, and that the far-field part is negligible in comparison. \cref{fig:lap_3p_moderate_aspect_ratio_overhead} shows the corresponding self-normalized runtimes, i.e.~with $\ucompactP{3}$ and $\ufarP{3}$ normalized by $\ucompactP{3}$ and $\ufarP{3}$ for the unit cube, respectively. The self-normalized runtime of the far-field component clearly displays the dependence on the aspect ratio, and the runtime decreases significantly when the top-level tiling is exact, i.e.~when $L_3$ is an integer. In this case the geometry is as in \cref{fig:top_level_proxy_uncut}, and the tensor-product structure of the proxy grid on the top-level of the octrees can be used to compute the structure factor \eqref{eq:structurefac}. The smaller dips, e.g.~at $L_3=2.5$ and $L_3=2.75$, have a similar origin, 
and occur because a full proxy point structure is available when using also boxes at finer level(s) of the tree. In the general case, as depicted in \cref{fig:top_level_proxy_cut}, additional original source points must be included in the evaluation. In either case, the computational cost for the far field component $\ufarP{3}$ constitutes only a small fraction of the total runtime. 

%\begin{figure}[!t]
\begin{figure}[p]
\centering
%\begin{subfigure}[t]{0.49\textwidth}
\begin{subfigure}[t]{0.33\textwidth}
\includegraphics[width=\linewidth]{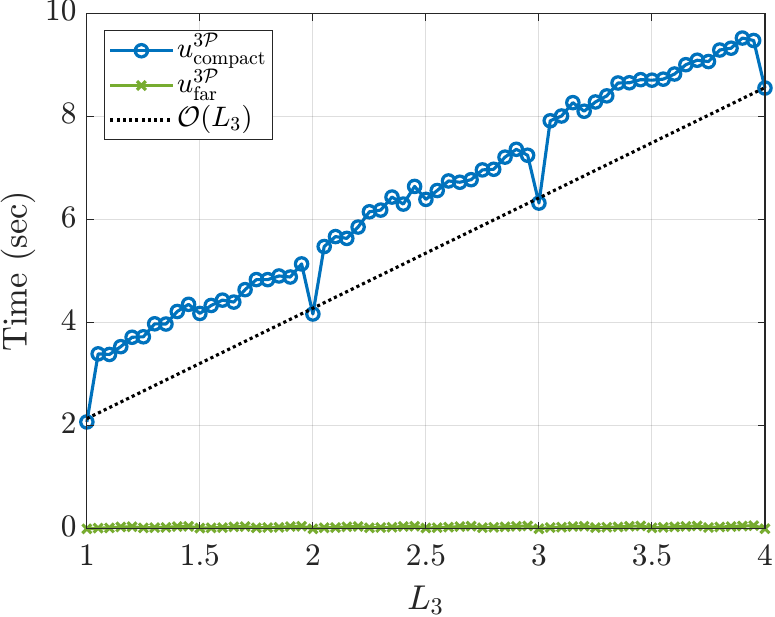}
\caption{Runtime vs $L_3$.}
\label{fig:lap_3p_moderate_aspect_ratio_timing}
\end{subfigure}
%\begin{subfigure}[t]{0.49\textwidth}
\begin{subfigure}[t]{0.33\textwidth}
\includegraphics[width=\linewidth]{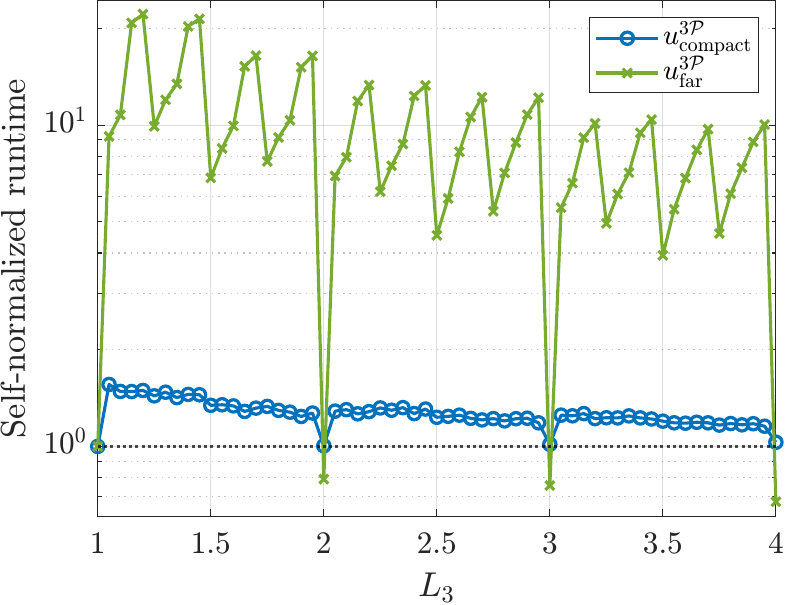}
\caption{Self-normalized runtime vs $L_3$.}
\label{fig:lap_3p_moderate_aspect_ratio_overhead}
\end{subfigure}
\vspace{-1em}
\caption{Runtimes and self-normalized runtimes of the compact and far-field contributions, $\ucompactP{3}$ and $\ufarP{3}$, for the triply periodic Laplace kernel on $L=[1,1,L_3]$.
The point density is fixed as $N=\Ncube L_3$, with $\Ncube=50\,000$, $\eeps=10^{-6}$.}
\label{fig:lap_3p_moderate_aspect_ratio}
\end{figure}

For large aspect ratios, however, the far-field contribution can become significant unless the root evaluation is accelerated. We use the same triply periodic setup as in \cref{fig:lap_3p_moderate_aspect_ratio}, but restrict $L_3$ to integer values and take $\Ncube=1000$. Since the tiling is then exact, $\ufarP{3}$ can be evaluated using the tensor-product proxy-grid structure. This is now what we refer to as the direct evaluation. We also evaluate $\ufarP{3}$ using the FFT-accelerated scheme described in  \cref{ss:large_aspect_ratio}. \cref{fig:lap_3p_large_aspect_ratio} shows the results. 
The computational time for the direct evaluation is reduced by using the tensor structure, but is expected to have an asymptotic quadratic growth in $L_3$, see \eqref{eq:compcompl-sumfact-1long} ($L_3=\notopboxes$ as $L_1=L_2=1$). 
This becomes visible as $L_3$ grows, while the FFT-accelerated evaluation has only approximately a linear increase (expected to be $L_3 \log L_3$, see \eqref{eq:compcompl-fft-acc-1long}).
This is even clearer in the self-normalized runtimes in \cref{fig:lap_3p_large_aspect_ratio_overhead}. The direct evaluation incurs increasing aspect-ratio overhead, while the FFT-accelerated evaluation stays constant. In terms of runtime, the crossover between the two approaches in our implementation occurs at an aspect ratio around $250$.

%\begin{figure}[!t]
\begin{figure}[p]
\centering
%\begin{subfigure}[t]{0.49\textwidth}
\begin{subfigure}[t]{0.33\textwidth}
\includegraphics[width=\linewidth]{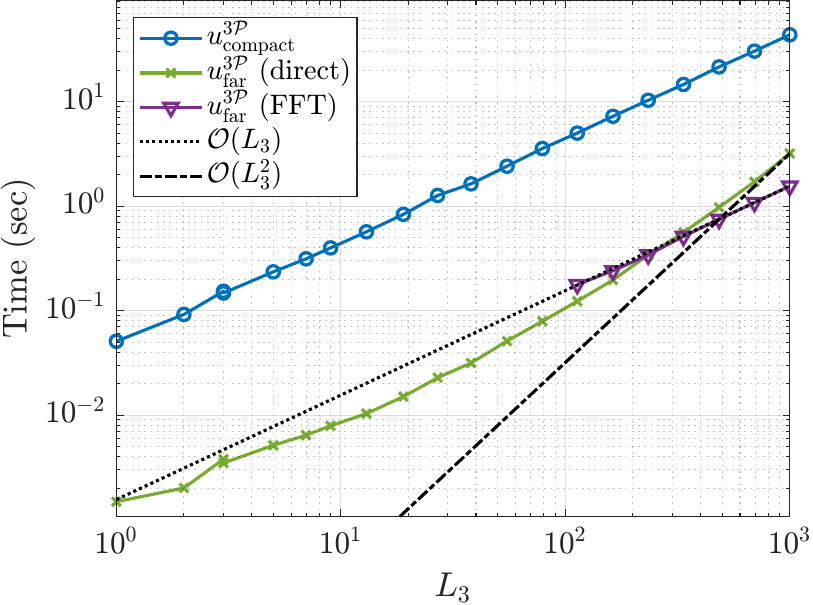}
\caption{Runtime vs $L_3$.}
\label{fig:lap_3p_large_aspect_ratio_timing}
\end{subfigure}
%\begin{subfigure}[t]{0.49\textwidth}
\begin{subfigure}[t]{0.33\textwidth}
\includegraphics[width=\linewidth]{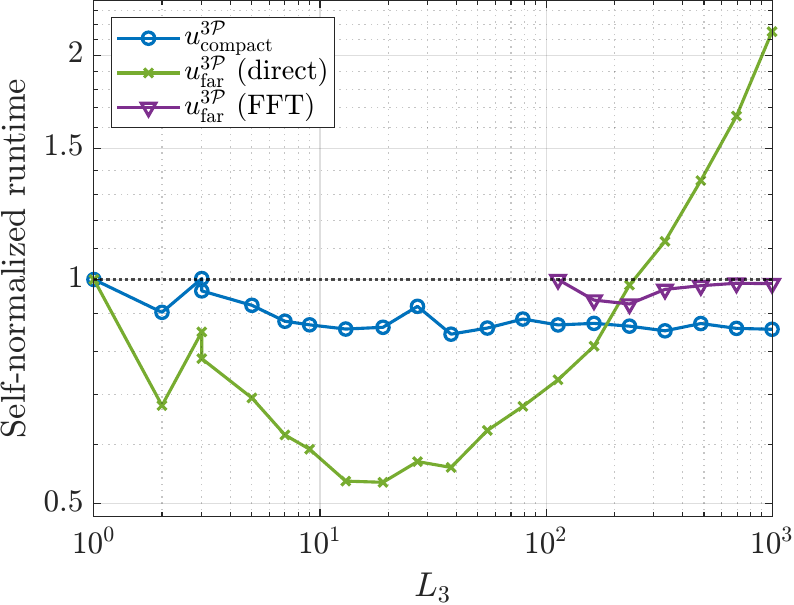}
\caption{Self-normalized runtime vs $L_3$.}
\label{fig:lap_3p_large_aspect_ratio_overhead}
\end{subfigure}
\vspace{-1em}
\caption{Large-aspect-ratio version of \cref{fig:lap_3p_moderate_aspect_ratio}, with integer $L_3$ and $\Ncube=1000$. The far-field contribution $\ufarP{3}$ is evaluated both directly using the tensor-product proxy-grid structure and with the FFT-acceleration. Timings are medians over 10 runs, except for the unit cube, where 100 runs were used. In panel (b), the FFT-accelerated $\ufarP{3}$ timings are normalized by the value at $L_3=113$, rather than by the unit-cube value.}
\label{fig:lap_3p_large_aspect_ratio}
\end{figure}

\subsection{Singly periodic Stokeslet cylinder}\label{ss:cylinder_stokeslet}
Finally, we demonstrate the periodic DMK method for a singly periodic pipe geometry representative of Stokes flow applications. Sources and targets are placed at trapezoidal nodes on the cylindrical surface, with periodicity in the axial direction, as illustrated in \cref{fig:cylinder_illustration}. The point density is kept fixed as the period length $L_1$ is increased. \cref{fig:cylinder_time_vs_L1} shows the resulting runtimes for the Stokeslet kernel for $\eeps=10^{-6},\,10^{-9},\,10^{-12}$. The observed scaling is linear in the number of points, with an expected increase in prefactor as the requested tolerance is tightened.

%\begin{figure}[!t]
\begin{figure}[p]
\centering
%\begin{subfigure}[!t]{0.53\textwidth}
\begin{subfigure}[t]{0.45\textwidth}
\includegraphics[trim={3.8cm 4.2cm 4.0cm 4.2cm},clip,width=\linewidth]{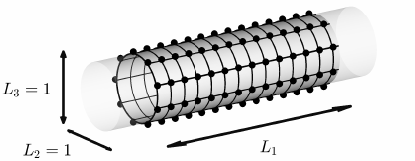}
\caption{Sketch of geometry}
\label{fig:cylinder_illustration}
\end{subfigure}
%\begin{subfigure}[!t]{0.46\textwidth}
\begin{subfigure}[t]{0.33\textwidth}
\includegraphics[width=\linewidth]{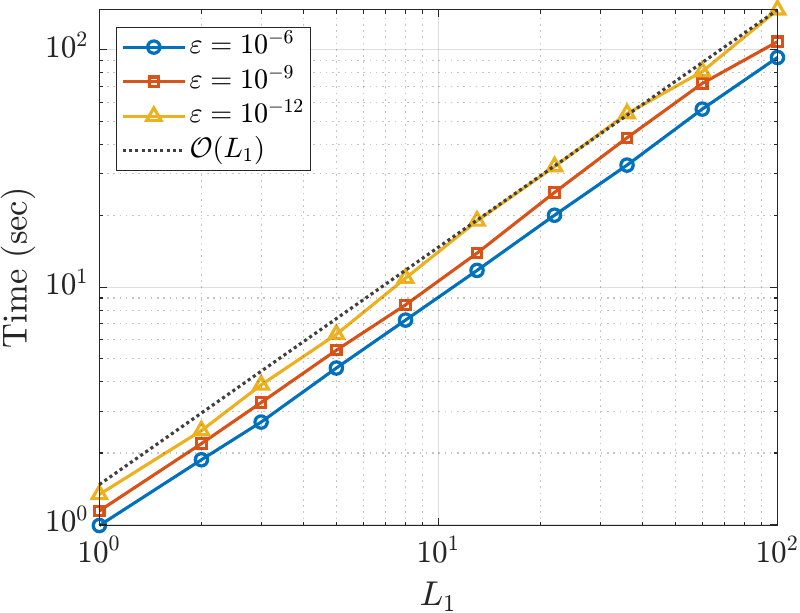}
\caption{Runtime}
\label{fig:cylinder_time_vs_L1}
\end{subfigure}
\vspace{-1em}
\caption{Panel (a) illustrates the singly periodic pipe geometry in the cell $L=[L_1,1,1]$, with sources and targets placed at trapezoidal nodes on the cylindrical surface. Panel (b) shows the runtime of $1P$-periodic DMK for the Stokeslet kernel 
with $N=\Ncube L_1$, $\Ncube=10\,000$. Timings are shown for $\eeps=10^{-6},\,10^{-9},\,10^{-12}$.}
\label{fig:cylinder_stokeslet}
\end{figure}

\clearpage

\section{Conclusions}\label{s:conclusions}

We have developed an extension of the dual-space multilevel
kernel-splitting (DMK) framework to rectangular cuboids with periodic
boundary conditions in one, two, or three coordinate directions.  The
method exploits the multilevel structure of DMK, where a telescoping
kernel split localizes all interactions below the root level, so that
the global contribution of the periodic lattice is confined to the
smooth root-level far field. The localized interactions are evaluated
on a cubical tiling of the primary cell, together with periodic
wrapping of near-neighbor interactions and, when the tiling is not
exact, a bounded domain extension populated by periodic source copies.

Periodicity is imposed on the root-level far field, which is evaluated
in Fourier space. For reduced periodicity, truncated kernels are used
to regularize the singular and near-singular modes that arise in the
non-periodic directions. This allows the Fourier integrals to be
discretized using the uniform trapezoidal rule, yielding a common
structure for all periodicities. For large-aspect-ratio cuboids with
exact tilings, the root-level summation over cubes is accelerated
using FFTs. The corresponding geometry-dependent cost is then $O(M
\log M)$, where $M$ is the number of cubes tiling the domain.

Numerical experiments for the electrostatic and Stokes kernels show
that the periodization adds only a small cost to the free-space DMK
algorithm. For each periodicity, the total cost on the tested cuboids
remains close to that of the corresponding cubic
reference. Specifically, the cost breakdown shows that the runtime for
periodic problems is dominated by the localized corrections, whose
per-source cost changes little with aspect ratio. The far-field
contribution is more geometry dependent, but its cost remains only a
fraction of the localized part and therefore has little effect on the
total cost. In the large-aspect-ratio tests, FFT acceleration of the
far-field contribution keeps the per-source cost approximately
independent of aspect ratio at constant point density. 
The results also confirm linear scaling
with the number of sources and targets, and show that the method
achieves the prescribed precision for the tested tolerances, kernels,
periodicities, and cuboid geometries.

The method provides a unified way of imposing periodic boundary
conditions in DMK, with only minor changes to the multilevel structure
of the original algorithm. We have focused on the harmonic and Stokes
kernels, but the proposed periodic extension is expected to apply
without significant modification to other non-oscillatory kernels for
which a kernel split is available. The implementation used for the
present work is intended for algorithmic development rather than
optimized performance, and incorporating the present extension scheme
into a high-performance DMK code is left for future work. Finally,
although only point-source potentials are considered here, the
periodization strategy should also be applicable to volume-potential
DMK implementations.

\section*{Acknowledgments}
Krantz and Tornberg acknowledge support from the Swedish Research Council under grant 2023-04269.

\appendix
\crefalias{section}{appendix} 

\section{Poisson's summation formula and Fourier transforms}
\label{sec:fourier-defs}
Considering periodic sums, with $P_{D\per}$ as defined in
\eqref{eq:periodic-images}, Poisson's summation formula yields the following relations
\begin{alignat}{2}
\sum_{\v p \in P_{3\per}} f(\x+\v p) & = \frac{1}{L_1L_2L_3} \sum_{\vk} \wh{f}(\vk) e^{i \vk
  \cdot \x}, \quad &  &
\wh{f}(\vk) = \int_{\reals^3} f(\x) e^{-i \vk \cdot \x} \dif\x,
\notag \\
\sum_{\v p \in P_{2\per}} f(\x+\v p) & = \frac{1}{L_1L_2} \sum_{\kbar} \wh{f}(\kbar,z) e^{i \kbar \cdot \v r},  \quad & &
\wh{f}(\kbar,z) = \int_{\reals^2} f(\x) e^{-i \kbar \cdot \v r} \dif\v r, 
\label{eqn:poisson_summation}
\\
\sum_{\v p \in P_{1\per}} f(\x+\v p) & = \frac{1}{L_1} \sum_{k_1} \wh{f}(k_1,\v r) e^{i k_1 z}, \quad & & \wh{f}(k_1,\v r) =\int_{\reals} f(\x) e^{-i k_1 z} \dif z.
\notag 
\end{alignat}
Here, the $\wh{f}$ is the continuous Fourier transform, applied in
the periodic direction(s), and 
$\x=(\v r,z)=(x,y,z)$, $\vk=(k_1,k_2,k_3)$ and $\kbar=(k_1,k_2)$. 

For $\vx, \vk \in \mathbb R^d$, without the notational splits above, the
Fourier transform pair $f(\vx),~\widehat{f}(\vk)$ is (under suitable conditions) defined by
\begin{align}
  \widehat f(\v k) &= \int_{\mathbb R^d} f(\v x) e^{-i \vk\cdot\vx} \dif \v x,\label{eq:app_fhat} \\
  f(\v x) &= \frac{1}{(2\pi)^d} \int_{\mathbb R^d} \widehat f(\vk) e^{i \vk\cdot\vx} \dif \vk.\label{eq:app_f} 
\end{align}
When the context is clear, we will often use the notation $r=|\v x|$ and $k=|\vk|$. For a
radially symmetric function in $\mathbb R^d$ we write
$f(\v x) = f(r)$ and $\widehat f(\vk) = \widehat f(k)$, with Fourier
transform pair
\begin{align}
  \widehat f(k) = 2\pi \int_0^\infty J_0(kr) f(r) r \dif r \quad &\leftrightarrow \quad
  f(r) = \frac{1}{2\pi} \int_0^\infty J_0(kr) \widehat f(k) k \dif k 
\quad {\rm in} \ \ \mathbb R^2 ,
\label{eq:radial_ft_2D}\\
  \widehat f(k) = 4\pi \int_0^\infty \frac{\sin(kr)}{kr} f(r) r^2 \dif r \quad &\leftrightarrow
 \quad f(r) = \frac{1}{2\pi^2} \int_0^\infty \frac{\sin(kr)}{kr} \widehat f(k) k^2 \dif k 
\quad {\rm in} \ \ \mathbb R^3,
\label{eq:radial_ft_3D}
\end{align}
where $J_0$ is the Bessel function of the first kind of order zero.

\section{Definition and splits of the Stokes kernels}\label{sec:stokes-split}
The tensorial Stokeslet~$\bmat{\stokeslet}$, stresslet~$\bmat{\stresslet}$ and rotlet~$\rotlett$ kernels can be defined through differentiation of the scalar biharmonic kernel $B(r)$ and harmonic kernel $H(r)$. With
\begin{equation}
H(r)=\frac{1}{4\pi r},\qquad B(r)=-\frac{r}{8\pi},
\label{eq:app_HB}
\end{equation}
we write ($\vx \in \mathbb{R}^3$, $r=|\vx|$):
\begin{align}
\stokeslet_{jl}(\vx) &= \Dop^{\stokeslet}_{jl} \biharmonic(|\vx|)=(-\delta_{jl}\nabla^2 + \nabla_j\nabla_l) \biharmonic(|\vx|)
=\frac{1}{8\pi} \left(\frac{\delta_{jl}}{r} + \frac{x_jx_l}{r^3}\right), 
\label{eq:stokeslet_def_deriv} \\
\stresslet_{jlm}(\vx) & = \Dop^{\stresslet}_{jlm} \biharmonic(|\vx|)=
\left[ -\left(     \delta_{jl}\nabla_m+\delta_{lm}\nabla_j+\delta_{mj}\nabla_l
    \right) \nabla^2 + 2\nabla_j\nabla_l\nabla_m \right] \biharmonic(|\vx|)=-\frac{3}{4\pi} \frac{x_jx_lx_m}{r^5}, \label{eq:stresslet_def_deriv} \\
\rotlet_{jl}(\vx) &=\Dop^{\rotlet}_{jl} \harmonic(|\vx|) =
  - \frac{1}{2} \epsilon_{jlm} \nabla_m \harmonic(|\vx|)
  = \frac{1}{8\pi}\epsilon_{jlm} \frac{x_m}{r^3}.
  \label{eq:rotlet_def_deriv} 
\end{align}
Here, $\delta_{jl}$ denotes the Kronecker delta and
$\epsilon_{jlm}$ denotes the Levi-Civita symbol. Einstein's
summation convention is used, with repeated indices implicitly
summed over $\{1,2,3\}$.

The Fourier symbols of the differential operators above are given by ($k=\lvert \vk \rvert$)
\begin{align}
  \label{eq:diffop-rel-fourier-S}
  \wh{\Dop}^{\stokeslet}_{jl}(\vk)
  &= k^2 \delta_{jl} - k_j k_l
  , \\
  \label{eq:diffop-rel-fourier-T}
  \wh{\Dop}^{\stresslet}_{jlm}(\vk)
  &=
  i\left[(k_m\delta_{jl} + k_l\delta_{mj} + k_j\delta_{lm}) k^2 - 2k_j k_l k_m\right]
  , \\
 \label{eq:diffop-rel-fourier-R}
  \wh{\Dop}^{\rotlet}_{jl}(\vk)
  &= -\frac{1}{2}i\epsilon_{jlm}k_m
  .
\end{align}

In \cref{sec:screening}, we introduced the screening function $\screen^H$ for the harmonic kernel. As noted in \cref{rem:biharm_screening}, the biharmonic requires a different screening. We here state the results obtained in \cite{StokesDMK2026}.
Consider the window function $\wfunc(x)$, an even function with $\wfunchat(0)=1$, and let $\Phi(r)$ be defined as in \eqref{eq:def_erfc_like_fcn}. Define the biharmonic mollifier
\begin{equation}
\screenhat^B(k)=\wfunchat(k) - \frac{1}{2} k \wfunchat'(k),
\label{eq:gammaB_in_phi}
\end{equation}
where $\wfunchat(k)$ is the Fourier transform of $\wfunc(x)$. 
From this, define the mollified biharmonic in Fourier space 
\begin{equation}
\Bmollhat(k)=\frac{\screenhat^B(k)}{k^4}.
\label{eq:Bmollhat}
\end{equation}
For the harmonic, we have $\Hmollhat(k)=\screenhat^H(k)/k^2$, where $\screenhat^H(k)=\wfunchat(k)$.
The rate of decay of all the functions defined below will depend on the choice of $\wfunc(r)$. 

Using the differential operators in \eqref{eq:stokeslet_def_deriv}--\eqref{eq:rotlet_def_deriv} and the Fourier symbols in  \eqref{eq:diffop-rel-fourier-S}--\eqref{eq:diffop-rel-fourier-R}, the following results can be derived \cite{StokesDMK2026}.
For the Stokeslet we have the {\em mollified Stokeslet} in Fourier space:
\begin{align}
    \widehat S_{M,jl}(\v k) = \left(k^2 \delta_{jl} - k_j k_l\right) \Bmollhat(k),
  \label{eq:moll_stokeslet}
\end{align}
the {\em residual Stokeslet},
\begin{align}
    S_{R,jl}(\vx) 
    &=   \RSdiag(r) \frac{\delta_{jl}}{8\pi r}
      +   \RSoffd(r)  \frac{x_jx_l}{8\pi r^3},
      \label{eq:res_stokeslet}
\end{align}
where 
\begin{equation}
  \RSdiag(r) =
               \Phi(r) - 2 r\wfunc(r),  \qquad
  \RSoffd(r) =
               \Phi(r) + 2 r\wfunc(r),
               \label{eq:Soffdiag}
\end{equation}
and the self-interaction term 
\begin{align}
\uself(\vx_\tind) = -\lim_{|\v x|\to 0} \bmat{S}_{M}(\v x) \rho_\tind=
-\frac{1}{2\pi} \wfunc(0) \rho_\tind.
\label{eq:S_self}
\end{align}

For the stresslet we have, 
the {\em mollified stresslet} in Fourier space: 
\begin{align}
    \widehat T_{M,jlm}(\v k) &=
 i \left[(k_m \delta_{jl} +k_j \delta_{lm} +k_l  \delta_{mj}) k^2   -2
  k_j k_l k_m \right]  
                             \Bmollhat(k), 
                             \label{eq:moll_stresslet}
\end{align}
and the {\em residual stresslet}, 
\begin{align}
  T_{R,jlm}(\vx) = \frac{\delta_{jl}x_m + \delta_{lm}x_j + \delta_{mj}x_l}{8\pi r^3} \, \RTdiag(r)
  - \frac{3}{4\pi}\frac{x_jx_lx_m}{r^5} \, \RToffd(r),
  \label{eq:res_stresslet}
\end{align}
where 
\begin{equation}
  \RTdiag(r) = 
               -2 r^2 \wfunc'(r) ,
\quad                
  \RToffd(r) =
               \Phi(r) + 2r \wfunc(r)  - \tfrac{2}{3} r^2 \wfunc'(r).
               \label{eq:T-offd}
\end{equation}
The self-interaction term for the stresslet vanishes: 
$\uself(\vx_\tind) = 0$.

The {\em mollified rotlet} in Fourier space is given by
\begin{align}
  \widehat \Omega_{M,jl}(\v k) = - \frac{1}{2}i\epsilon_{jlm} k_m \widehat H_M(k),
  \label{eq:moll_rotlet}
\end{align}
and the {\em residual rotlet} 
is given by
\begin{align}
  \begin{split}
    \Omega_{R,jl}(\vx) = 
    \frac{1}{8\pi} \epsilon_{jlm} \frac{x_m}{r^3} \ROoffd(r),
  \end{split}
  \label{eq:res_rotlet}
\end{align}
with
\begin{align}
  \ROoffd(r) =\Phi(r) + 2r\wfunc(r).
  \label{eq:ROoffd}
\end{align}
The self-interaction term for the rotlet vanishes: $\uself(\vx_\tind) = 0$.

\begin{remark}\label{rem:scaling_stokes_split}
  For a split at length scale $\nu>0$, the real-space functions $\RSdiag$, $\RSoffd$, $\RTdiag$, $\RToffd$, $\ROoffd$ are evaluated at $r/\nu$, while the Fourier-space screening functions $\screenHhat$, $\screenBhat$ are evaluated at $k\nu$. For example,
  \begin{align}
    S_{R,jl}(\vx, \decpar) = \RSdiag(r/\decpar) \frac{\delta_{jl}}{8\pi r}
    +   \RSoffd(r/\decpar)  \frac{x_jx_l}{8\pi r^3},
    \quad
    \text{and}
    \quad
    \widehat S_{M,jl}(\v k) = \left(k^2 \delta_{jl} - k_j k_l\right) \Bmollhat(k \decpar).
  \end{align}
  The Stokeslet self-interaction term \eqref{eq:S_self} scales as $1/\nu$,
  \begin{align}
    \uself(\vx_\tind,\decpar) =
    -\frac{1}{2\pi\decpar} \wfunc(0) \rho_\tind.
  \end{align}
  The stresslet and rotlet self-interaction terms remain zero. The aforementioned real-space functions are smooth and, in the implementation, are evaluated using precomputed polynomial approximations for the chosen window function $\wfunc$.
\end{remark}

\section{Truncated kernels related to the biharmonic equation}\label{sec:trunc_biharmonic} \Cref{ss:truncated_kernel_modification} described the truncated kernel modification for the harmonic kernels that that arise in the 1$\per$ and 2$\per$ Fourier integrals after applying Poisson summation in the periodic directions. We now collect the corresponding formulas for the biharmonic kernel. In the PDE viewpoint of \cref{ss:pde_perspective}, these kernels are the Green's functions of the modified Helmholtz and squared modified Helmholtz problems in \eqref{eq:2DmodH}--\eqref{eq:1DmodH} and \eqref{eq:2DmodB}--\eqref{eq:1DmodB}, respectively. The formulas below are needed for the Stokeslet and stresslet as these kernels are generated from the biharmonic kernel (see \cref{sec:trunc_Stokes} for their construction).

We use the notation of \cref{ss:truncated_kernel_modification}. Thus $d$ denotes the number of free directions, $\alpha\geq0$ is the magnitude of the discrete periodic wavenumber determined by \eqref{eq:alpha}, and $\kappa=|\vkappa|$ with $\vk\in\R^d$. For each fixed periodic mode, the harmonic problem in the free variables has Fourier symbol $\wh{H}_d^\alpha(\kappa)=(\alpha^2+\kappa^2)^{-1}$. The corresponding biharmonic problem in the free variables has the Fourier symbol  $\wh{B}_d^\alpha(\kappa)=(\alpha^2+\kappa^2)^{-2}$, associated with the squared modified Helmholtz operator.

For $\alpha=0$, these symbols are singular at $\kappa=0$. For small nonzero values of $\alpha$, they are smooth but sharply peaked near the origin. As in the harmonic case, we want to truncate the corresponding biharmonic kernel in real space, so that the Fourier transform of the truncated Green's function has a finite limit value at $\kappa=0$. For $\alpha=0$, we use the truncated kernel choices from \cite[Appendix D]{bagge_fast_2023}. For $\alpha>0$, these formulas are obtained from the harmonic ones in \eqref{eq:GreenR2D_Fourier}--\eqref{eq:GreenR1D_Fourier} by differentiating with respect to $\alpha$. This relation is summarized in the following lemma. The same identities holds for the truncated kernels, since the truncation radius $\Rtrunc$ is independent of $\alpha$.

\begin{lemma}\label{lemma:biharmonic_from_harmonic}
Let $\mathcal{L}_\alpha=-\Delta_d+\alpha^2$ with $\alpha>0$. Let $H_d^\alpha$ and $B_d^\alpha$ be the Green's functions associated with $\mathcal{L}_\alpha$ and $\mathcal{L}_\alpha^2$, respectively, so that $\mathcal{L}_\alpha H_d^\alpha(\vx)=\delta(\vx)$ and $\mathcal{L}_\alpha^2 B_d^\alpha(\vx)=\delta(\vx)$, $\vx\in\mathbb{R}^d$, where $\delta$ is the Dirac delta function.

Then
\begin{equation}
B_d^\alpha(\vx) = -\frac{1}{2\alpha}\frac{\partial}{\partial\alpha}H_d^\alpha(\vx).
\label{eq:app_B}
\end{equation}
Consequently,
\begin{equation}
\wh{B}_d^\alpha(\kappa) = -\frac{1}{2\alpha}\frac{\partial}{\partial\alpha}\wh{H}_d^\alpha(\kappa).
\label{eq:app_Bhat}
\end{equation}
\end{lemma}

\begin{proof}
Differentiating $\mathcal{L}_\alpha H_d^\alpha=\delta$ with respect to $\alpha$, and using $\partial_\alpha\mathcal{L}_\alpha=2\alpha\mathcal{I}$, gives $\mathcal{L}_\alpha(-\partial_\alpha(H_d^\alpha)/(2\alpha))=H_d^\alpha$. Applying $\mathcal{L}_\alpha$ once more gives \eqref{eq:app_B}. The identity \eqref{eq:app_Bhat} follows directly from \eqref{eq:app_B} by taking the Fourier transform.
\end{proof}

We now list the Fourier transforms of truncated scalar kernels. Let $B_d^{\alpha,\Rtrunc}$ denote the biharmonic kernel, truncated in the free variables to $|\vx|\leq\Rtrunc$, analogously to \eqref{eq:Green2DRect}. The Fourier transform of $B_d^{\alpha,\Rtrunc}$ for $d=2$ and $d=1$ is:
\begin{equation}
\widehat{B}^{\alpha,\Rtrunc}_{2}(\kappa) = \left\{ 
\begin{array}{cl}
\Big(1-\Rtrunc J_0(\kappa\Rtrunc) \big(\Rtrunc(\alpha^2+\kappa^2) K_0(\alpha\Rtrunc)+2 \alpha K_1(\alpha\Rtrunc)\big)/2\\
+\kappa \Rtrunc J_1(\kappa \Rtrunc) \big(\Rtrunc(\alpha^2+\kappa^2) K_1(\alpha\Rtrunc)/(2\alpha)+K_0(\alpha \Rtrunc)\big)\Big)/(\alpha^2+\kappa^2)^2,&\alpha\neq 0,\\
\left(1-J_0(\kappa\Rtrunc)-\kappa\Rtrunc J_1(\kappa\Rtrunc)/2\right)/\kappa^4,&\alpha=0, \kappa\neq 0, \\
\Rtrunc^4/64,&\alpha=0, \kappa=0,
\end{array}
\right.
\label{eq:BGreenR2D_Fourier}
\end{equation}
\begin{equation}
\widehat{B}^{\alpha,\Rtrunc}_{1}(\kappa) = \left\{ 
\begin{array}{cl}
\Big(2 \alpha^3+\kappa e^{-\alpha\Rtrunc} \big(\alpha \Rtrunc\left(\alpha^2+\kappa^2\right)+3 \alpha^2+\kappa^2\big) \sin (\kappa\Rtrunc) \\
-\alpha^2e^{-\alpha\Rtrunc} \big(\Rtrunc\left(\alpha^2+\kappa^2\right)+2 \alpha\big) \cos (\kappa\Rtrunc)\Big)/\big(2\alpha^3(\alpha^2+\kappa^2)^2\big),&\alpha\neq 0,\\
\left(1-\cos(\kappa\Rtrunc)-\kappa\Rtrunc\sin(\kappa\Rtrunc)/2\right)/\kappa^4,&\alpha=0, \kappa\neq 0, \\
\Rtrunc^4/24,&\alpha=0, \kappa=0.
\end{array}
\right.
\label{eq:BGreenR1D_Fourier}
\end{equation}
Here $J_m$ denotes the Bessel function of the first kind and $K_m$ the modified Bessel function of the second kind, both of order $m$. For the zero-mode $\widehat{B}^{0,\Rtrunc}_{1}$ when $\kappa\neq 0$, we have used the non-uniqueness of the one-dimensional biharmonic Green's function to choose an even kernel that vanishes with $\mathcal{C}^1$-regularity at the cut-off. 
This improves the decay of the Fourier transform. The limiting values at $\kappa=0$ are obtained by taking the limit $\kappa\rightarrow0$ in the corresponding $\kappa\neq 0$ formulas. 

\begin{remark}
The formulas for $\wh{B}^{0,\Rtrunc}_{1}$ are included in \eqref{eq:BGreenR1D_Fourier} for completeness, since it is not used in practice for the doubly periodic Stokes zero mode. As discussed in \cite[Appendix D]{bagge_fast_2023}, the biharmonic kernel can be avoided in this case. Instead, the zero-mode Stokes kernels can be written using the one-dimensional harmonic kernel together with the sign-function kernel
\begin{equation}
Z_1(x) = \frac{\sgn(x)}{4} \quad\leftrightarrow\quad \wh{Z}_1(\kappa)= - \frac{i}{2\kappa}.
\end{equation}
Truncating $Z_1$ to $|x|\leq\Rtrunc$ gives
\begin{equation}
\widehat{Z}^{0,\Rtrunc}_{1}(\kappa) = \left\{ 
\begin{array}{cl}
-i\left(1-\cos(\kappa\Rtrunc)\right)/(2\kappa),&\kappa\neq 0,\\
0,&\kappa=0.
\end{array}
\right.
\label{eq:ZGreenR1D_Fourier}
\end{equation}
\end{remark}

\section{Truncated kernels for Stokes flow in reduced periodicity}\label{sec:trunc_Stokes}The Stokeslet $\bmat{\stokeslet}$, stresslet $\bmat{\stresslet}$ and rotlet $\rotlett$ were defined in \cref{sec:stokes-split} by applying differential operators to the scalar harmonic kernel $H$ or biharmonic kernel $B$. The same construction is used for the truncated kernels, but with $H$ and $B$ replaced, in the free variables, by their truncated counterpart. Here we collect the Fourier transforms of the truncated Stokes kernels needed for the Fourier integrals in the singly and doubly periodic cases. In the discretized far-field formula \eqref{eq:ewald_far_allP_discr}, these kernels replace the generic Fourier kernel $\genericGhatopt$. The zero-mode formulas are those derived in \cite[Appendix D]{bagge_fast_2023}, adapted here to the normalization of \cref{sec:stokes-split}.

We write each kernel in the form
\begin{equation}
\genericGopt(\vr)=\Dopt_{\genericGopt}A(r),\qquad A(r) = 
\begin{cases}
B(r), & \genericGopt=\bmat{\stokeslet}\ \text{or}\ \bmat{\stresslet},\\
H(r), & \genericGopt=\rotlett,
\end{cases}
\end{equation}
where $\Dopt_{\genericGopt}$ is defined by one of the differential operators in \eqref{eq:stokeslet_def_deriv}--\eqref{eq:rotlet_def_deriv} depending on $\genericGopt$, and with $H$ and $B$ defined in \eqref{eq:app_HB}.

Let $D\in\{1,2\}$ be the number of periodic directions and let $d=3-D$ be the number of free directions. Discrete wavenumbers in periodic directions are denoted $k_i$, and continuous wavenumbers in free directions are denoted $\kappa_i$. For each fixed periodic mode, set $\alpha=|\vk|$, $\kappa=|\vkappa|$ with $\k=(k_1,\vkappa)=(k_1,\kappa_2,\kappa_3)$, $\vkappa\in\mathbb{R}^d$, and $k=\sqrt{\alpha^2+\kappa^2}$. The truncated scalar factors are radial in the free Fourier variables, and hence depend on $\vkappa$ only through $\kappa$.

Define the scalar truncated kernel symbol by
\begin{equation}
\wh{A}_d^{\alpha,\Rtrunc}(\kappa)=
\begin{cases}
\wh B_d^{\alpha,\Rtrunc}(\kappa),
& \genericGopt=\bmat{\stokeslet}\ \text{or}\ \bmat{\stresslet},\\
\wh H_d^{\alpha,\Rtrunc}(\kappa),
& \genericGopt=\rotlett,
\end{cases}
\label{eq:truncated_stokes_kernel_general}
\end{equation}
where $\wh H_d^{\alpha,\Rtrunc}$ and $\wh B_d^{\alpha,\Rtrunc}$ are given in \eqref{eq:GreenR2D_Fourier}--\eqref{eq:GreenR1D_Fourier} and \eqref{eq:BGreenR2D_Fourier}--\eqref{eq:BGreenR1D_Fourier}, respectively. 

For the singly periodic case, $D=1$, we write $\vk=(k_1,\kappa_2,\kappa_3)$, $\alpha=|k_1|$, and $\kappa=\sqrt{\kappa_2^2+\kappa_3^2}$. The truncated Stokes kernel is defined by
\begin{equation}
\genericGhatopt_{\Rtrunc}^{1\per}(k_1,\kappa_2,\kappa_3) \coloneqq
\begin{cases}
\wh{\Dopt}_{\genericGopt}(k_1,\kappa_2,\kappa_3) \wh{A}_2^{|k_1|,\Rtrunc}(\kappa), &
k_1 \neq 0, \\[5pt]
\bmat{C}_{\genericGopt}^H(\kappa_2,\kappa_3)
\wh{\harmonic}_2^{0,\Rtrunc}(\kappa)
+
\bmat{C}_{\genericGopt}^B(\kappa_2, \kappa_3)
\wh{\biharmonic}_2^{0,\Rtrunc}(\kappa)
, &
k_1 = 0.
\end{cases}
\end{equation}
Here $\wh{\Dopt}_{\genericGopt}$ is the Fourier symbol associated with $\genericGopt$, as defined from \eqref{eq:diffop-rel-fourier-S}--\eqref{eq:diffop-rel-fourier-R}.

For the Stokeslet, the coefficient matrices are
\begin{equation}
\bmat{C}_{\bmat{\stokeslet}}^H =
\begin{bmatrix}
1&0&0\\
0&0&0\\
0&0&0
\end{bmatrix},\qquad
\bmat{C}_{\bmat{\stokeslet}}^B =
\begin{bmatrix}
0&0&0\\
0&\kappa_3^2&-\kappa_2\kappa_3\\
0&-\kappa_2\kappa_3&\kappa_2^2
\end{bmatrix},
\end{equation}

For the rotlet,
\begin{equation}
\bmat{C}_{\rotlett}^H = -\frac{i}{2}
\begin{bmatrix}
0 & \kappa_3 & -\kappa_2\\
-\kappa_3 & 0 & 0\\
\kappa_2 & 0 & 0
\end{bmatrix},\qquad
\bmat{C}_{\rotlett}^B = \bmat{0}.
\end{equation}

For the stresslet, the coefficients are third-order tensors. We represent them by matrices in the last two indices. That is, for each fixed first index $j$ we get $(\bmat{C}_{\bmat{\stresslet},j}^H)_{lm}$ and $(\bmat{C}_{\bmat{\stresslet},j}^B)_{lm}$, with $\bmat{C}_{\bmat{\stresslet},j}^H$ and $\bmat{C}_{\bmat{\stresslet},j}^B$ being symmetric $3\times 3$ matrices. Write
\begin{equation}
\bmat{C}_{\bmat{\stresslet},j}^H
=
i\,\widetilde{\bmat{C}}_{\bmat{\stresslet},j}^H,
\qquad
\bmat{C}_{\bmat{\stresslet},j}^B
=
i\,\widetilde{\bmat{C}}_{\bmat{\stresslet},j}^B,
\end{equation}
The harmonic parts are
\begin{equation}
\widetilde{\bmat{C}}_{\bmat{\stresslet},1}^H=
\begin{bmatrix}
0 & \kappa_2 & \kappa_3\\
\kappa_2 & 0 & 0\\
\kappa_3 & 0 & 0
\end{bmatrix},\qquad
\widetilde{\bmat{C}}_{\bmat{\stresslet},2}^H=
\begin{bmatrix}
\kappa_2 & 0 & 0\\
0        & 0 & 0\\
0        & 0 & 0
\end{bmatrix},\qquad
\widetilde{\bmat{C}}_{\bmat{\stresslet},3}^H=
\begin{bmatrix}
\kappa_3 & 0 & 0\\
0        & 0 & 0\\
0        & 0 & 0
\end{bmatrix}.
\end{equation}
For the biharmonic parts, $\bmat{C}_{\bmat{\stresslet},1}^B=\bmat{0}$, and
\begin{equation}
\widetilde{\bmat{C}}_{\bmat{\stresslet},2}^B=
\begin{bmatrix}
0 & 0 & 0\\
0 & \kappa_2(3\kappa^2-2\kappa_2^2)
  & \kappa_3(\kappa^2-2\kappa_2^2)\\
0 & \kappa_3(\kappa^2-2\kappa_2^2)
  & \kappa_2(\kappa^2-2\kappa_3^2)
\end{bmatrix},\qquad
\widetilde{\bmat{C}}_{\bmat{\stresslet},3}^B=
\begin{bmatrix}
0 & 0 & 0\\
0 & \kappa_3(\kappa^2-2\kappa_2^2)
  & \kappa_2(\kappa^2-2\kappa_3^2)\\
0 & \kappa_2(\kappa^2-2\kappa_3^2)
  & \kappa_3(3\kappa^2-2\kappa_3^2)
\end{bmatrix}.
\end{equation}

For the doubly periodic case, $D=2$, we write $\vk=(k_1,k_2,\kappa_3)$, $\alpha=\sqrt{k_1^2+k_2^2}$, $\kappa=|\kappa_3|$, and define
\begin{equation}
\genericGhatopt_{\Rtrunc}^{2\per}(k_1,k_2,\kappa_3) \coloneqq
\begin{cases}
\wh{\Dopt}_{\genericGopt}(k_1,k_2,\kappa_3) \wh{A}_1^{\alpha,\Rtrunc}(\kappa), &
(k_1,k_2) \neq (0,0), \\[5pt]
\bmat{C}_{\genericGopt}^E(\kappa_3)
\wh{E}_1^{0,\Rtrunc}(\kappa_3), &
(k_1,k_2) = (0,0).
\end{cases}
\end{equation}
Here the scalar zero-mode factor is
\begin{equation}
\wh{E}_1^{0,\Rtrunc}(\kappa_3)=
\begin{cases}
\wh Z_1^{0,\Rtrunc}(\kappa_3),
& \genericGopt=\rotlett\ \text{or}\ \bmat{\stresslet},\\
\wh H_1^{0,\Rtrunc}(\kappa_3),
& \genericGopt=\bmat{\stokeslet}.
\end{cases}
\end{equation}
The truncated sign-function kernel $\wh{Z}_1^{0,\Rtrunc}$ is given in \eqref{eq:ZGreenR1D_Fourier}. Since this kernel is odd, it is evaluated at the signed wavenumber $\kappa_3$. The harmonic factor $\wh{H}_1^{0,\Rtrunc}$ is even and is therefore written as a function of $\kappa=|\kappa_3|$. 

The coefficient matrices for the Stokeslet and rotlet are
\begin{equation}
\bmat{C}_{\bmat{\stokeslet}}^E=
\begin{bmatrix}
1&0&0\\
0&1&0\\
0&0&0
\end{bmatrix},
\qquad
\bmat{C}_{\rotlett}^E=
\begin{bmatrix}
0&1&0\\
-1&0&0\\
0&0&0
\end{bmatrix}.
\label{eq:C_stokeslet_rotlet_E_2p}
\end{equation}
For the stresslet, the coefficient tensor is represented, as in the singly periodic case, by three matrices in the last two indices:
\begin{equation}
\bmat{C}_{\bmat{\stresslet},1}^E
=
\begin{bmatrix}
0&0&-2\\
0&0&0\\
-2&0&0
\end{bmatrix},\qquad
\bmat{C}_{\bmat{\stresslet},2}^E
=
\begin{bmatrix}
0&0&0\\
0&0&-2\\
0&-2&0
\end{bmatrix},\qquad
\bmat{C}_{\bmat{\stresslet},3}^E
=
\begin{bmatrix}
-2&0&0\\
0&-2&0\\
0&0&-2
\end{bmatrix}.
\label{eq:C_stresslet_E_2p}
\end{equation}

\begin{remark}[Constant correction for the singly periodic Stokeslet]\label{rem:stokeslet_const_corr}The gauge choices used in the truncated singly periodic zero-mode kernels introduce an additive constant in the Stokeslet far-field contribution. The corresponding correction is zero for the stresslet and rotlet, and no such correction is needed for any of the doubly periodic zero-mode Stokes kernels \cite[Appendix D]{bagge_fast_2023}. The correction added to the singly periodic far-field term in \eqref{eq:ewald_far_1P_es} is
\begin{equation}
\bmat{u}_{\Rtrunc}^{1\per,\text{corr}}=
-\frac{1}{L_1}
\begin{bmatrix}
\dfrac{1+\log \Rtrunc}{2\pi} & 0 & 0\\[1ex]
0 & \dfrac{\log \Rtrunc+1/2}{4\pi} & 0\\[1ex]
0 & 0 & \dfrac{\log \Rtrunc+1/2}{4\pi}
\end{bmatrix}
\sum_{\alpha=1}^N \boldsymbol{\rho}_\alpha.
\end{equation}
\end{remark}

\section{Estimates for Fourier integrals in reduced periodicities}
\label{sec:error_estimates}
The far-field contribution in reduced periodicity, \eqref{eq:ewald_far_2P_es} for the $2\per$ case and \eqref{eq:ewald_far_1P_es} for the $1\per$ case, contain Fourier integrals over the free directions. Their discretization was described in  \cref{ss:discretization_fourier_integrals} and this appendix gives estimates of the errors introduced.  

We consider the one-dimensional model integral, representative of \eqref{eq:ewald_far_2P_es} and \eqref{eq:ewald_far_1P_es},
\begin{equation}
I_m(\alpha,b) = \int_{\mathbb R}\frac{\screenhat\left(\decpar\sqrt{\alpha^2+\kappa^2}\right)}{(\alpha^2+\kappa^2)^m}e^{i\kappa b}\,\dif\kappa,\qquad \alpha>0,\qquad m=1,2,
\label{eq:app_model_integral}
\end{equation}
where $b$ is the source-to-target separation in the free direction. We write $s=|b|$ with $0\leq s \leq L_j$, where $L_j$ is the domain length in the corresponding free direction. The screening function $\screenhat$ is either the harmonic screening $\screenHhat$ or the biharmonic screening $\screenBhat$, defined in \eqref{eq:wfunchat} and \eqref{eq:gammaB_in_phi}, respectively, and $\decpar>0$ is the splitting parameter. The case $m=1$ corresponds to the harmonic Green's function, while $m=2$ is associated with the biharmonic Green's function relevant to kernels of Stokes flow (cf.~\eqref{eq:Hhatsplit_wfunc} and \eqref{eq:Bmollhat}). The infinite trapezoidal rule with spacing $\hk$ is
\begin{equation}
Q_{m,\hk}(\alpha,b) = \hk\sum_{n=-\infty}^{\infty}\frac{\screenhat\left(\decpar\sqrt{\alpha^2+(n\hk)^2}\right)}{\left(\alpha^2+(n\hk)^2\right)^m}e^{in\hk b},
\label{eq:app_model_trapz}
\end{equation}
with the associated quadrature error
\begin{equation}
\mathcal{E}_m(\alpha,b;\hk) = \big|I_m(\alpha,b)-Q_{m,\hk}(\alpha,b)\big|.
\end{equation}

Assuming that $\screenhat$ is analytic near the origin and normalized by $\screenhat(0)=1$, we estimate the error, following the residue calculus in \cite[Theorem 2]{saffar_shamshirgar_spectral_2017}, by
\begin{align}
\mathcal{E}_1(\alpha,b;\hk)
&\approx
\frac{2\pi}{\alpha}
\frac{\cosh(\alpha s)}{E}, \label{eq:app_m1_error_estimate} \\
\mathcal{E}_2(\alpha,b;\hk)
&\approx
\frac{\pi}{\alpha^3}
\bigg|
\frac{
    \cosh(\alpha s)
    -
    \alpha(s-\alpha C_{\screenhat})\sinh(\alpha s)
}{E}+
\frac{
    2\pi\alpha (E+1)\cosh(\alpha s)
}{
    \hk E^2
}
\bigg|,
\label{eq:app_m2_error_estimate}
\end{align}
where
\begin{equation}
E = e^{2\pi\alpha/\hk}-1, \qquad C_{\screenhat} = \decpar^2\screenhat''(0).
\label{eq:app_E_Cgamma}
\end{equation}
The estimate for $m=1$ is the analogue of the estimate in \cite[Appendix B]{shamshirgar_fast_2021}, here written for a more general screening function. The estimate for $m=2$ is obtained by the same residue calculus for the corresponding double-pole case.

\begin{remark}
The estimates are written in terms of $s=|b|$, since the model integral and the symmetric infinite trapezoidal rule are even functions of the separation
$b$. For $m=1$, \eqref{eq:app_m1_error_estimate} is monotone in $s$. For $m=2$, monotonicity is less clear, but numerical maximization over $s\in[0,L_j]$ indicates that the largest value is attained at the endpoint. Thus, in both cases we use $s=L_j$ for the corresponding free direction.
\end{remark}

\begin{remark}
The constant $C_{\screenhat}$ depends only on the screening function. For the harmonic and biharmonic screening functions we find $C_{\screenHhat}=\decpar^2\wfunchat''(0)$ and $C_{\screenBhat}=0$, where $\wfunchat''(0)$ can be evaluated as described in \cref{rem:PSWF_eval}.
\end{remark}

In the $2\per$ case there is one free Fourier direction. For $(k_1,k_2)\neq(0,0)$, we use
\begin{equation}
\mathcal E_m^{2\per}(k_1,k_2;\hk_3)=\mathcal E_m\!\left(\sqrt{k_1^2+k_2^2},L_3;\hk_3\right),\qquad m=1,2.
\label{eq:app_2P_error_estimate}
\end{equation}
In the $1\per$ case there are two free directions, resulting in a two-dimensional integral in $(\kappa_2,\kappa_3)$. We estimate this two-dimensional quadrature error by the sum of two one-dimensional estimates, one for each free direction:
\begin{equation}
\mathcal E_m^{1\per}(k_1;\hk_2,\hk_3)=\mathcal E_m(|k_1|,L_2;\hk_2)+\mathcal E_m(|k_1|,L_3;\hk_3),\qquad k_1\ne0.
\label{eq:app_1P_error_estimate}
\end{equation}

\section{Numerical results for Stokes kernels}\label{sec:stokes_results}Here we give the corresponding Stokes-kernel results for the parameter-selection and accuracy tests in  \cref{ss:error_control_param_selection}. \cref{tab:param-table-stokes} reports parameters selected on the unit cube for the requested tolerances. \cref{tab:stokes-tolerance} reports the achieved relative $\ell_2$ errors for a number of different periodicities and domain geometries.

\begin{table}[tb]
\centering
\caption{Parameters selected for the Stokes kernels on the unit cube for requested tolerance $\eeps$ and periodicity $D\per$.}
\label{tab:param-table-stokes}
\begin{tabular}{ll|llllllllll}
\multicolumn{12}{c}{\textbf{Stokeslet}} \\
\hline
$D\per$ & $\varepsilon$ & $c$ & $p$ & $\NF$ & $\Nper_1$ & $\Nper_2$ & $\Nper_3$ & $\Nfree_1$ & $\Nfree_2$ & $\Nfree_3$ & $\Rtrunc$ \\
\hline
$3\per$ & $10^{-3}$ & 10.47 & 13 & 19 & 3 & 3 & 3 & -- & -- & -- & -- \\
$3\per$ & $10^{-6}$ & 17.80 & 23 & 33 & 5 & 5 & 5 & -- & -- & -- & -- \\
$3\per$ & $10^{-9}$ & 25.13 & 34 & 47 & 7 & 7 & 7 & -- & -- & -- & -- \\
$3\per$ & $10^{-12}$ & 32.46 & 44 & 61 & 11 & 11 & 11 & -- & -- & -- & -- \\
\hline
$2\per$ & $10^{-3}$ & 10.47 & 13 & 19 & 3 & 3 & -- & -- & -- & 9 & 2.00 \\
$2\per$ & $10^{-6}$ & 17.80 & 23 & 33 & 5 & 5 & -- & -- & -- & 25 & 2.00 \\
$2\per$ & $10^{-9}$ & 25.13 & 34 & 47 & 7 & 7 & -- & -- & -- & 35 & 2.00 \\
$2\per$ & $10^{-12}$ & 32.46 & 44 & 61 & 11 & 11 & -- & -- & -- & 43 & 2.00 \\
\hline
$1\per$ & $10^{-3}$ & 10.47 & 13 & 19 & 3 & -- & -- & -- & 9 & 9 & 2.83 \\
$1\per$ & $10^{-6}$ & 17.80 & 23 & 33 & 5 & -- & -- & -- & 35 & 35 & 2.83 \\
$1\per$ & $10^{-9}$ & 25.13 & 34 & 47 & 7 & -- & -- & -- & 47 & 47 & 2.83 \\
$1\per$ & $10^{-12}$ & 32.46 & 44 & 61 & 11 & -- & -- & -- & 61 & 61 & 2.83 \\
\hline
\end{tabular}
\vspace{1em}

\begin{tabular}{ll|llllllllll}
\multicolumn{12}{c}{\textbf{Stresslet}} \\
\hline
$D\per$ & $\varepsilon$ & $c$ & $p$ & $\NF$ & $\Nper_1$ & $\Nper_2$ & $\Nper_3$ & $\Nfree_1$ & $\Nfree_2$ & $\Nfree_3$ & $\Rtrunc$ \\
\hline
$3\per$ & $10^{-3}$ & 10.47 & 12 & 19 & 3 & 3 & 3 & -- & -- & -- & -- \\
$3\per$ & $10^{-6}$ & 17.80 & 23 & 33 & 5 & 5 & 5 & -- & -- & -- & -- \\
$3\per$ & $10^{-9}$ & 25.13 & 33 & 47 & 7 & 7 & 7 & -- & -- & -- & -- \\
$3\per$ & $10^{-12}$ & 33.51 & 45 & 63 & 11 & 11 & 11 & -- & -- & -- & -- \\
\hline
$2\per$ & $10^{-3}$ & 10.47 & 12 & 19 & 3 & 3 & -- & -- & -- & 9 & 2.00 \\
$2\per$ & $10^{-6}$ & 17.80 & 23 & 33 & 5 & 5 & -- & -- & -- & 25 & 2.00 \\
$2\per$ & $10^{-9}$ & 25.13 & 33 & 47 & 7 & 7 & -- & -- & -- & 35 & 2.00 \\
$2\per$ & $10^{-12}$ & 33.51 & 45 & 63 & 11 & 11 & -- & -- & -- & 45 & 2.00 \\
\hline
$1\per$ & $10^{-3}$ & 10.47 & 12 & 19 & 3 & -- & -- & -- & 9 & 9 & 2.83 \\
$1\per$ & $10^{-6}$ & 17.80 & 23 & 33 & 5 & -- & -- & -- & 35 & 35 & 2.83 \\
$1\per$ & $10^{-9}$ & 25.13 & 33 & 47 & 7 & -- & -- & -- & 47 & 47 & 2.83 \\
$1\per$ & $10^{-12}$ & 33.51 & 45 & 63 & 11 & -- & -- & -- & 63 & 63 & 2.83 \\
\hline
\end{tabular}
\vspace{1em}

\begin{tabular}{ll|llllllllll}
\multicolumn{12}{c}{\textbf{Rotlet}} \\
\hline
$D\per$ & $\varepsilon$ & $c$ & $p$ & $\NF$ & $\Nper_1$ & $\Nper_2$ & $\Nper_3$ & $\Nfree_1$ & $\Nfree_2$ & $\Nfree_3$ & $\Rtrunc$ \\
\hline
$3\per$ & $10^{-3}$ & 6.28 & 8 & 11 & 3 & 3 & 3 & -- & -- & -- & -- \\
$3\per$ & $10^{-6}$ & 13.61 & 18 & 25 & 5 & 5 & 5 & -- & -- & -- & -- \\
$3\per$ & $10^{-9}$ & 20.94 & 29 & 39 & 7 & 7 & 7 & -- & -- & -- & -- \\
$3\per$ & $10^{-12}$ & 28.27 & 39 & 53 & 9 & 9 & 9 & -- & -- & -- & -- \\
\hline
$2\per$ & $10^{-3}$ & 6.28 & 8 & 11 & 3 & 3 & -- & -- & -- & 7 & 2.00 \\
$2\per$ & $10^{-6}$ & 13.61 & 18 & 25 & 5 & 5 & -- & -- & -- & 19 & 2.00 \\
$2\per$ & $10^{-9}$ & 20.94 & 29 & 39 & 7 & 7 & -- & -- & -- & 29 & 2.00 \\
$2\per$ & $10^{-12}$ & 28.27 & 39 & 53 & 9 & 9 & -- & -- & -- & 37 & 2.00 \\
\hline
$1\per$ & $10^{-3}$ & 6.28 & 8 & 11 & 3 & -- & -- & -- & 7 & 7 & 2.83 \\
$1\per$ & $10^{-6}$ & 13.61 & 18 & 25 & 5 & -- & -- & -- & 27 & 27 & 2.83 \\
$1\per$ & $10^{-9}$ & 20.94 & 29 & 39 & 7 & -- & -- & -- & 39 & 39 & 2.83 \\
$1\per$ & $10^{-12}$ & 28.27 & 39 & 53 & 9 & -- & -- & -- & 53 & 53 & 2.83 \\
\hline
\end{tabular}
\end{table}

\begin{table}[tb]
\centering
\caption{Achieved relative $\ell_2$ errors for the Stokes kernels for different periodicities, domains, and requested tolerances $\eeps$. The total number of sources and targets is $N=1000$, with half of the points uniformly distributed in the domain, and the other half clustered near the center of one top-level cube.}
\label{tab:stokes-tolerance}
\begin{tabular}{ll|llll}
\multicolumn{6}{c}{\textbf{Stokeslet}} \\
\hline
$D\per$ & $[L_1,L_2,L_3]$ & $\varepsilon=10^{-3}$ & $\varepsilon=10^{-6}$ & $\varepsilon=10^{-9}$ & $\varepsilon=10^{-12}$ \\
\hline
3$\per$ & $[2.0,1.0,1.0]$ & $2.1{\times}10^{-4}$ & $3.3{\times}10^{-7}$ & $3.4{\times}10^{-10}$ & $3.1{\times}10^{-13}$ \\
3$\per$ & $[1.0,1.3,1.7]$ & $3.3{\times}10^{-4}$ & $4.2{\times}10^{-7}$ & $3.6{\times}10^{-10}$ & $3.9{\times}10^{-13}$ \\
\hline
2$\per$ & $[3.0,1.0,2.0]$ & $2.5{\times}10^{-4}$ & $2.1{\times}10^{-7}$ & $2.0{\times}10^{-10}$ & $2.6{\times}10^{-13}$ \\
2$\per$ & $[1.0,5.0,2.0]$ & $1.7{\times}10^{-4}$ & $2.4{\times}10^{-7}$ & $3.9{\times}10^{-10}$ & $4.4{\times}10^{-13}$ \\
\hline
1$\per$ & $[3.0,1.0,2.0]$ & $1.6{\times}10^{-4}$ & $1.7{\times}10^{-7}$ & $1.8{\times}10^{-10}$ & $3.0{\times}10^{-12}$ \\
1$\per$ & $[1.0,1.0,2.0]$ & $2.0{\times}10^{-4}$ & $2.8{\times}10^{-7}$ & $2.0{\times}10^{-10}$ & $2.2{\times}10^{-13}$ \\
\hline
\end{tabular}

\vspace{1em}

\begin{tabular}{ll|llll}
\multicolumn{6}{c}{\textbf{Stresslet}} \\
\hline
$D\per$ & $[L_1,L_2,L_3]$ & $\varepsilon=10^{-3}$ & $\varepsilon=10^{-6}$ & $\varepsilon=10^{-9}$ & $\varepsilon=10^{-12}$ \\
\hline
3$\per$ & $[2.0,1.0,1.0]$ & $4.2{\times}10^{-5}$ & $8.5{\times}10^{-8}$ & $2.4{\times}10^{-10}$ & $1.0{\times}10^{-13}$ \\
3$\per$ & $[1.0,1.3,1.7]$ & $3.7{\times}10^{-5}$ & $6.0{\times}10^{-8}$ & $9.4{\times}10^{-11}$ & $5.7{\times}10^{-14}$ \\
\hline
2$\per$ & $[3.0,1.0,2.0]$ & $6.9{\times}10^{-5}$ & $7.5{\times}10^{-8}$ & $1.5{\times}10^{-10}$ & $9.8{\times}10^{-14}$ \\
2$\per$ & $[1.0,5.0,2.0]$ & $4.8{\times}10^{-5}$ & $7.2{\times}10^{-8}$ & $1.9{\times}10^{-10}$ & $1.1{\times}10^{-13}$ \\
\hline
1$\per$ & $[3.0,1.0,2.0]$ & $1.4{\times}10^{-5}$ & $2.3{\times}10^{-8}$ & $4.3{\times}10^{-11}$ & $1.2{\times}10^{-13}$ \\
1$\per$ & $[1.0,1.0,2.0]$ & $5.0{\times}10^{-5}$ & $8.3{\times}10^{-8}$ & $1.7{\times}10^{-10}$ & $8.5{\times}10^{-14}$ \\
\hline
\end{tabular}
\vspace{1em}

\begin{tabular}{ll|llll}
\multicolumn{6}{c}{\textbf{Rotlet}} \\
\hline
$D\per$ & $[L_1,L_2,L_3]$ & $\varepsilon=10^{-3}$ & $\varepsilon=10^{-6}$ & $\varepsilon=10^{-9}$ & $\varepsilon=10^{-12}$ \\
\hline
3$\per$ & $[2.0,1.0,1.0]$ & $4.3{\times}10^{-5}$ & $5.3{\times}10^{-8}$ & $3.7{\times}10^{-11}$ & $7.2{\times}10^{-14}$ \\
3$\per$ & $[1.0,1.3,1.7]$ & $8.5{\times}10^{-5}$ & $1.6{\times}10^{-7}$ & $9.4{\times}10^{-11}$ & $1.5{\times}10^{-13}$ \\
\hline
2$\per$ & $[3.0,1.0,2.0]$ & $8.5{\times}10^{-6}$ & $1.4{\times}10^{-8}$ & $8.9{\times}10^{-12}$ & $5.4{\times}10^{-14}$ \\
2$\per$ & $[1.0,5.0,2.0]$ & $2.6{\times}10^{-5}$ & $4.1{\times}10^{-8}$ & $3.6{\times}10^{-11}$ & $1.4{\times}10^{-13}$ \\
\hline
1$\per$ & $[3.0,1.0,2.0]$ & $3.5{\times}10^{-5}$ & $7.1{\times}10^{-8}$ & $4.6{\times}10^{-11}$ & $1.3{\times}10^{-13}$ \\
1$\per$ & $[1.0,1.0,2.0]$ & $4.2{\times}10^{-5}$ & $6.4{\times}10^{-8}$ & $5.5{\times}10^{-11}$ & $6.4{\times}10^{-14}$ \\
\hline
\end{tabular}

\end{table}

\clearpage
\bibliographystyle{siamplain}
\bibliography{dmk_per}

\end{document}

%% file: defs.tex
\def\reals{\R}
\renewcommand{\v}[1]{\bm{#1}}

\newcommand{\vx}{{\v x}}
\newcommand{\x}{{\v x}}
\renewcommand{\k}{{\v k}}
\newcommand{\vk}{{\v k}}

\newcommand{\vr}{{\v r}} % don't renew \r!
\newcommand{\vkappa}{{\v \kappa}}
\newcommand{\eeps}{\varepsilon}

\DeclareMathOperator\supp{supp}

\newcommand{\prol}{{\psi_0^c}}
\newcommand{\hatprol}{{\widehat{\psi}_0^c}}
\newcommand{\erf}{\operatorname{erf}}

\newcommand{\erfc}{\operatorname{erfc}}

\newcommand{\wfunc}{\phi} % General window function
\newcommand{\wfunchat}{\widehat{\phi}} % General window function
\newcommand{\decpar}{\nu} % Decomposition parameter (length)

\newcommand{\unitbox}{\mathcal{B}}

\newcommand*{\harmonic}{H}
\newcommand*{\biharmonic}{B}
\newcommand*{\stokeslet}{S}
\newcommand*{\rotlet}{\Omega}
\newcommand*{\rotlett}{\boldsymbol{\Omega}}
\newcommand*{\stresslet}{T}

\newcommand{\Bmollhat}{\widehat{B}_M}

\newcommand{\Hmollhat}{\widehat{H}_M}

\newcommand*{\Dop}{\mathcal{D}}
\newcommand*{\Dopt}{\bmatu{D}}
\newcommand*{\wh}[1]{\widehat{#1}}

\newcommand{\genericGhat}{\wh{\mathcal{G}}}
\newcommand{\Kres}{{K_R}}
\newcommand{\Kmoll}{{K_M}}
\newcommand{\Kdiff}{{K_D}}
\newcommand{\Khat}{\widehat{K}}
\newcommand{\Kmollhat}{\Khat_M}

\def\kmhat_#1{\widehat{K}_{M_{#1}}}
\def\kdhat_#1{\widehat{K}_{D_{#1}}}
\def\km_#1{K_{M_{#1}}}
\def\kr_#1{K_{R_{#1}}}
\def\kd_#1{K_{D_{#1}}}

\def\bmohat_#1{\widehat{B}_{M_{#1}}}
\def\brhat_#1{\widehat{B}_{R_{#1}}}
\def\bdhat_#1{\widehat{B}_{D_{#1}}}
\def\bmo_#1{B_{M_{#1}}}
\def\br_#1{B_{R_{#1}}}
\def\bd_#1{B_{D_{#1}}}

\newcommand{\vm}{\mathbf{m}}
\newcommand{\setm}[1]{\mathcal{I}_{#1}}

\newcommand{\R}{\mathbb{R}}

\newcommand{\tind}{\beta}
\newcommand{\sind}{\alpha}

\newcommand{\RSdiag}{S_{\text{diag}}}
\newcommand{\RSoffd}{S_{\text{offd}}}
\newcommand{\RTdiag}{T_{\text{diag}}}
\newcommand{\RToffd}{T_{\text{offd}}}
\newcommand{\ROoffd}{\Omega_{\text{offd}}}

\newcommand{\be}{\begin{equation}}
\newcommand{\ee}{\end{equation}}
\newcommand{\ba}{\begin{aligned}} 
\newcommand{\ea}{\end{aligned}}

\newcommand{\ulocal}{u_{\text{local}}}
\newcommand{\ufar}{u_{\text{far}}}
\newcommand{\uself}{u_{\text{self}}}

\newcommand{\udiffl}[1]{u_{\text{diff},#1}}
\newcommand{\ucompact}{u_{\text{compact}}}
\newcommand*{\per}{\mathcal{P}}
\newcommand{\perindsetD}{P_{D\per}}
\newcommand{\perindset}[1]{P_{#1\per}}

\newcommand*{\wavenumset}{\mathcal{K}}
\newcommand*{\wavenumsettrunc}{\mathcal{K}_{\mathrm{trunc}}}
\newcommand{\uP}[1]{u^{#1\per}}
\newcommand{\ulocalP}[1]{\ulocal^{#1\per}}
\newcommand{\ufarP}[1]{\ufar^{#1\per}}
\newcommand{\uselfP}[1]{\uself^{#1\per}}

\newcommand{\udifflP}[2]{\udiffl{#1}^{#2\per}}
\newcommand{\ucompactP}[1]{\ucompact^{#1\per}}
\newcommand{\noLevels}{\mathcal{L}}

\newcommand{\uFourierP}{U}
\newcommand{\uHatFourierP}{\widehat{U}}
\newcommand{\boxset}{\bar{\Omega}}
\newcommand{\Omegapad}{\Omega_{\mathrm{pad}}}
\newcommand{\Omegaext}{\Omega_{\mathrm{ext}}}
\newcommand{\notopboxes}{M}
\newcommand{\notopboxesext}{\notopboxes^{\mathrm{ext}}}
\newcommand{\boxind}{m}
\newcommand{\boxcenter}{\cb_\boxind}

\newcommand{\cb}{\mathbf{c}}

\newcommand{\Kmax}{K_{\mathrm{max}}}
\newcommand{\kint}{\bar{k}}
\newcommand{\kintx}{\kint_1}
\newcommand{\kinty}{\kint_2}
\newcommand{\kintz}{\kint_3}

\newcommand{\kmaxx}{\kintx^{\mathrm{max}}}
\newcommand{\kmaxy}{\kinty^{\mathrm{max}}}
\newcommand{\kmaxz}{\kintz^{\mathrm{max}}}
\newcommand{\kmaxgen}{\kint^{\mathrm{max}}}
\newcommand{\kmaxi}{\kint_i^{\mathrm{max}}}
\newcommand{\kmaxj}{\kint_j^{\mathrm{max}}}

\newcommand{\freesetk}{\mathcal{F}}
\newcommand{\freesetkP}[1]{\freesetk^{{#1}\per}}
\newcommand{\hk}{\Delta \kappa}

\newcommand{\screen}{\gamma}
\newcommand{\screenhat}{\widehat{\gamma}}

\newcommand{\screenHhat}{\widehat{\gamma}^H}
\newcommand{\Rtrunc}{\mathcal{R}}
\newcommand{\rect}{\operatorname{rect}}

\newcommand{\screenBhat}{\widehat{\gamma}^B}
\newcommand{\alphamin}{\alpha_{\mathrm{min}}}

\newcommand{\sgn}{\operatorname{sgn}}
\newcommand{\genericGopt}{\bmat{\mathcal{G}}}
\newcommand{\genericGhatopt}{\wh{\bmat{\mathcal{G}}}}
\newcommand{\Nper}{\bar{K}^{\mathrm{per}}}
\newcommand{\Nfree}{\bar{K}^{\mathrm{free}}}
\newcommand{\NFP}[1]{\bar{K}_{\mathrm{tot}}^{#1\per}}
\newcommand{\NF}{\bar{K}_{F}}
\newcommand{\Ncube}{N_{\mathrm{cube}}}

\newcommand{\kbar}{{\bar{k}}}

\usepackage{fix-cm}
\DeclareMathAlphabet{\mathsfit}{T1}{\sfdefault}{\mddefault}{\sldefault}
\SetMathAlphabet{\mathsfit}{bold}{T1}{\sfdefault}{\bfdefault}{\sldefault}
\newcommand*{\bmat}[1]{\boldsymbol{\mathsfit{#1}}}
\newcommand*{\bmatu}[1]{\boldsymbol{\mathsf{#1}}}

%% file: interaction_levels.tikz
\begin{tikzpicture}

\newcommand\bluecolor{blue!40};

\renewcommand\xx{1.85};
\renewcommand\yy{1.9};

\newcommand\LL{3};
\newcommand\DD{1};
\newcommand\dotsize{0.04}

\coordinate (P1) at (\xx,\yy);
\coordinate (P2) at (\xx+\LL+\DD,\yy);
\coordinate (P3) at (\xx+2*\LL+2*\DD,\yy);

\draw[fill=\bluecolor] (0,0) rectangle (3,3);
\draw[fill=\bluecolor] (5,1) rectangle (6.5,2.5);
\draw[fill=\bluecolor] (5+\LL+\DD+0.5,1+0.5) rectangle (5+\LL+\DD+1.25,1+1.25);

\draw[step=1, very thick] (0,0) grid (3,3);

\draw[step=1, very thick] (3.99,0) grid (7,3);
\draw[step=1/2, thick] (3.99,0) grid (7,3);

\draw[step=1/4, semithick] (7.99,0) grid (11,3);
\draw[step=1/2, thick] (7.99,0) grid (11,3);
\draw[step=1, very thick] (7.99,0) grid (11,3);

\draw[fill] (P1) circle (\dotsize);
\draw[fill] (P2) circle (\dotsize);
\draw[fill] (P3) circle (\dotsize);

\draw[red,semithick] (P1) circle (1);
\draw[red,semithick] (P2) circle (1/2);
\draw[red,semithick] (P3) circle (1/4);

\draw[<->,red,thin] (P1)--(\xx+1,\yy) node[midway,below,scale=0.8]{$r_\ell$};

\node[below] (label1) at (\LL/2,0){Level $\ell$};
\node[below] (label1) at (\LL/2+\LL+\DD,0){Level $\ell+1$};
\node[below] (label1) at (\LL/2+2*\LL+2*\DD,0){Level $\ell+2$};

\end{tikzpicture}

%% file: periodic_L1.tikz
\begin{tikzpicture}[scale=2.4]
\newcommand\LI{1};
\newcommand\LII{1.85};

\draw[|<->|] (0,-0.1)--(\LII,-0.1) node[midway,below,scale=1]{$L_2$};

\draw[|<->|] (3.2*\LI,0)--(3.2*\LI,\LI) node[midway,right,scale=1]{$L_1$};

\clip (-1.1\LI,-0.1) rectangle + (4.2*\LI,\LI+0.2);

\foreach \i in {0}
{
	\foreach \j in {-1,1}
	{
		\draw[fill=gray!10,very thick, densely dotted] (\j*\LII,\i*\LI)  rectangle (\j*\LII+\LII,\i*\LI+\LI);
	}
}

\draw[fill=gray!50,line width=1.8pt] (0,0) rectangle (\LII,\LI);

% Clustered source points in the main box, with periodic copies
\pgfmathsetseed{2}

% cluster center x / cluster center y / number of points
\foreach \cx/\cy/\n in {
    0.24/0.28/12,
    1.4/0.30/13,
    0.72/0.74/13,
    1.48/0.68/13
}{
    \foreach \k in {1,...,\n}{
        \pgfmathrandominteger{\dx}{-200}{250}
        \pgfmathrandominteger{\dy}{-220}{200}
        \pgfmathsetmacro{\px}{\cx + \dx/1000}
        \pgfmathsetmacro{\py}{\cy + \dy/1000}

        % original source point: black circle
        \filldraw[very thin] (\px,\py) circle (0.02);

        % periodic copies: blue circles
        \foreach \i in {0}
        {
            \foreach \j in {-1,1}
            {
                \filldraw[very thin,blue]
                    ({\px+\LII*\j},{\py+\LI*\i}) circle (0.02);
            }
        }
    }
}

\draw[step=\LI,red] (-\LI,0) grid (3*\LI,\LI);
\draw[red, thick, opacity=0.8] (-\LI,0) rectangle (3*\LI,\LI);
\end{tikzpicture}